\documentclass{article}
\usepackage{arxiv}

\usepackage[utf8]{inputenc} % allow utf-8 input
\usepackage[T1]{fontenc}    % use 8-bit T1 fonts
\usepackage{hyperref}       % hyperlinks
\usepackage{url}            % simple URL typesetting
\usepackage{booktabs}       % professional-quality tables
\usepackage{multirow,amsthm,amsmath,amssymb,mathtools,bbm,subcaption}
\usepackage{amsfonts}       % blackboard math symbols
\usepackage{nicefrac}       % compact symbols for 1/2, etc.
\usepackage{microtype}      % microtypography
\usepackage{cleveref}       % smart cross-referencing
\usepackage{lipsum}         % Can be removed after putting your text content
\usepackage{graphicx}
\usepackage[numbers]{natbib}
\usepackage{doi}

\newtheorem{theorem}{Theorem}[part]

\newtheorem{proposition}{Proposition}[part]

\newtheorem{assumption}{Assumption}[part]
\newtheorem{lemma}{Lemma}[part]
\newtheorem{corollary}{Corollary}[part]
\newtheorem{remark}{Remark}[part]

\makeatletter \@addtoreset{equation}{section}

\@addtoreset{definition}{section}

\@addtoreset{theorem}{section}

\@addtoreset{proposition}{section}

\@addtoreset{property}{section}

\@addtoreset{assumption}{section}

\@addtoreset{corollary}{section}

\@addtoreset{lemma}{section}

\@addtoreset{remark}{section}

\@addtoreset{example}{section}

\DeclareFontFamily{U}{mathx}{\hyphenchar\font45}
\DeclareFontShape{U}{mathx}{m}{n}{
      <5> <6> <7> <8> <9> <10>
      <10.95> <12> <14.4> <17.28> <20.74> <24.88>
      mathx10
      }{}
\DeclareSymbolFont{mathx}{U}{mathx}{m}{n}
\DeclareFontSubstitution{U}{mathx}{m}{n}
\DeclareMathAccent{\widecheck}{0}{mathx}{"71}
\DeclareMathAccent{\wideparen}{0}{mathx}{"75}

\title{Multi-stage Euler-Maruyama methods for backward stochastic differential equations driven by continuous-time Markov chains}
\date{November 15, 2023}
\newif\ifuniqueAffiliation
\uniqueAffiliationtrue

\ifuniqueAffiliation % Standard variant of author block
\author{ \href{https://orcid.org/0009-0009-0554-4983}{\includegraphics[scale=0.06]{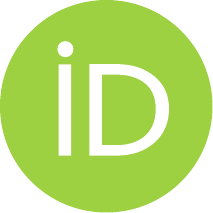}\hspace{1mm}Akihiro Kaneko}\\
	Graduate School of Engineering Science\\
	Osaka University\\
	1-3, Machikaneyama, Toyonaka, 5600531, Osaka, Japan\\
	\texttt{akihirokaneko@ymail.ne.jp}
}
\else
\usepackage{authblk}

\newbox{\orcid}\sbox{\orcid}{\includegraphics[scale=0.06]{orcid.pdf}} 
\author[1]{%
	\href{https://orcid.org/0000-0000-0000-0000}{\usebox{\orcid}\hspace{1mm}Akihiro Kaneko\thanks{\texttt{akihirokaneko@ymail.ne.jp}}}%
}
\affil[1]{Graduate School of Engineering Science, Osaka University, 1-3, Machikaneyama, Toyonaka, 5600531, Osaka, Japan}
\fi

\renewcommand{\shorttitle}{Multi-stage Euler-Maruyama methods for BSDEs driven by CTMCs}

\hypersetup{
pdftitle={Multi-stage Euler-Maruyama methods for backward stochastic differential equations driven by continuous-time Markov chains},
pdfsubject={math.PR},
pdfauthor={Akihiro Kaneko},
pdfkeywords={Backward stochastic differential equation, 
Continuous-time Markov chain, 
Multistage Euler-Maruyama method,
Method of lines, 
Exponential integrators, 
Sparse grids},
}

\begin{document}
\maketitle

\begin{abstract}
Numerical methods for computing the solutions of Markov backward stochastic differential equations (BSDEs) driven by continuous-time Markov chains (CTMCs) are explored. The main contributions of this paper are as follows: (1) we observe that Euler-Maruyama temporal discretization methods for solving Markov BSDEs driven by CTMCs are equivalent to exponential integrators for solving the associated systems of ordinary differential equations (ODEs); (2) we introduce multi-stage Euler-Maruyama methods for effectively solving ``stiff'' Markov BSDEs driven by CTMCs; these BSDEs typically arise from the spatial discretization of Markov BSDEs driven by Brownian motion; (3) we propose a multilevel spatial discretization method on sparse grids that efficiently approximates high-dimensional Markov BSDEs driven by Brownian motion with a combination of multiple Markov BSDEs driven by CTMCs on grids with different resolutions. We also illustrate the effectiveness of the presented methods with a number of numerical experiments in which we treat nonlinear BSDEs arising from option pricing problems in finance.
\end{abstract}


\section{Introduction}
\label{sec:Introduction}
\subsection{Background and motivation}
\label{sec:BackgroundandMotivation}
Backward stochastic differential equations (BSDEs) have been intensively studied from both theoretical and application points of view: Bismut studies BSDEs related to stochastic control problems in \cite{Bismut1976,Bismut1978}, and Pardoux and Peng \cite{Pardoux1990} introduces general {\it nonlinear} BSDEs driven by Brownian motions as noise processes, typically written as
\begin{equation}
\label{eq:GeneralBMBSDE}
\mathcal Y_t = \xi + \int_t^Tf(s,\mathcal Y_s,\mathcal Z_s)ds - \int_t^T\mathcal Z_sdW_s,\quad t\in[0,T].
\end{equation}
Here, $(W_t)_{t\in [0,T]}$ is a Brownian motion, $\int_t^T\mathcal Z_sdW_s$ is It\^o's stochastic integral, and the data $(\xi,f)$, a pair of a {\it terminal condition} $\xi$ (a random variable) and a {\it driver} $f$ (a function), is given in advance. A solution of \eqref{eq:GeneralBMBSDE} means a pair $(\mathcal Y_t,\mathcal Z_t)_{t\in [0,T]}$ of adapted processes that satisfies \eqref{eq:GeneralBMBSDE}. Since the late 1990s, the study of BSDEs has been highly connected to mathematical finance. It has provided a multitude of research topics leading to the development of BSDE theory and its applications (see El Karoui et al. \cite{Karoui1997}, for example). For example, hedging derivative securities under nonlinear wealth process dynamics (e.g., different interest rates for borrowing and lending), dynamic risk measures and recursive utilities are successful applications of BSDEs. Markov BSDE, written in the form
\begin{align}
\mathcal Y_t &=\displaystyle g(\mathcal X_T)+\int_t^Tf(s,\mathcal X_s,\mathcal Y_s,\mathcal Z_s)ds-\int_t^T\mathcal Z_sdW_s,\label{eq:MarkovBSDE}\\
\mathcal X_t &=\displaystyle x_0 +\int_0^t\mu(s,\mathcal X_s)ds +\int_0^t\sigma(s,\mathcal X_s)dW_s,\label{eq:FSDE}
\end{align}
is useful in applications. Here, $\mu:[0,T]\times\mathbb R^d\to\mathbb R^d$, $\sigma:[0,T]\times\mathbb R^d\to\mathbb R^{d\times d}$, $f:[0,T]\times\mathbb R^d\times\mathbb R\times\mathbb R^d\to\mathbb R$ and $g:\mathbb R^d\to\mathbb R$ are deterministic functions, $x_0\in{\mathbb R}^d$, and process $(\mathcal X_t)_{t\in[0,T]}$ is interpreted as the state variable of a system. An important relation exists between the solution of a Markov BSDE and the solution of a second-order semilinear parabolic partial differential equation (PDE), referred to as the nonlinear Feynman-Kac formula. For details, see Theorem \ref{thm:NonlinearFC} in Section \ref{sec:BrownianBSDE}.

In this paper, we are interested in a different class of BSDEs, that is, (Markov) BSDEs driven by continuous-time Markov chains (CTMCs), written as 
\[
Y_t = X_T^\ast G 
+ \int_{]t,T]}h(X_{s-},s,Y_{s-},Z_s)ds - \int_{]t,T]}dM_s^\ast Z_s,
\]
and study numerical schemes for solving them. Here, $(X_t)_{t\in[0,T]}$ is a CTMC having a finite state space ${\mathcal I}$, $N\coloneqq\# {\mathcal I}$, $G\in\mathbb R^N$, $h:{\mathcal I}\times[0,T]\times\mathbb R\times\mathbb R^N\to\mathbb R$, and $(M_t)_{t\in[0,T]}$ is the associate martingale with $X$ (ses Section \ref{sec:CTMCBSDE} for the details). 
As for the studies of such BSDEs driven by CTMCs, we refer to \cite{Cohen2008,Cohen2010,Cohen2012,Cohen2013,Cohen2014,Djehiche2016}. In particular, a ``nonlinear Feynman-Kac type'' formula for Markov BSDEs driven by CTMCs has been derived in \cite{Cohen2012,Djehiche2016}; the solutions of the BSDEs can be represented using solutions of the associated systems of ordinary differential equations (ODEs). 

The main contributions of this paper are summerized as follows.
\begin{enumerate}
\item We observe that Euler-Maruyama temporal discritization methods 
for solving a Markov BSDE driven by a CTMC is equivalent to exponential integrators \cite{Hochbruck2010} for solving the associated system of ODEs. (See Section \ref{sec:Euler-Maruyama}.)
\item We introduce multi-stage Euler-Maruyama methods for efficiently solving ``stiff'' BSDEs driven by CTMCs. Together with a spatial discretization, they can be applied to solve BSDEs driven by Brownian motion. (See Section \ref{sec:Euler-Maruyama} and Section \ref{sec:MOL}.)
\item We propose a multilevel spatial discretization on a sparse grid for solving high-dimensional Markov BSDEs driven by Brownian motion, in which we construct a sequence of BSDEs driven by CTMCs on grids with different resolutions and suitably combine their solutions. With the help of the idea of sparse grid methods \cite{Bungartz2004,Griebel1992}, it reduces the computational cost and efficiently mitigates the curse of dimensionality. (See Section \ref{sec:SGcomb}.) The efficiency has been confirmed in numerical experiments in Section \ref{sec:numericalexperiments}.
\end{enumerate}

\subsection{Organization of the paper}
\label{sec:Organization}
The paper is organized as follows: At the end of this section, we introduce notations frequently used in the paper. In Section \ref{sec:preliminray}, we present preliminary results on BSDEs driven by CTMCs and ones driven by Brownian motion that are required for the subsequent arguments. In Section \ref{sec:Euler-Maruyama}, we construct multi-stage Euler-Maruyama methods for BSDEs driven by CTMCs and observe that they are equivalent to exponential integrators, solvers that calculate stiff systems of ODEs successfully. An application of the multi-stage Euler-Maruyama methods to BSDEs driven by Brownian motion is presented in Section \ref{sec:Application}; we present a concrete discretization and the resulting BSDE driven by a CTMC is presented in Section \ref{sec:MOL}, and a multilevel spatial discretization on a sparse grid for efficiently solving high-dimensional BSDEs driven by Brownian motion in Section \ref{sec:SGcomb}. Section \ref{sec:numericalexperiments} provides experiments highlighting the effectiveness of our schemes. Specifically, we treat option pricing problems under nonlinear wealth dynamics with several asset price process models, such as the Black-Scholes model and stochastic volatility models including the SABR model.
\subsection*{Notations}
\label{sec:notations}
For $N\in\mathbb N$, $e_i$ means the $i$-th unit vector in the Euclidean
space $\mathbb R^N$ whose $i$-th element is $1$. 
The notations $|\cdot|$ and $\|\cdot\|$ represent the absolute value and the Euclidean norm, respectively. For any matrix $Q$, $Q^\ast$ denotes the matrix transposition, $Q^+$ denotes the Moore-Penrose inverse, and $\operatorname{Tr}(Q)$ denotes the trace of $Q$. For any vector $v$, $\operatorname{diag}(v)$ is a diagonal matrix whose $i$-th diagonal element is $e_i^\ast v$. For any two vectors $v,w\in\mathbb R^N$, denote 
\[
	v\le w \iff e_i^\ast v\le e_i^\ast w,\quad i=1,\dots,N.
\]
We set as follows.
\begin{itemize}
\item $C([0,T]\times\mathbb R)$ and $C(\mathbb R)$ are the sets of $\mathbb R$-valued continuous functions defined on $[0,T]\times\mathbb R$ and $\mathbb R$, respectively. 
\item $C_b([0,T]\times\mathbb R)$ and $C_b(\mathbb R)$ are the sets of $\mathbb R$-valued bounded continuous functions defined on $[0,T]\times\mathbb R$ and $\mathbb R$, respectively.
\item $C_b^2(\mathbb R)$ is is the set of $\mathbb R$-valued, twice continuously differentiable functions $u$ such that $\partial_xu$, $\partial_{xx}u$ as well as $u$ are in $C_b(\mathbb R)$.
\item $C_b^{1,2}([0,T]\times\mathbb R)$ is the set of $\mathbb R$-valued functions $u$, which is once continuously differentiable in its first argument, twice continuously differentiable in its second, and $\partial_tu$, $\partial_xu$, $\partial_{xx}u$ as well as $u$ are in $C_b([0,T]\times\mathbb R)$.
\end{itemize}
Here, $\displaystyle\partial_tu(t,x) \coloneqq \frac{\partial u}{\partial t}(t,x)$, $\displaystyle\partial_xu(t,x) \coloneqq \frac{\partial u}{\partial x}(t,x)$ and $\displaystyle\partial_{xx}u(t,x) \coloneqq \frac{\partial^2 u}{\partial x^2}(t,x)$.
For a vector-valued c\`adl\`ag stochastic process $X_t$, $X_{t-}$ denotes the left limit and $\Delta X_t\coloneqq X_t-X_{t-}$. Additinally, if $X_t$ is a semimartingale, $\langle X,X\rangle$ denotes the predictable quadratic variation matrix.
Throughout the paper, we will work on a probability space $(\Omega,\mathcal F,\mathbb P)$ and a finite time horizon $T>0$. For $k,m\in\mathbb N$, a filtration $\mathbb F$ with the usual conditions, and a square-integrable c\`{a}dl\`{a}g $\mathbb F$-martingale $M$, we define the following spaces of stochastic processes.
\begin{itemize}
	\item $L^2(\mathbb F,\mathbb R^k)$ is the set of c\`{a}dl\`{a}g $\mathbb F$-adapted processes $\mathcal X:[0,T]\times\Omega\to\mathbb R^k$ with $\displaystyle\mathbb E\left[\int_{[0,T]}\|\mathcal X_t\|^2dt\right]<\infty$.
	\item $\mathbb S^2(\mathbb F,\mathbb R^k)$ is the set of c\`{a}dl\`{a}g $\mathbb F$-adapted processes $\mathcal Y:[0,T]\times\Omega\to\mathbb R^k$ with $\displaystyle\mathbb E\left[\sup_{0\le t\le T}|\mathcal Y_t|^2\right]<\infty$.
	\item $L^2(\langle M\rangle,\mathbb F,\mathbb R^{k\times m})$ is the set of $\mathbb F$-predictable processes $\mathcal Z:[0,T]\times\Omega\to\mathbb R^{k\times m}$ with 
	\[
		\mathbb E\left[\left|\int_{[0,T]}\mathcal Z_tdM_t\right|^2\right]=\mathbb E\left[\int_{[0,T]}\operatorname{Tr}(\mathcal Z_td\langle M,M\rangle_t\mathcal Z_t^\ast)\right]<\infty.
	\]
\end{itemize}

\section{Setups and preliminary results}
\label{sec:preliminray}
\subsection{BSDEs driven by a CTMC}
\label{sec:CTMCBSDE}
Let $X=(X_t)_{t\in[0,T]}$ be a continuous-time, finite-state Markov chain with state space $\mathcal I=\{e_1,\dots,e_N\}$, for some $N\in\mathbb N$. Suppose that $X$ is defined on the filtered probability space $(\Omega,\mathcal F,\mathbb P, \mathbb G)$ where $\mathbb G\coloneqq(\mathcal G_t)_{t\in[0,T]}$ is the completion of the filtration generated by $X$. Note that $X$ is a c\`adl\`ag pure jump process in this case. 

$X$ is associated with a family of Q-matrices; recall that $N\times N$ matrices $Q_t, t\in[0,T]$ are called Q-matrices on $\mathcal I$ if $e_i^\ast Q_te_j \ge0$ for all $i\neq j$ and $\sum_je_i^\ast Q_te_j=0$ for all $i$. We suppose Q-matrices appeared in the present paper are uniformly bounded in time $t$. Note that some literatures refer to its transpose as Q-matrix such as \cite{Elliott1995,Cohen2008}. We sometimes call $Q_t$ as the transition rate matrix for $X$. For $p_t = (\mathbb P(X_t = e_1),\dots,\mathbb P(X_t=e_N))$ and $t\ge0$, we see that it satisfies the following Kolmogorov's forward equation
\begin{equation}
	\frac{dp_t}{dt} = p_tQ_t.
\end{equation}
Hence, the transition probability matrix of $X$, given by 
\[
	\Phi(t,s)\coloneqq
	\begin{pmatrix}
		\mathbb P(X_t=e_1|X_s=e_1) & \ldots & \mathbb P(X_t=e_N|X_s=e_1)\\
		\vdots & \ddots & \vdots\\
		\mathbb P(X_t=e_1|X_s=e_N) & \ldots & \mathbb P(X_t=e_N|X_s=e_N)
	\end{pmatrix}
\]
for $t\ge s$ satisfies the following equations
\begin{align}
	\frac{d\Phi(t,s)}{dt} &= \Phi(t,s)Q_t,\quad\Phi(s,s)=I,\label{KolmogorovFwd}\\
	\frac{d\Phi(t,s)}{ds} &= -Q_s\Phi(t,s),\quad\Phi(t,t)=I,\label{KolmogorovBwd}
\end{align}
for $t\ge s\ge0$ where $I$ is the $N\times N$ identity matrix. \eqref{KolmogorovFwd} and \eqref{KolmogorovBwd} are referred to as the forward and backward Kolmogorov equation, respectively. $X$ is time-(in)homogeneous if $Q_t$ does (not) depend on $t\in[0,T]$. Note also that the transition probability of the time-homogeneous chain $X$ with a transition rate matrix $Q$ is the matrix exponential $\Phi(t,s)=\exp((t-s)Q)$. 

From Appendix B in \cite{Elliott1995}, $X$ has the following semi-martingale representation
\begin{equation}
	\label{CTMC}
  X_t = x_0 + \int_{]0,t]}Q_s^\ast X_{s-}ds + M_t.
\end{equation}
Here, $x_0\in\{e_1,\dots,e_N\}$ and $M_t$ is an $\mathbb R^N$-valued $\mathbb G$-martingale. The predictable quadratic covariation matrix of $M$ is given by
\[
	\langle M,M\rangle_t = \int_{]0,t]}(\operatorname{diag}(Q_s^\ast X_{s-})-\operatorname{diag}(X_{s-})Q_s-Q_s^\ast\operatorname{diag}(X_{s-}))ds,
\]
which is also shown in Appendix B in \cite{Elliott1995}. Let $\psi_t\coloneqq \operatorname{diag}(Q_t^\ast X_{t-})-\operatorname{diag}(X_{t-})Q_t-Q_t^\ast\operatorname{diag}(X_{t-})$. Note that $\psi_t$ is a predictable process, valued in $N\times N$ real symmetric nonnnegative semi-definite matrices. For later use, define the seminorm for $z\in\mathbb R^N$ by 
\[
	\|z\|_v^2 \coloneqq z^\ast(\operatorname{diag}(Q_t^\ast v)-\operatorname{diag}(v)Q_t-Q_t^\ast\operatorname{diag}(v))z,
\]
where $v\in\{e_1,\dots,e_N\}$. Note that $\|\cdot\|_v$ depends on $t$ when $Q_t$ is time-dependent. The following It\^o's isometry is a key property of this seminorm for $v=X_{t-}$. That is, for any $\mathbb R^N$-valued predictable process $Z$, it holds
\begin{equation}
	\label{isometry}
	\mathbb E\left[\left(\int_{]s,t]}Z_u^\ast dM_u\right)^2\right] = \mathbb E\left[\int_{]s,t]}\|Z_u\|_{X_{u-}}^2du\right]\text{ for }t>s\ge0.
\end{equation}
The proof is given in \cite{Cohen2010}. We also define the equivalence relation $Z\sim_M Z'$ on $\mathbb R^N$-valued predictable processes as $\|Z_t-Z_t'\|_{X_{t-}}=0,dt\otimes d\mathbb P$-a.s. 

In \cite{Cohen2008}, Cohen and Elliott treat BSDEs driven by a CTMC in the form of
\begin{equation}
	\label{eq:GeneralCTMCBSDE}
	Y_t = \xi+\int_{]t,T]}h(t,Y_{s-},Z_s)ds-\int_{]t,T]}dM_s^\ast Z_s,
\end{equation}
where $\xi$ is an $\mathcal G_T$-measurable square-integrable random variable, $h:\Omega\times[0,T]\times\mathbb R\times\mathbb R^N\ni(\omega,t,y,z)\mapsto h(\omega,t,y,z)\in\mathbb R$ is $\mathbb G$-predictable in $(\omega,t)$ and Borel measurable in $(y,z)$. The following result on the existence and a uniqueness of the solution $(Y,Z)\in\mathbb S^2(\mathbb G,\mathbb R)\times L^2(\mathbb G,\langle M\rangle,\mathbb R^N)$ of \eqref{eq:GeneralCTMCBSDE} has been established.
\begin{theorem}[\cite{Cohen2008}]
	\label{thm:wellposednessofCTMCBSDE}
	Assume that,
	\[
		\mathbb E\left[\int_{]0,T]}h(t,0,0)^2dt\right] < \infty,
	\]
	and that for some constant $L>0$, 
	\begin{equation}
		\label{Lipschitz}
		|h(t,y,z)-h(t,y',z')|^2\le L(|y-y'|^2+\|z-z'\|_{X_{t-}}^2),\quad dt\otimes d\mathbb P\text{-a.s.} 
	\end{equation}
	for all $y,y'\in\mathbb R$ and $z,z'\in\mathbb R^N$. Then, it admits a unique solution $(Y,Z)\in\mathbb S^2(\mathbb G,\mathbb R)\times L^2(\mathbb G,\langle M\rangle,\mathbb R^N)$. We remark that it is unique up to indistinguishability for $Y$ and up to $\sim_M$ equivalence for $Z$.
\end{theorem}
% We consider the following spaces.
% \begin{align*}
% 	\mathbb S^2&\coloneqq\{Y:\mathbb R\text{-valued c\`{a}dl\`{a}g }(\mathcal G_t)_{t\in[0,T]}\text{-adapted with }\mathbb E[\sup_{0\le t\le T}|Y_t|^2]<\infty\},\\
% 	L^2(\langle M\rangle)&\coloneqq\{Z:\mathbb R^{1\times N}\text{-valued }(\mathcal G_t)_{t\in[0,T]}\text{-predictable}\\
% 	&\qquad\qquad\qquad\qquad\qquad\qquad\qquad\qquad\qquad\text{with }\mathbb E\int_{[0,T]}\|Z_t\|_{X_{t-}}^2dt<\infty\}.
% \end{align*}
\begin{remark}
	Since $\displaystyle\|z\|_{e_i}^2\le 3\max_{j,k=1,\dots,N}|e_j^\ast Q_te_k|\cdot\|z\|^2$ for $z\in\mathbb R^N$, \eqref{Lipschitz} leads to the usual Lipschitz continuity as $|h(t,y,z)-h(t,y',z')|^2\le L(|y-y'|^2+\|z-z'\|^2)$, $dt\otimes\mathbb P$-a.s. Note that, however, the converse does not hold; there does not exist $C>0$ such that $\|z\|\le C\|z\|_{e_i}$ for all $z$. Taking $z_1=(1,1,\dots,1)\in\mathbb R^N$, it is easy to see $\|z_1\|_{e_i}^2 = 0 < N = \|z_1\|^2$.
\end{remark}
Next, we consider the Markov BSDE driven by a CTMC of the form
\begin{equation}
	\label{eq:MarkovCTMCBSDE1}
	Y_t = X_T^\ast G + \int_{]t,T]}h(X_{s-},s,Y_{s-},Z_s)ds - \int_{]t,T]}dM_s^\ast Z_s,
\end{equation}
where $G\in\mathbb R^N$ and $h:\{e_1,\dots,e_N\}\times[0,T]\times\mathbb R\times\mathbb R^N\to\mathbb R$ is a Borel measurable function. Associated with \eqref{eq:MarkovCTMCBSDE1}, setting $t\in[0,T]$ as the starting time for the BSDE, we consider
\begin{equation}
	\label{eq:MarkovCTMCBSDE}
	\begin{cases}
		\displaystyle X_s^{t,e_i} = e_i + \int_{]t,s]}Q_u^\ast X_{u-}^{t,e_i}du + M_s-M_t,\quad s>t,\\
		\displaystyle X_s^{t,e_i} = e_i,\quad s\le t,\\
		\displaystyle Y_s^{t,e_i} = (X_T^{t,e_i})^\ast G + \int_{]s,T]}h(X_{u-}^{t,e_i},u,Y_{u-}^{t,e_i},Z_u^{t,e_i})du-\int_{]s,T]}dM_u^\ast Z_u^{t,e_i},\quad s\in[0,T].
	\end{cases}
\end{equation}
Then, we give the following nonlinear Feynman-Kac type result. Recall that, similar statements can be found in \cite{Cohen2012,Djehiche2016}. 
\begin{theorem}
	\label{thm:CTMCFeynmanKac}
	Assume that there exists a constant $L>0$ such that
		\begin{equation}
			\label{thmlip}
			|h(e_i,t,y,z)-h(e_i,t,y',z')|^2\le L^2(|y-y'|^2+\|z-z'\|_{e_i}^2),
		\end{equation}
	for any $y,y'\in\mathbb R$, $z,z'\in\mathbb R^N$, $t\in[0,T]$ and $i=1,\dots,N$, and $\int_0^Th(e_i,u,0,0)^2du<\infty$. Define $H:[0,T]\times\mathbb R^N\to\mathbb R^N$ such that
	\[
		e_i^\ast H(t,z) = h(e_i,t,e_i^\ast z,z)\quad\text{ for }t\in[0,T],z\in\mathbb R^N,i=1,\dots,N.
	\]
	\begin{enumerate}
		\item For a solution $U_t$ of the system of ODEs
		\begin{equation}
			\label{eq:ODE}
			\frac{dU_t}{dt} + Q_tU_t + H(t,U_t)=0,\quad U_T=G,
		\end{equation}
		$(Y_s^{t,e_i},Z_s^{t,e_i})=(U_sX_s^{t,e_i},U_s)\in\mathbb S^2(\mathbb G,\mathbb R)\times L^2(\mathbb G,\langle M\rangle,\mathbb R^N)$ uniquely solves \eqref{eq:MarkovCTMCBSDE}.
		\item Conversely, for a unique solution of \eqref{eq:MarkovCTMCBSDE} $(Y_s^{t,e_i},Z_s^{t,e_i})\in\mathbb S^2(\mathbb G,\mathbb R)\times L^2(\mathbb G,\langle M\rangle,\mathbb R^N)$, a continuous function $V_t = (Y_t^{t,e_1},\dots,Y_t^{t,e_N})^\ast$ satisfies $V\sim_MZ^{t,e_i}$ for $i=1,\dots,N$ and $t\in[0,T]$, and solves \eqref{eq:ODE}.
\end{enumerate}
\end{theorem}
\begin{proof}
	See Appendix \ref{sec:ProofOfThm2.2}.
\end{proof}
\begin{remark}
	If we assume the continuity of $t\mapsto h(e_i,t,y,z)$ for all $i,y,z$ additionally, a uniqueness of \eqref{eq:ODE} immediately holds from the well-known Picard-Lindel\"of theorem (e.g. Theorem 110C, P.23 in \cite{Butcher2003}).
\end{remark}
\begin{corollary}
	Under the square integrability $t\mapsto h(e_i,t,y,z)$ in $[0,T]$ and the uniform Lipschitz continuity \eqref{thmlip}, a unique solution $(Y,Z)$ of \eqref{eq:MarkovCTMCBSDE1} is also a unique solution of 
	\begin{equation}
		\label{eq:CTMCBSDE}
		Y_t = X_T^\ast G+\int_{]t,T]}X_{s-}^\ast H(s,Z_s)ds-\int_{]t,T]}dM_s^\ast Z_s.
	\end{equation}
	Moreover, the relation 
	\[
		Y_t = X_t^\ast U_t\quad\text{up to indistinguishability and}\quad Z\sim_M U
	\]
	holds, where $U$ is a solution of \eqref{eq:ODE}.
\end{corollary}
\subsection{BSDEs driven by a Brownian motion}\label{sec:BrownianBSDE}
Let $W=(W_t)_{t\ge0}$ be a $d$-dimensional standard Brownian motion. Let $\mathbb F=(\mathcal F_t)_{t\ge0}$ be the completion of the filtration generated by $W$. We consider the following Markov BSDE driven by Brownian motion.
\begin{equation}
	\begin{cases}
		\mathcal X_t &=x_0 +\displaystyle\int_0^t\mu(s,\mathcal X_s)ds +\int_0^t\sigma(s,\mathcal X_s)dW_s,\\
		\mathcal Y_t &= g(\mathcal X_T)+\displaystyle\int_t^Tf(s,\mathcal X_s,\mathcal Y_s,\mathcal Z_s)ds-\int_t^T\mathcal Z_s^\ast dW_s,
	\end{cases}
	\label{BMBSDE}
\end{equation}
where $\mu:[0,T]\times\mathbb R^d\to\mathbb R^d$, $\sigma:[0,T]\times\mathbb R^d\to\mathbb R^{d\times d}$, $f:[0,T]\times\mathbb R^d\times\mathbb R\times\mathbb R^d\to\mathbb R$, $g:\mathbb R^d\to\mathbb R$ are Borel measurable, and referred to as the drift coefficient, the diffusion coefficient, the driver and the terminal condition, respectively. Assuming that, there exists $L>0$ and $p\in\mathbb N$ such that
\begin{equation}
	\label{lip1}
	\begin{gathered}
		\|\mu(t,x)-\mu(t,x')\|+\|\sigma(t,x)-\sigma(t,x')\|\le L\|x-x'\|,\\
		|f(t,x,y,z)-f(t,x,y',z')| \le L(|y-y'|+\|z-z'\|),\\
		\|\mu(t,x)\| + \|\sigma(t,x)\| \le L(1+\|x\|^2)\\
		|f(t,x,y,z)|+|g(x)|\le L(1+\|x\|^p) 
	\end{gathered}
\end{equation}
for all $t\in[0,T]$, $x,x',z,z'\in\mathbb R^d$ and $y,y'\in\mathbb R$, \eqref{BMBSDE} has a unique solution $((\mathcal X_t)_{t\in[0,T]},(\mathcal Y_t)_{t\in[0,T]},(\mathcal Z_t)_{t\in[0,T]})\in \mathbb S^2(\mathbb F,\mathbb R^d)\times \mathbb S^2(\mathbb F,\mathbb R)\times L^2(\langle W\rangle,\mathbb F,\mathbb R^d)$. $\mathcal X$ is sometimes referred to as the state process, and it is solvable independently of $(\mathcal Y,\mathcal Z)$.

The nonlinear Feynman-Kac formula describes the relation between \eqref{BMBSDE} and
\begin{equation}
\label{PDE}
\begin{cases}
  \partial_tu(t,x)+\mathcal L_tu(t,x)+f(t,x,u(t,x),\sigma^\ast(x)\nabla_xu(t,x))=0,\quad (t,x)\in[0,T]\times\mathbb R^d,\\
  u(T,x)=g(x),\quad x\in\mathbb R^d.
\end{cases}
\end{equation}
Here, 
\begin{equation}
	\label{eq:infinitesimal}
	\mathcal L_tu(t,x) = \sum_{i=1}^d\mu^{(i)}(t,x)\frac{\partial u}{\partial x_i}(t,x)+\frac12\sum_{i,j=1}^d(\sigma\sigma^\ast)^{(i,j)}(t,x)\frac{\partial^2 u}{\partial x_i\partial x_j}(t,x)
\end{equation}
is the infinitesimal generator of the Markov process $\mathcal X$, $\nabla_xu(t,x) = \left(\frac{\partial u}{\partial x_1}(t,x),\dots,\frac{\partial u}{\partial x_d}(t,x)\right)^\ast\in \mathbb R^d$ is the gradient vector, $\mu^{(i)}(t,x)$ is the $i$-th component of $\mu(t,x)$, and $(\sigma\sigma^\ast)^{(i,j)}(t,x)$ is the $(i,j)$-th component of $\sigma(t,x)\sigma^\ast(t,x)$. The precise statement is as follows.
\begin{theorem}[The nonlinear Feynman-Kac formula (e.g. pp.487-489 in \cite{Cohen2015})]
	\label{thm:NonlinearFC}
	Suppose that $\mu,\sigma, f$ and $g$ are defined as above. For $(t,x)\in[0,T]\times\mathbb R^d$, let $(\mathcal X^{t,x},\mathcal Y^{t,x},\mathcal Z^{t,x})$ be a unique solution of the Markov BSDE
	\begin{equation}
		\label{MCBSDEtx}
		\begin{cases}
			\displaystyle\mathcal X_s^{t,x} =x + \int_t^s\mu(\tau,\mathcal X_\tau^{t,x})d\tau + \int_t^s\sigma(\tau,\mathcal X_\tau^{t,x})dW_\tau\quad\text{for}\quad s\ge t,\\
			\displaystyle\mathcal X_s^{t,x}=x\in\mathbb R^d \quad\text{for}\quad s\le t,\\
			\displaystyle\mathcal Y_s^{t,x} =g(\mathcal X_T^{t,x})+\int_s^Tf(\tau,\mathcal X_\tau^{t,x},\mathcal Y_\tau^{t,x},\mathcal Z_\tau^{t,x})d\tau -\int_s^T(\mathcal Z_\tau^{t,x})^\ast dW_\tau\quad\text{for}\quad s\in[0,T].
		\end{cases}
	\end{equation}
	Then, 
	\begin{enumerate}
		\item for every classical solution $u\in C^{1,2}([0,T]\times \mathbb R^d;\mathbb R)$ of \eqref{PDE}, such that, for some $K>0$,
		\begin{equation}
			\label{eq:esti}
			|u(t,x)|+\|\nabla_xu(t,x)\|\le K(1+\|x\|)\quad\text{for}\quad (t,x)\in[0,T]\times\mathbb R^d
		\end{equation}
		a unique solution of BSDE \eqref{MCBSDEtx} is represented as 
		\begin{equation}
			\label{FeynmanKac}
			\mathcal Y_s^{(t,x)} = u(s,\mathcal X_s^{(t,x)}),\quad \mathcal Z_s^{(t,x)} = \sigma^\ast(\mathcal X_s^{(t,x)})\nabla_xu(s,\mathcal X_s^{(t,x)})\quad\text{for}\quad s\ge t.
		\end{equation}
		(The inequality \eqref{eq:esti} is sufficient for showing $(\mathcal Y_s^{(t,x)},\mathcal Z_s^{(t,x)})$ is of the class $\mathbb S^2(\mathbb F,\mathbb R)\times L^2(\langle W\rangle,\mathbb F,\mathbb R^d)$.)
		\item Suppose further that $f$ and $g$ are Lipschitz continuous and uniformly continuous with respect to $x$, uniformly in $t$. Then, $u(t,x)\coloneqq \mathcal Y_t^{t,x}$ is a viscosity solution of \eqref{PDE}.
		\item Additionally, if for each $R>0$ there exists a continuous function $m_R:[0,\infty)\to[0,\infty)$ such that $m_R(0)=0$, and 
		\[
			|f(t,x,y,z)-f(t,x',y,z)|\le m_R(\|x-x'\|(1+\|z\|))
		\]
		holds for $y\in\mathbb R, x,x',z\in\mathbb R^d$ such that $\max\{\|x\|,\|x'\|,\|z\|\}<R$, then a uniqueness of $u$ also holds.
	\end{enumerate}
\end{theorem}

\section{Multi-stage Euler-Maruyama methods}
\label{sec:Euler-Maruyama}
In this section, we introduce sevaral (multi-stage) Euler-Maruyama methods for solving BSDEs driven by CTMCs. Hereafter, we always assume that $(X_t)_{t\in[0,T]}$ is time-homogeneous, namely, the transition rate matrix $Q_t$ equals some constant matrix $Q$ for all $t$. The transition expectation is then represented as the action of a matrix exponential on the present state $X_t$. This can be seen from
\begin{equation}
	\mathbb E[X_s|X_t] = \sum_{i=1}^Ne_i\mathbb P(X_s=e_i|X_t) = \begin{pmatrix}\mathbb P(X_s=e_1|X_t)\\\vdots\\\mathbb P(X_s=e_N|X_t)\end{pmatrix} = e^{(s-t)Q}X_t,
	\label{eq:markovcondexp}
\end{equation}
for all $t\le s\le T$. 

Euler-Maruyama methods are constructed in the following two steps:
\begin{enumerate}
	\item Slice the time interval $[0,T]$ into a temporal grid $\{0=t_0<t_1<\cdots<t_M=T\}$ and derive a stochastic difference equation on the grid.
	\item Take conditional expectations and suitably approximate the (Riemann) integral part that appeared.
\end{enumerate}
Let $(Y,Z)=(X^\ast Z,Z)$ be a unique solution of a BSDE driven by a CTMC \eqref{eq:CTMCBSDE}. Discretize $[0,T]$ on a uniform grid $t_m = m\Delta t$ for $m=0,1,\dots,M$, where $\Delta t = T/M$. We immediately see that $(Y_{t_m})_{m=0}^M$ satisfies the following stochastic difference equation
\begin{equation}
	\label{eq:DifferenceEq}
	Y_{t_m} = Y_{t_{m+1}} + \int_{]t_m,t_{m+1}]}X_{s-}^\ast H(s,Z_s)ds - \int_{]t_m,t_{m+1}]}dM_s^\ast Z_s,
\end{equation}
for $m=0,1,\dots,M-1$. Taking the conditional expectations $\mathbb E[\cdots|X_{t_m}]$ in both-hand sides of \eqref{eq:DifferenceEq}, we observe that
\begin{equation}
	Y_{t_m} = \mathbb E[X_{t_{m+1}}|X_{t_m}]^\ast Z_{t_{m+1}} + \int_{t_m}^{t_{m+1}}\mathbb E[X_s|X_{t_m}]^\ast H(s,Z_s)ds. \label{eq:condErep}
\end{equation}
We can consider the following Euler-Maruyama-type approximations from \eqref{eq:condErep}: 

\paragraph{(The Lawson-Euler method).} 
The simplest would be
\begin{equation}
	Y_{t_m}\approx \mathbb E[X_{t_{m+1}}|X_{t_m}]^\ast Z_{t_{m+1}} +\Delta t\mathbb E[X_{t_{m+1}}|X_{t_m}]^\ast H(t_{m+1},Z_{t_{m+1}}). \label{eq:LawsonEM}
\end{equation}
From \eqref{eq:markovcondexp}, \eqref{eq:LawsonEM} is reduced to 
\begin{equation}
	Y_{t_m}\approx X_{t_m}^\ast\left(e^{\Delta tQ}Z_{t_{m+1}}+\Delta te^{\Delta tQ}H(t_{m+1},Z_{t_{m+1}})\right).\label{eq:LawsonEM1}
\end{equation}
As a consequence, we have the following 1-stage Euler-Maruyama scheme.
\begin{equation}
	\begin{cases}
		\displaystyle Z_{t_M}^M \coloneqq G,\\
		\displaystyle Z_{t_m}^M \coloneqq e^{\Delta tQ}(Z_{t_{m+1}}^M+\Delta tH(t_{m+1},Z_{t_{m+1}}^M)),\quad m=0,1\dots,M-1,
	\end{cases}
	\label{eq:LawsonEuler}
\end{equation}
which is known as the Lawson-Euler method \cite{Lawson1967} for solving system of ODEs \eqref{eq:ODE}. Note that we take $Y_{t_m}^M \coloneqq X_{t_m}^\ast Z_{t_m}^M$ for $m=0,1\dots,M$. 
\paragraph{(The N\o rsett-Euler method).} 
We can consider another (1-stage) Euler-Maruyama approximation as
\begin{equation}
	Y_{t_m}\approx \mathbb E[X_{t_{m+1}}|X_{t_m}]^\ast Z_{t_{m+1}} +\left(\int_{t_m}^{t_{m+1}}\mathbb E[X_s|X_{t_m}]^\ast ds\right) H(t_{m+1},Z_{t_{m+1}}). \label{eq:NorsettEM}
\end{equation}
From \eqref{eq:markovcondexp} as before, \eqref{eq:NorsettEM} is reduced
\begin{equation}
	Y_{t_m}\approx X_{t_m}^\ast\left(e^{\Delta tQ}Z_{t_{m+1}}+\Delta t\left(\int_0^1e^{(1-\theta)\Delta tQ}d\theta\right)H(t_{m+1},Z_{t_{m+1}})\right), \label{eq:NorsettEM1}
\end{equation}
and we have
\begin{equation}
	\begin{cases}
		\displaystyle Z_{t_M}^M \coloneqq G,\\
		\displaystyle Z_{t_m}^M \coloneqq e^{\Delta tQ}Z_{t_{m+1}}^M+\Delta t\left(\int_0^1e^{(1-\theta)\Delta tQ}d\theta\right)H(t_{m+1},Z_{t_{m+1}}^M),\quad m=0,1\dots,M-1.
	\end{cases}
	\label{eq:NorsettEuler}
\end{equation}
It is known as the N\o rsett-Euler method \cite{Cox2002} for solving system of ODEs \eqref{eq:ODE}.

\paragraph{(The ETD2RK method).}
We can consider 2-stage Euler-Maruyama methods, as well as 1-stage ones. For example, a 2-stage N\o rsett-Euler-Maruyama method is described as 
\begin{align*}
	Y_{t_m} &\approx \begin{multlined}[t]
		\mathbb E[X_{t_{m+1}}|X_{t_m}]^\ast Z_{t_{m+1}} +\int_{t_m}^{t_{m+1}}\mathbb E[X_s|X_{t_m}]^\ast \\
		\left(H(t_{m+1},Z_{t_{m+1}})+\frac{H(t_m,Z_{t_m})-H(t_{m+1},Z_{t_{m+1}})}{\Delta t}(t_{m+1}-s)\right)ds
		\end{multlined}\\
	&\approx \begin{multlined}[t]
		\mathbb E[X_{t_{m+1}}|X_{t_m}]^\ast Z_{t_{m+1}} +\left(\int_{t_m}^{t_{m+1}}\mathbb E[X_s|X_{t_m}]^\ast ds\right) H(t_{m+1},Z_{t_{m+1}})\\
		+\left(\int_{t_m}^{t_{m+1}}(t_{m+1}-s)\mathbb E[X_s|X_{t_m}]^\ast ds\right)\frac{H(t_m,\zeta_m)-H(t_{m+1},Z_{t_{m+1}})}{\Delta t},
		\end{multlined}
\end{align*}
where $\zeta_m$ is defined by 
\[
	\zeta_m = \mathbb E[X_{t_{m+1}}|X_{t_m}]^\ast Z_{t_{m+1}} +\left(\int_{t_m}^{t_{m+1}}\mathbb E[X_s|X_{t_m}]^\ast ds\right) H(t_{m+1},Z_{t_{m+1}}).
\]
From \eqref{eq:markovcondexp}, it results in
\begin{equation}
	\begin{cases}
		\displaystyle Z_{t_M}^M \coloneqq G,\\
		\displaystyle Z_{t_m}^M \coloneqq 
		\begin{multlined}[t]
			e^{\Delta tQ}Z_{t_{m+1}}^M+\Delta t\left(\int_0^1e^{(1-\theta)\Delta tQ}d\theta-\int_0^1e^{(1-\theta)\Delta tQ}\theta d\theta\right)H(t_{m+1},Z_{t_{m+1}}^M)\\
			+\Delta t\left(\int_0^1e^{(1-\theta)\Delta tQ}\theta d\theta\right)H(t_{m+1},\zeta_m^M)
		\end{multlined}\\
		\displaystyle \zeta_m^M \coloneqq e^{\Delta tQ}Z_{t_{m+1}}^M+\Delta t\left(\int_0^1e^{(1-\theta)\Delta tQ}d\theta\right)H(t_{m+1},Z_{t_{m+1}}^M),\quad m=0,1\dots,M-1,
	\end{cases}
\end{equation}
which is known as the second-order exponential time differencing Runge-Kutta (ETD2RK) method \cite{Cox2002} for solving system of ODEs \eqref{eq:ODE}.
\paragraph{(General $s$-stage exponential integrators).}
Furthermore, we consider general $s$-stage Euler-Maruyama methods taking the form of
\begin{equation}
	\label{eq:expRK}
	\begin{gathered}
		Z_{t_m}^M = \chi_0(\Delta tQ)Z_{t_{m+1}}^M+\Delta t\sum_{i=1}^sb_i(\Delta tQ)G_{mi},\\
		G_{mi} = H(t_{m+1}-c_i\Delta t,\zeta_{ni}^M),\quad\text{for}\quad i=1,\dots,s,\\
		\zeta_{mi}^M = \chi_i(kQ)Z_{m+1}^M+\Delta t\sum_{j=1}^sa_{ij}(\Delta tQ)G_{nj},\quad\text{for}\quad i=1,\dots,s,
	\end{gathered}
\end{equation}
for $m=0,\dots,M-1$. Here, $Z_{t_m}^M$ approximates $Z_{t_m}$, $U_{mi}$ is the $i$-th internal stage, $s\in\mathbb N$ is the number of stages, $c_i$ are real numbers, and $\chi_i$, $a_{ij}$ and $b_i$ are functions constructed from ``$\phi$-functions'' defined by 
\[
	\phi_l(Q)=\int_0^1e^{(1-\theta)Q}\frac{\theta^{l-1}}{(l-1)!}d\theta\quad\text{for}\quad l\in\mathbb N,\quad\text{and}\quad \phi_0(Q) = e^{Q}.
\]
For example, the three methods mentioned above are obtained from the settings:
\begin{itemize}
	\item Lawson-Euler: $s=1$, $\chi_0(z)=e^z$, $\chi_1(z)=1$, $a_{11}(z)=0$, $b_1(z)=e^z$ and $c_1=0$.
	\item N\o rsett-Euler: $s=1$, $\chi_0(z)=e^z$, $\chi_1(z)=1$, $a_{11}(z)=0$, $b_1(z)=\phi_1(z)$ and $c_1=0$.
	\item ETD2RK: $s=2$, $\chi_0(z)=e^z$, $\chi_1(z)=1$, $\chi_2(z)=e^z$, $a_{11}(z)=a_{12}(z)=a_{22}(z)=0$, $a_{21}(z)=\phi_1(z)$, $b_1(z)=\phi_1(z)-\phi_2(z)$, $b_2(z)=\phi_2(z)$, $c_1=0$ and $c_2=1$.
\end{itemize}
Note that the multi-stage Euler-Maruyama method \eqref{eq:expRK} is the same as exponential integrators (exponential Runge-Kutta methods) for solving systems of ODEs \eqref{eq:ODE}. For details on exponential integrators, we refer to \cite{Hochbruck2010}, a comprehensive survey. We remark on the following:
\begin{remark}
	\label{rem:stiff}
	Exponential integrators work to calculate solutions for stiff systems of ODEs. A system of ODEs is called ``stiff'' if it raises numerical instability to explicit ODE solvers; solvers require exceedingly small step sizes to solve such equations. As described in \cite{Higham1993,Wanner1996,Trefethen1996,Trefethen2005,Strang2007}, method-of-lines approach (spatial variable discretization) for solving parabolic PDEs often results in large stiff systems of ODEs. Exponential integrators are applicable to these systems effectively.
\end{remark}
\begin{remark}
	$\chi_i$, $a_{ij}$, $b_i$ and $c_i$ in \eqref{eq:expRK} are prescribed parameters to be set so that we can obtain various schemes that have (stiff / nonstiff) orders of convergence.
\end{remark}
\begin{remark}
	\label{rem:krylov}
	Exponential integrators exploits matrix functions $\phi_l$. However, evaluating them numerically is not straightforward and has been studied in numerical literature. A standard approach that is widely used is a combination of Pad\`e approximations and scaling-and-squaring methods. Although it enables efficient evaluation, note that it is only applicable to $\phi_l$ of moderate dimension. For solving ODE systems whose dimension is large, it is advantageous to apply Krylov subspace methods; instead of evaluating $\phi_l$ itself, its action on a state vector is approximated with a vector on a Krylov subspace whose dimension is small.
\end{remark}

\section{Application to BSDEs driven by Brownian motion}
\label{sec:Application}
In this section, we are interested in computing Markov BSDEs driven by Brownian motion \eqref{BMBSDE} with an appropriate spatial discretization. From a probabilistic point of view, it can be seen as approximating a BSDE driven by a Brownian motion with a BSDE driven by a CTMC. From a differential equation point of view, on the other hand, it can be seen as the method of lines, approximating a second-order parabolic PDE with a system of ODEs. As mentioned in remark \ref{rem:stiff}, the method of lines discretization of parabolic PDEs leads to stiff systems of ODE, and our multi-stage Euler-Maruyama methods efficiently work for them. One can represent this situation as Figure \ref{fig:diagram}.

\begin{figure}[htbp]
  \centering
  \resizebox{.9\textwidth}{!}{\setlength{\unitlength}{.05cm}
  \begin{picture}(290,50)
  \put(30,10){\makebox(0,0){Parabolic PDE}}
  \put(30,60){\makebox(0,0){Brownian BSDE}}
	\put(180,60){\makebox(0,0){CTMC BSDE}}
  \put(180,10){\makebox(0,0){ODE system}}
  \put(290,40){\makebox(0,0){Exponential}}
  \put(290,30){\makebox(0,0){Integrator}}
  \put(80,10){\vector(1,0){50}}
  \put(105,0){\makebox(0,0){Method of lines}}
  \put(80,60){\vector(1,0){50}}
  \put(105,50){\makebox(0,0){Spatial discretization}}
  \put(30,15){\vector(0,1){40}}
  \put(30,55){\vector(0,-1){40}}
  \put(0,40){\makebox(0,0){Nonlinear}}
  \put(0,30){\makebox(0,0){Feynman-Kac}}
  \put(180,15){\vector(0,1){40}}
  \put(180,55){\vector(0,-1){40}}
  %\put(150,35){\makebox(0,0){Theorem \ref{thm:CTMCFeynmanKac}}}
  \put(145,35){\makebox(0,0){Cohen-Szpruch\cite{Cohen2012}}}
  \put(215,60){\vector(3,-1){45}}
  \put(215,10){\vector(3,1){45}}
  \put(260,15){\makebox(0,0){Variation}}
  \put(260,5){\makebox(0,0){-of-constants}}
  \put(275,58){\makebox(0,0){Euler-Maruyama($\ast$)}}
  \end{picture}
  }
  \caption{A diagram representing the relation between key ingredients for the argument up to here on the present paper. It is based on the diagram given in \cite{Cohen2012}. The arrow denoted as $(\ast)$ has been newly drawn by the arguments in section \ref{sec:Euler-Maruyama}.}
  \label{fig:diagram}
\end{figure}
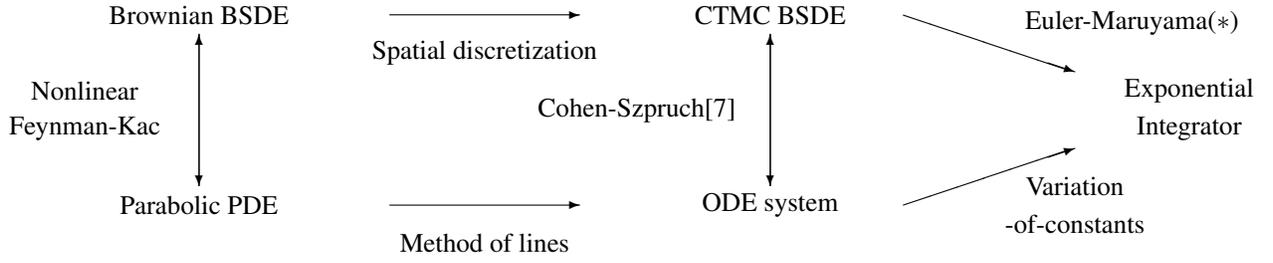

In Section \ref{sec:MOL}, we present a construction of Markov BSDEs driven by CTMCs from spatial discretization of Markov BSDEs driven by Brownian motion. Through the method of lines discretization to the associated second-order parabolic PDEs, systems of ODEs are obtained. We see that the systems of ODEs are equivalent to Markov BSDEs driven by CTMCs. In section \ref{sec:SGcomb}, we propose a numerical scheme for solving high-dimensional Markov BSDEs driven by Brownian motion, that incorporates the Euler-Maruyama methods with the sparse grid combination technique. 

Hereafter, we focus on \eqref{BMBSDE} such that $(\mathcal X_t)_{t\in[0,T]}$ is time homogeneous, namely, $\mu(t,x)$ and $\sigma(t,x)$ do not depend on $t$. Then, we can simply write $\mu(t,x)=\mu(x)$ and $\sigma(t,x)=\sigma(x)$, and the subscript of the infinitesimal generator can also be omitted: $\mathcal L_t = \mathcal L$.

\subsection{Method-of-lines spatial discretization}
\label{sec:MOL}
\subsubsection{The case of $1$-dimensional state space}
First, we shall discuss the case of $d=1$ for simplicity. Let a strictly increasing sequence $\Pi=\{x_i\}_{i=-N_0}^{N_0}$ of length $N\coloneqq 2N_0+1$ be the set of nodes on $\mathbb R$ and define $\delta x_i\coloneqq x_{i+1}-x_i$ for $-N_0\le i<N_0$. For any function $v:\mathbb R\to\mathbb R$, let $v^\Pi = (v(t,x_{-N_0}),v(t,x_{-N_0+1}),\dots,v(t,x_{N_0}))^\ast\in\mathbb R^N$ be the evaluation of $v$ over $\Pi$. Then, derivatives evaluated at nodes of $\Pi$ are replaced by 
\[
	\begin{pmatrix}
		\displaystyle\frac{\partial v}{\partial x}(x_{-N_0})\\
		\vdots\\
		\displaystyle\frac{\partial v}{\partial x}(x_{N_0}))
	\end{pmatrix}
	\approx D_1u^\Pi,\quad 
		\begin{pmatrix}
		\displaystyle\frac{\partial^2 v}{\partial x^2}(x_{-N_0})\\
		\vdots\\
		\displaystyle\frac{\partial^2 v}{\partial x^2}(x_{N_0}))
	\end{pmatrix}
	\approx D_2u^\Pi.
\]
Here, $N\times N$-matrices $D_1$ and $D_2$ are defined by
\begin{align}
	\begin{split}
	&e_i^\ast D_1e_j = 
	\begin{cases}
		\displaystyle\frac{-\delta x_i}{\delta x_{i-1}(\delta x_{i-1}+\delta x_i)},&j=i-1,\\
		\displaystyle\frac{\delta x_i-\delta x_{i-1}}{\delta x_i\delta x_{i-1}},&j=i,\\
		\displaystyle\frac{\delta x_{i-1}}{\delta x_i(\delta x_{i-1}+\delta x_i)},&j=i+1,\\
		0&\text{otherwise},
	\end{cases}\quad\text{for}\quad-N_0<i<N_0,\\
	&e_{-N_0}^\ast D_1e_i=e_{N_0}^\ast D_1e_i=0,\quad \text{for}\quad -N_0\le i\le N_0,
	\end{split}\label{D}\\
	\begin{split}
	&e_i^\ast D_2e_j = 
	\begin{cases}
		\displaystyle\frac{2}{\delta x_{i-1}(\delta x_{i-1}+\delta x_i)},&j=i-1,\\
		\displaystyle\frac{-2}{\delta x_i\delta x_{i-1}},&j=i,\\
		\displaystyle\frac{2}{\delta x_i(\delta x_{i-1}+\delta x_i)},&j=i+1,\\
		0&\text{otherwise},
	\end{cases}\quad\text{for}\quad-N_0<i<N_0,\\
	&e_{-N_0}^\ast D_2e_i=e_{N_0}^\ast D_2e_i=0,\quad \text{for}\quad -N_0\le i\le N_0,
	\end{split}\label{D2}
\end{align}
that result from the central difference scheme, and we denote $e_i$ as the $(i+N_0+1)$-th unit vector in $\mathbb R^N$ whose $(i+N_0+1)$-th element is $1$. We then solve, in place of PDE \eqref{PDE}, a system of $N$ ODEs in what follows:
\begin{equation}
  \label{TVP}
  \frac{dU_t^\Pi}{dt} + QU_t^\Pi +F(t,U_t^\Pi)=0\quad\text{for}\quad t\in[0,T],\quad U_T^\Pi = G.
\end{equation}
Here, $G=(g(x_{-N_0}),\dots,g(x_{N_0}))^\ast\in\mathbb R^N$, $F:[0,T]\times\mathbb R^N\to\mathbb R^N$ is defined by
\begin{equation}
	\label{eq:definitionF}
	F(t,z)=\begin{pmatrix}
		f(t,x_{-N_0},e_{-N_0}^\ast z,\sigma(x_{-N_0})e_{-N_0}^\ast D_1z)\\\vdots\\f(t,x_{N_0},e_{N_0}^\ast z,\sigma(x_{N_0})e_{N_0}^\ast D_1z)
	\end{pmatrix},
\end{equation}
and
\[
	Q=\operatorname{diag}(\mu^\Pi)D_1 + \frac12\operatorname{diag}((\sigma^2)^\Pi)D_2
\]
approximates the infinitesimal generator $\mathcal L$ of $\mathcal X$. Using \eqref{D} and \eqref{D2}, each element of $Q$ is 
\begin{equation}
	\begin{split}
	&e_i^\ast Qe_j = 
	\begin{cases}
		\displaystyle\frac{\sigma^2(x_i)-\delta x_i\mu(x_i)}{\delta x_{i-1}(\delta x_{i-1}+\delta x_i)},&j=i-1,\\
		\displaystyle\frac{(\delta x_i-\delta x_{i-1})\mu(x_i)-\sigma^2(x_i)}{\delta x_i\delta x_{i-1}},&j=i,\\
		\displaystyle\frac{\sigma^2(x_i)+\delta x_{i-1}\mu(x_i)}{\delta x_i(\delta x_{i-1}+\delta x_i)},&j=i+1,\\
		0&\text{otherwise},
	\end{cases}\quad\text{for}\quad-N_0<i<N_0,\\
	&e_{-N_0}^\ast Qe_i=e_{N_0}^\ast Qe_i=0,\quad \text{for}\quad -N_0\le i\le N_0.
	\end{split}\label{Q}
\end{equation}
In Appendix \ref{sec:convergence}, we give a convergence result of \eqref{TVP} to \eqref{PDE} in a case which a unique classical solution of \eqref{PDE} exists. 
\begin{theorem}[Convergence]
	\label{thm:convergence}
	Consider the case of $d=1$ and take a spatial grid $\Pi^{(N,\Delta x)}\coloneqq\{i\Delta x\}_{i=-N_0}^{N_0}$ for some $\Delta x>0$. Suppose that Assumption \ref{basicassumption}, \ref{globallipshcitz} in Appendix \ref{sec:convergence}, and \eqref{PDE} admits a unique solution $u$. (For example, \eqref{PDE} is uniquely solvable in the case of Lemma \ref{wellposedness}.) Denote $U_t^{(N,\Delta x)}$ as a unique solution of \eqref{TVP} in the case of $\Pi^{(N,\Delta x)}$. For any compact set $K\subset\mathbb R$, it holds
	\[
		\lim_{\Delta x\to0}\lim_{N\to\infty}\sup_{\substack{-N_0\le i\le N_0\\i\in\mathbb Z,i\Delta x\in K}}|u(t,i\Delta x)-e_i^\ast U_t^{(N,\Delta x)}| = 0,
	\]
	where $e_i$ is the $(i+N_0+1)$-th unit vector in $\mathbb R^N$ whose $(i+N_0+1)$-th element is $1$.
\end{theorem}
Thus, for obtaining the system of ODE which approximate the PDE with a small error in given bounded $K$, we should take a sufficiently small $k>0$ and expand the spatial grid for sufficiently large $N$ that depends on $k$.

\subsubsection{The case of $d$-dimensional state space}
\label{sec:d-MOL}
Using the Kronecker product ``$\otimes$'', the argument in the case of $d=1$ can be carried out in multidimensional cases, that is, for any $d\in\mathbb N$. For $p=1,\dots,d$, let a strictly increasing sequence $\Pi^{(p)} = (x_i^{(p)})_{i=-N_0^{(p)}}^{N_0^{(p)}}$ of length $N^{(p)} = 2N_0^{(p)}+1$ be the set of nodes on the $p$-th axis in $\mathbb R^d$, and let $D_1^{(p)}$ and $D_2^{(p)}$ are the corresponding $N^{(p)}\times N^{(p)}$ difference matrices defined by \eqref{D} and \eqref{D2}, constructed on $\Pi^{(p)}$. Consider the grid on $\mathbb R^d$ by
\[
	\Pi = \Pi^{(1)}\otimes\Pi^{(2)}\otimes\cdots\otimes\Pi^{(d)} = (x_i = (x_{i_1}^{(1)},x_{i_2}^{(2)},\dots,x_{i_d}^{(d)}):i=1,\dots,N),
\]
where $N\coloneqq\prod_{p=1}^dN^{(p)}$ is the total size of $\Pi$, and multi-indices $(i_1,i_2,\dots,i_d)$ are ordered lexicographically. For $v:\mathbb R^d\to\mathbb R$, 1st and second derivatives along the $p$-th axis are approximated by
\[
	\frac{\partial v}{\partial x^{(p)}}(x_i)\approx e_i^\ast\widetilde D_1^{(p)}v^\Pi\quad\text{and}\quad \frac{\partial^2 v}{(\partial x^{(p)})^2}(x_i)\approx e_i^\ast\widetilde D_2^{(p)}v^\Pi,
\]
where matrix $\widetilde D_k^{(p)}$ for $k=1,2$ and $p=1,\dots,d$ is given by
\[
	\widetilde D_k^{(p)} \coloneqq I_{N^{(1)}}\otimes\cdots\otimes I_{N^{(p-1)}}\otimes D_k^{(p)}\otimes I_{N^{(p+1)}}\otimes\cdots\otimes I_{N^{(d)}}.
\]
In multidimensional cases, we additionally need to specify the approximation of cross derivatives since $\mathcal L$ possibly contains them. In this work, we approximate the cross derivative along the $p$-th and $q$-th axes as
\[
	\frac{\partial^2 v}{\partial x^{(p)}\partial x^{(q)}}(x_i) \approx e_i^\ast D_1^{(p,q)}v^\Pi,
\]
where
\begin{multline*}
	D_1^{(p,q)} \coloneqq I_{N^{(1)}}\otimes\cdots\otimes I_{N^{(p-1)}}\otimes D_1^{(p)}\otimes I_{N^{(p+1)}}\otimes\cdots\\
	\otimes I_{N^{(q-1)}}\otimes D_1^{(q)}\otimes I_{N^{(q+1)}}\otimes\cdots\otimes I_{N^{(d)}},
\end{multline*}
for $p<q$. In this situation, $Q$ is defined by 
\begin{multline}
	\label{eq:QDdim}
	Q = \sum_{p=1}^d \operatorname{diag}((\mu^{(p)})^\Pi)\widetilde D_1^{(p)} + \sum_{p=1}^{d-1}\sum_{q=p+1}^d\operatorname{diag}(((\sigma\sigma^\ast)^{(p,q)})^\Pi)\widetilde D_1^{(p,q)}\\
	+\frac12\sum_{p=1}^d\operatorname{diag}(((\sigma\sigma^\ast)^{(p,p)})^\Pi)\widetilde D_2^{(p)}
\end{multline}
Defining $F:[0,T]\times\mathbb R^N\to\mathbb R^N$ by 
\[
	F(t,z)=\begin{pmatrix}
		f(t,x_1,e_1^\ast z,\sigma^\ast(x_1)(e_1^\ast\widetilde D_1^{(1)}z,\dots,e_1^\ast\widetilde D_1^{(d)}z)^\ast)\\
		\vdots\\
		f(t,x_N,e_N^\ast z,\sigma^\ast(x_N)(e_N^\ast\widetilde D_1^{(1)}z,\dots,e_N^\ast\widetilde D_1^{(d)}z)^\ast)
	\end{pmatrix}
\]
and $G=(g(x_1),\dots,g(x_N))^\ast$, the system of ODEs results in the same form as \eqref{TVP}:
\begin{equation}
	\label{eq:multiODE}
  \frac{dU_t^\Pi}{dt} + QU_t^\Pi+F(t,U_t^\Pi)=0,\quad U_T^\Pi = G.
\end{equation}

\subsubsection{Probabilistic interpretation}
Recall that $Q$ in \eqref{Q} or \eqref{eq:QDdim} is constructed from the spatial discretization of the infinitesimal generator $\mathcal L$. In the probabilistic manner, it is natural to interpret $Q$ as the Q-matrix of a time-homogeneous CTMC. Since $Q$ might no longer be a Q-matrix, it is required to see the ``validity'' conditions of $Q$ to be the Q-matrix. In the case of $d=1$, one can easily give the following sufficiency condition. It guarantees the validity of a CTMC constructed by $Q$ provided the spatial difference is sufficiently fine. 
\begin{proposition}[Validity]
	\label{3.1}
  $Q$ defined by \eqref{Q} is the transition rate matrix of a continuous-time Markov chain if
  \begin{equation}
    \label{cond}
    0<\max_{-N_0\le i\le N_0-1}\{\delta x_i\}\le\min_{\substack{-N_0\le i\le N_0\\\mu(x_i)\neq0}}\left\{\frac{\sigma^2(x_i)}{|\mu(x_i)|}\right\}.
  \end{equation}
  Additionally, if the above inequality is strict, $e_{i-1}^\ast Qe_i$ and $e_{i+1}^\ast Qe_i$ are positive for all $i=-N_0+1,\dots,N_0-1$.
\end{proposition}
\begin{proof}
	See Appendix \ref{sec:ProofOfLem3.1}.
\end{proof}
Note that validity conditions in multi-dimensional settings would be more complicated; the existence of terms that involve cross derivatives occasionally violates the definition of Q-matrices. Our approach presented in Section \ref{sec:Euler-Maruyama}, that substantially solves systems of ODEs using exponential integrators, works without any issues at least numerically, regardless of the validity. However, in certain situations when we need to simulate the CTMC, we should instead employ several discretization schemes that avoid the invalidity issues \cite{BouRabee2018,Meier2023}.

Assuming the validity of $Q$, let $X$ be a finite-state Markov chain with $Q$ as its Q-matrix. We now consider the Markov BSDE arising from \eqref{TVP} as 
\begin{equation}
	\label{eq:molCTMCBSDE}
	\begin{split}
	Y_t = X_T^\ast G+\int_{]t,T]}X_{s-}^\ast F(s,Z_s)ds -\int_{]t,T]}dM_s^\ast Z_s.
	\end{split}
\end{equation}
Thus, we can regard \eqref{eq:molCTMCBSDE} as spatially discretized counterpart of \eqref{BMBSDE}, and applying exponential integrators to \eqref{TVP} is equivalent to the (multi-stage) Euler-Maruyama methods of \eqref{eq:molCTMCBSDE}. 

We give a result on a uniqueness of \eqref{eq:molCTMCBSDE} under standard conditions. Theoretical justification of this probabilistic interpretation is completed with it.
\begin{proposition}
	\label{lemma3.2}
	Suppose that $\int_0^Tf(t,x,0,0)^2dt <\infty$ for any $x\in\mathbb R^d$, and that for some $L>0$, 
	\[
		|f(t,x,y,z)-f(t,x,y',z')|^2 \le L(|y-y'|^2+\|z-z'\|^2)
	\]
	for all $t\in[0,T]$, $x\in\mathbb R^d$, $y,y'\in\mathbb R$, and $z,z'\in\mathbb R^d$. Suppose further that $Q$ is valid, $e_i^\ast\widetilde D_1^{(p)}1=0$, and that 
	\begin{equation}
		e_j^\ast Qe_i = 0 \Longrightarrow e_j^\ast \widetilde D_1^{(p)}e_i = 0\quad\text{for}\quad p=1,\dots,d,
		\label{cond:q}
	\end{equation}
	for $i,j=1,\dots,N$. Then, \eqref{eq:molCTMCBSDE}, which is derived from the Markov chain approximation of \eqref{PDE}, has a unique solution.
\end{proposition}
\begin{proof}
	See Appendix \ref{sec:ProofOfLem3.2}.
\end{proof}

\begin{remark}
	As a related study on CTMC approximation of SDEs, we refer to \cite{Mijatovic2013,Cui2018,Li2018,Kirkby2020,Cui2019}.
\end{remark}

\begin{remark}
	In \cite{Cohen2012}, the authors presented a CTMC version of least-squares Monte Carlo methods. Although this type of method is also a natural counterpart of numerical solutions of Markov BSDEs based on the Euler-Maruyama temporal discretization, we note that it is not suitable to solve CTMC-driven Markov BSDEs arising from spatial discretization of Brownian motion-driven Markov BSDEs. A major bottleneck is that the resulting $Q$-matrix may contain quite larger absolute values. To illustrate it, suppose that $\mu(x)=\mu x$, $\sigma(x)=\sigma x$, and $\delta x_i\equiv \Delta x>0$. Plugging them into \eqref{Q}, we obtain
	\[
		e_i^\ast Qe_{i-1} = \frac{\sigma^2 i^2-\mu i}2,\quad e_i^\ast Qe_{i+1} = \frac{\sigma^2 i^2+\mu i}2.
	\]
	If $i$ is nearby $N$ and $N$ is large, these elements take large values, which implies the resulting Markov chain jumps too rapidly. It interferes with us simulating CTMCs naively using Gillespie's exact simulation or the 1st-order approximation of the transition probability matrix; the former suffers from a tremendously large number of jumps, and the latter method requires the temporal step small enough for each row of the resulting matrix to represent probabilities.
\end{remark}

\subsection{Multilevel spatial approximation using sparse grids}
\label{sec:SGcomb}
In multi-dimensional settings, even in 2-dimensional cases, spatial discretization typically results in CTMCs with a large state space. That brings us to a numerical limitation known as the curse of dimensionality, and computation in more than 4-dimensional cases is challenging. In this section, we propose a CTMC approximation method of solving high-dimensional BSDEs, using sparse grids. 

Sparse grid methods date back to a study by Smolyak \cite{Smolyak1963}, in which he presented an algorithm for constructing multivariate quadrature and interpolation rules from a linear combination of tensor products of univariate ones in a specific manner. It mitigates the curse of dimensionality, which conventional formulas that compute solutions on the ``full grid'' cannot overcome. Beyond quadratures and interpolations, the construction of sparse grids has applications to various fields such as data mining \cite{Garcke2001} or differential equations. As introduced in Section 4.5 in \cite{Bungartz2004}, different approaches have been presented to apply the idea of sparse grids to solving partial differential equations. The most straightforward one among them will be the sparse grid combination technique. Motivated by the observation that the sparse grid can be decomposed to a combination of several coarser rectangular grids, this approach constructs a numerical solution from a linear combination of numerical solutions of PDEs on those grids. As a result, one can compute solutions by applying the PDE solver at hand to our PDE on each resolution grid without specific treatments. 

Let us apply the sparse grids combination technique to the presented method in section \ref{sec:MOL}. We first need to approximate the spatial domain $\mathbb R^d$ of \eqref{PDE} with a fixed, bounded, and rectangular domain on which we can define sparse grids. For the sake of simplicity, we set a hypercube $[-1,1]^d$ as such a domain; we note that the argument in this section can be applied to arbitrary bounded rectangular domains using dilation and translation. For $l\in\mathbb N$, let $\chi^l = \{z_1^l,\dots,z_{m_l}^l\}\subset\mathbb R$ be a prescribed set of nodes parametrized by $l$; in this section, we specifically set equidistant nodes
\[
  z_i^l = \frac{2(i-1)}{m_l-1}-1,\quad\text{for}\quad i=1,\dots,m_l,
\]
with $m_l = 2^l+1$. For $t\in[0,T]$, consider the piecewise linear interpolation of the solution $U_t^{(l_1,\dots,l_d)}\coloneqq U_t^{\chi^{l_1}\otimes\cdots\otimes\chi^{l_d}}$ of \eqref{eq:multiODE} for $\Pi=\chi^{l_1}\otimes\cdots\otimes\chi^{l_d}$ defined as
\[
	u^{(l_1,\dots,l_d)}(t,x_1,\dots,x_d) = \sum_{i_1=1}^{m_{l_1}}\cdots\sum_{i_d=1}^{m_{l_d}}U_t^{(l_1,\dots,l_d)}(z_{i_1}^{l_1},\dots,z_{i_d}^{l_d})\prod_{p=1}^da_{i_p}^{l_p}(x_p)
\]
for $(x_1,\dots,x_d)\in[-1,1]^d$, where $a_{i_p}^{l_p}(x_p)$ are the one-dimensional piecewise linear bases defined as
\[
	a_i^l(x) = (1-\frac{m_l-1}2|x-z_i^l|)\vee0\quad\text{for}\quad x\in[-1,1],\quad i=1,\dots,m_l,\quad l\in\mathbb N,
\]
and $U_t^{(l_1,\dots,l_d)}(z_{i_1}^{l_1},\dots,z_{i_d}^{l_d})$ means the component of the vector $U_t^{(l_1,\dots,l_d)}$ corresponding to $(z_{i_1}^{l_1},\dots,z_{i_d}^{l_d})\in\chi^{l_1}\otimes\cdots\otimes\chi^{l_d}$. The sparse grid solution with the level parameter $q\in\mathbb N$ is then constructed from $u^{(l_1,\dots,l_d)}(t,x_1,\dots,x_d)$ by 
\begin{equation}
	\label{eq:SGformula}
  u^{\text{SG},q}(t,x_1,\dots,x_d) = \sum_{\substack{\mathbf{l} = (l_1,\dots,l_d)\in\mathbb N^d\\q-d+1\le|\mathbf l|_1\le q}}(-1)^{q-|\mathbf l|_1}\binom{d-1}{q-|\mathbf l|_1}\cdot u^{\mathbf l}(t,x_1,\dots,x_d)
\end{equation}
for $(x_1,\dots,x_d)\in[-1,1]^d$, where $|\mathbf l|_1 = l_1+\cdots+l_d$ for $\mathbf l = (l_1,\dots,l_d)$ and $\displaystyle\binom{d-1}{q-|\mathbf l|_1}$ is the binomial coefficient. 

Numerically, we obtain an algorithm in what follows: 
\begin{enumerate}
	\item Solve systems of ODEs for each $\mathbf l$ such that $q-d+1\le|\mathbf l|_1\le q$.
	\item For each multi-index $\mathbf l$ and discrete time $t_m$, construct the $d$-dimensional piecewise linear interpolant $\mathbf x\mapsto u^{\mathbf l}(t_m,\mathbf x)$ of $\{U_{t_m}^{\mathbf l}(\mathbf z), \mathbf z\in\chi^{l_1}\otimes\cdots\otimes\chi^{l_d}\}$.
	\item Combine interpolants according to \eqref{eq:SGformula}. The obtained function $u^{\text{SG},q}(t,\mathbf x)$ can be evaluated at any $\mathbf x\in[-1,1]^d$ for each discretized time $t_m$. 
\end{enumerate}
The numerical solution $u^{\text{SG},q}$ achieves much less computational cost with a slight deterioration in its quality, in comparison to the corresponding full grid formula i.e. $u^{(n,\dots,n)}$ for $n=q-d+1$. Specifically, the total number of spatial points that the sparse grid formula with parameter $q$ contains is $O(2^qq^{d-1})$. Taking that the full grid formula at the same level is comprised of $O(2^{qd})$ spatial points into account, it turns out to be a significant reduction in computational complexity \cite{Griebel1992}. We can also give an error estimate based on arguments in existing leteratures such as \cite{Griebel1992,Leentvaar2008,LopezSalas2017} . Let $\alpha>0$ be the order of accuracy of the spatial discretization scheme. Suppose that we know the exact solutions $U_t^{\mathbf l}$ of each systems of ODEs in \eqref{eq:SGformula}, the exact solution $u$ (of the approximated PDE on $[-1,1]^d$) are sufficiently smooth, and an error expansion
\begin{multline*}
	u(t,x_1,\dots,x_d) - u^{\mathbf l}(t,x_1,\dots,x_d) \\
	= \sum_{i=1}^d\sum_{j_1=1}^d\sum_{j_2=j_1+1}^d\cdots\sum_{j_i=j_{i-1}+1}^dC^{(i)}(t,x_1,\dots,x_d,h_{j_1},h_{j_2},\dots,h_{j_i})h_{j_1}^\alpha h_{j_2}^\alpha\cdots h_{j_i}^\alpha
\end{multline*}
exists for some bounded functions $C^{(i)}$ for $i=1,\dots,d$. Here, $h_i=2^{-i}$ for $i\in\mathbb N$. Then, according to the existing literatures aforementioned, the error estimate reads
\[
	u(t,x_1,\dots,x_d)-u^{\text{SG},q}(t,x_1,\dots,x_d) = O(h_q^{-\alpha}(\log_2(h_q^{-1}))^{d-1})(=O(2^{-\alpha q}q^{d-1}))
\]
as $q\to\infty$, which is slightly worse than $O(h_q^{-d})(=O(2^{-\alpha q}))$ the estimate in the full grid case. 

\section{Numerical Results}
\label{sec:numericalexperiments}
In this section, we demonstrate the efficiency and stability of the numerical approach presented in Section \ref{sec:Application}, using several examples. We apply spatial discretization to the BSDEs driven by Brownian motions, obtain BSDEs driven by CTMCs (i.e. a system of ODEs), and calculate numerical solutions using multi-stage Euler-Maruyama methods (i.e. exponential integrators.) Before moving on to specific results, we explain the details on settings in what follows:

\paragraph{Spatial Discretization}
We approximate the unbounded spatial domain of the problem at hand with a bounded one. Because all of the BSDEs we solve in this section have $(0,\infty)^d$ as spatial domains of them, we approximate as 
\begin{equation}
	\label{eq:grid1}
	(0,\infty)^d\approx[0,2x_1]\times\cdots\times[0,2x_d]
\end{equation}
for some $x_1,\dots,x_d>0$, or sometimes as 
\begin{equation}
	\label{eq:grid2}
	(0,\infty)^d\approx[\Delta_{x_1},2x_1-\Delta_{x_1}]\times\cdots\times[\Delta_{x_d},2x_d-\Delta_{x_d}]
\end{equation}
for small $\Delta_{x_1},\dots,\Delta_{x_d}>0$. The latter is applied to the problems that should be evaluated only at points on the first quadrant. Here, $(x_1,\dots,x_d)$ is a point at which we want to evaluate the numerical solution. 

Throughout the section, $\Pi_x^{\mathrm{Unif}}(x_{\mathrm{left}},x_{\mathrm{center}},x_{\mathrm{right}},N_{x,0})=(x_i)_{i=-N_{x,0}}^{N_{x,0}}$ means the standard one-dimensional uniform grid such that $x_{-N_{x,0}}=x_{\mathrm{left}}$, $x_0=x_{\mathrm{center}}$ and $x_{N_{x,0}}=x_{\mathrm{right}}$. We occasionally omit its arguments and simply write $\Pi_x^{\mathrm{Unif}}$. In addition to the uniform grid, a non-uniform grid is also employed for spatial discretization. We use a version of the (one-dimensional) Tavella-Randall-type grids \cite{Randall2000} in what follows:
\begin{equation}
	x_k =
	\begin{cases}
		\displaystyle x_{\mathrm{center}} + g_1\sinh\left(\mathop{\operatorname{arc}\sinh}\left(\frac{x_{\mathrm{center}}-x_{\mathrm{left}}}{g_1}\right)\cdot\frac{k}{N_{x,0}}\right),\quad k=-N_{x,0},\dots,-1,0,\\
		\displaystyle x_{\mathrm{center}} + g_2\sinh\left(\mathop{\operatorname{arc}\sinh}\left(\frac{x_{\mathrm{right}}-x_{\mathrm{center}}}{g_2}\right)\cdot\frac{k}{N_{x,0}}\right),\quad k=1,2,\dots,N_{x,0},
	\end{cases}
	\label{eq:Tavella-Randall}
\end{equation}
where $N_x \coloneqq 2N_{x,0}+1$ is the grid size, $x_{\mathrm{left}}$ and $x_{\mathrm{right}}$ are the leftmost and rightmost points of the domain, $x_{\mathrm{center}}\in(x_{\mathrm{left}},x_{\mathrm{right}})$ is the central point of the grid, and $g_1$ and $g_2$ are parameters for the left- and right-side of the grid, respectively. Note that $x_{-N_{x,0}} = x_{\mathrm{left}}$, $x_0=x_{\mathrm{center}}$ and $x_{N_{x,0}}=x_{\mathrm{right}}$. Intuitively, setting $g_1,g_2\ll x_{\mathrm{right}}-x_{\mathrm{left}}$ leads to the grid that is highly concentrated around $x_{\mathrm{center}}$. It is commonly used in numerical computation for pricing options to mitigate the effect of the nonlinearity of the payoff function  \cite{Bodeau2000,Mijatovic2013,Randall2000}. Similarly to the uniform grid, denote $\Pi_x^{\mathrm{TR}}(x_{\mathrm{left}},x_{\mathrm{center}},x_{\mathrm{right}},N_{x,0},g_1,g_2)$ as the Tavella-Randall grid \eqref{eq:Tavella-Randall} whose parameters are $(x_{\mathrm{left}},x_{\mathrm{center}},x_{\mathrm{right}},N_{x,0},g_1,g_2)$.

\begin{figure}[ht]
	\centering
	\begin{minipage}[b]{\textwidth}
		\includegraphics[width=\textwidth]{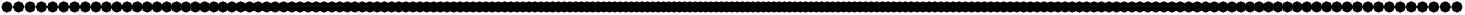}
		\caption{$g_1=g_2=50.0$}
		\label{fig:demo1}
	\end{minipage}\\
	\begin{minipage}[b]{\textwidth}
		\includegraphics[width=\textwidth]{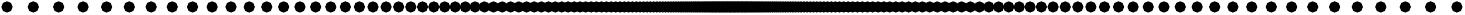}
		\caption{$g_1=g_2=5.0$}
		\label{fig:demo1}
	\end{minipage}\\
	\begin{minipage}[b]{\textwidth}
		\includegraphics[width=\textwidth]{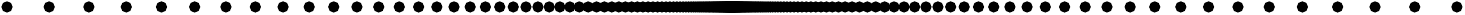}
		\caption{$g_1=g_2=0.5$}
		\label{fig:demo1}
	\end{minipage}
	\caption{Examples of Tavella-Randall grids with $x_{\mathrm{left}}=0$, $x_{\mathrm{center}}=100$, $x_{\mathrm{right}}=200$, $N_{x,0}=100$ and $g_1=g_2\in\{50,5,0.5\}$.}
	\label{fig:TVgrid}
\end{figure}

\paragraph{Temporal Discretization}
We employ solvers implemented in \verb+DifferentialEquations.jl+ \cite{Rackauckas2017}, listed below:
\begin{itemize}
	\item \verb+LawsonEuler+ : A single-stage method of classical/stiff order 1/1, referred to as the Lawson-Euler method \cite{Lawson1967}.
	\item \verb+NorsettEuler+ : A single-stage method of classical/stiff order 1/1, referred to as the N\o rsett-Euler method or ETD1RK method \cite{Cox2002}.
	\item \verb+ETDRK2+ : A 2-stage method of classical/stiff order 2/2 \cite{Cox2002}.
	\item \verb+ETDRK3+ : A 3-stage method of classical/stiff order 3/3 \cite{Cox2002}.
	\item \verb+ETDRK4+ : A 4-stage method of classical/stiff order 4/2 \cite{Cox2002}.
	\item \verb+HochOst4+ : A 5-stage method of classical/stiff order 4/4, developed by Hochbruck and Ostermann \cite{Hochbruck2005}.
\end{itemize}
Taking temporal grid size $N_t\in\mathbb N$, we calculate solutions on the grid $\Pi_t^{\mathrm{Unif}}(N_t)=(i\Delta t)_{i=0}^{N_t}$ using these exponential integrators. Here, $\Delta t = T/N_t$ is the step size. Note that a large-scale system of ODEs is obtained from the spatial discretization in each experiment. In this case, employing Krylov subspace methods in evaluating matrix exponentials and related $\phi$ functions is more effective, as described in Remark \ref{rem:krylov}. In all the experiments, we use the Arnoldi iteration with a size-$m$ Krylov subspace, which is readily available on all the solvers above. For simplicity, we always take $m=100$.

\paragraph{A least-squares Monte Carlo method: A reference}
BSDEs (driven by Brownian motion) that appeared in this section include those whose analytical solutions are unknown. In experiments of those BSDEs, as a reference, we shall report numerical solutions using the least squares Monte Carlo (LSMC) method \cite{Longstaff2001,Gobet2005,Bender2012} using Laguerre polynomials as
\begin{equation}
	\mathrm{poly}_p^{\mathrm{Lag}}(x)=\sum_{k=0}^p\frac{(-1)^k}{k!}\binom pk x^k,\quad p=0,1,2,\dots.
	\label{eq:Laguerre}
\end{equation}
Note that for the multi-dimensional case, the basis function corresponds to the Cartesian product of \eqref{eq:Laguerre}. Since LSMC methods include randomness, we independently calculate solutions for $50$ times, and report the mean values, the standard deviations, and the total runtimes.

\paragraph{Implementation}
All of our experiments were performed on a 3.70 GHz, 64-GB RAM Linux workstation. Our code was written entirely in Julia \cite{Bezanson2017} and all the plots were produced using \verb+Plot.jl+ \cite{Christ2023}. The full code for the experiments is available at \href{https://github.com/kanekoakihiro/EMCTMCBSDE}{https://github.com/kanekoakihiro/EMCTMCBSDE}.

\subsection{European call option under the Black-Scholes model}
\label{sec:exp1BSlinear}
First, we consider a linear BSDE arising from pricing a European call option under the Black-Scholes model:
\begin{equation}
	\begin{split}
		\mathcal S_t &= s_0 + \int_0^t\mu\mathcal S_sds + \int_0^t\sigma\mathcal S_sdW_s,\\
		\mathcal Y_t &= (\mathcal S_T-K)^+-\int_t^Tr\mathcal Y_tdt-\int_t^T\mathcal Z_tdW_t.
	\end{split}
	\label{eq:exp1BSlinear}
\end{equation}
Here, $K$ is the strike price of the European call option, $T$ is the maturity, $r$ is the interest rate, and $\mathcal S_t$ is the spot price of the underlying risky asset with initial price $s_0$, appreciation $\mu$ and volatility $\sigma$. Then, terminal condition $\mathcal Y_T = (\mathcal S_T-K)^+ = \max\{S_T-K,0\}$ is the payoff of the European call opton, and the solution $\mathcal Y_t=\mathcal Y_t^{t,\mathcal S_t}$ means the price of the option at time $t\in[0,T]$ and spot price $\mathcal S_t$. For details on the derivation of \eqref{eq:exp1BSlinear}, see Section 4.5.1. (p.91) in \cite{Zhang2017}. Using the cumulative distribution function of the standard Gaussian distribution $\Psi(x)$, the solution can be evaluated as 
\begin{equation}
	\label{BSprice}
	\begin{cases}
	\mathcal Y_t = \mathcal S_t\cdot\Psi(d_1)-K\exp(-r(T-t))\cdot\Psi(d_2),\\
	\mathcal Z_t = \mathcal S_t\cdot\Psi(d_1)\sigma,
	\end{cases}
\end{equation}
where $d_1$ and $d_2$ are constants as follows
\[
	d_1 = \frac{\ln\left(\frac{\mathcal S_t}K\right)+\left(r+\frac{\sigma^2}2\right)(T-t)}{\sigma\sqrt{T-t}},\quad d_2 = \frac{\ln\left(\frac{\mathcal S_t}K\right)+\left(r-\frac{\sigma^2}2\right)(T-t)}{\sigma\sqrt{T-t}}.
\]
We choose the parameters of \eqref{eq:exp1BSlinear} as follows:
\[
\begin{tabular}{ccccc}
   $T$ & $K$ & $r$ & $\mu$ & $\sigma$ \\\hline
   $1$ & $100$ & $0.03$ & $0.03$ & $0.2$ 
\end{tabular}
\]
Despite a pretty simple case, its spatial discretization leads to a stiff system of ODEs. To this end, we discretize \eqref{eq:exp1BSlinear} on the Tavella-Randall grid $\Pi_x^{\mathrm{TR}}$ with $(x_{\mathrm{left}},x_{\mathrm{center}},x_{\mathrm{right}},N_{x,0},g_1,g_2)=(0,100,200,1000,50,50)$ and calculate solutions of the resulting system of ODEs using \verb+DP5+, an implementation of the Dormand-Prince explicit solver in Julia, for different time steps $N_t$. Consequently, we observed that it requires approximately $58005$ steps along the temporal direction to achieve the ``stable'' solution; otherwise, terribly large and rapid oscillations occur in some parts of the numerical solution. Fig. \ref{fig:1} shows the surface plots of the numerical solutions of $(t,x)\mapsto\mathcal Y_t^{t,x}$ for $57990$, $57995$, $58000$, and $58005$ temporal steps, respectively. For visibility, the surfaces displayed are the $30\times30$ arrays uniformly sampled from numerical solutions whose values heve been clipped in the $[0,250]$ range. Absolute errors of $\mathcal Y_t^{t,x}$ at $(t,x)=(0,100)$, maximum absolute errors of $\mathcal Y_t^{t,x}$ in $(t,x)\in \Pi_t^{\mathrm{Unif}}\times([80,120]\cap\Pi_x^{\mathrm{TR}})$ and runtime in seconds are reported in Table \ref{tab:DP5}. These results epitomize how stiff systems arise from the spatial discretization of parabolic PDEs and prevent explicit solvers from calculating solutions efficiently.

% how to crop margins in pdfs: pdfcrop hoge.pdf
% how to get the bounding box: rungs -o /dev/null -sDEVICE=bbox 20230814_Nt62115.pdf
% https://tex.stackexchange.com/questions/124340/latex-error-cannot-determine-size-of-graphic-in-simlinkerror-pdf-no-bounding
\begin{figure}[ht]
	\begin{tabular}{cc}
		\begin{minipage}[t]{0.45\textwidth}
			\centering
			\includegraphics[width=0.8\textwidth]{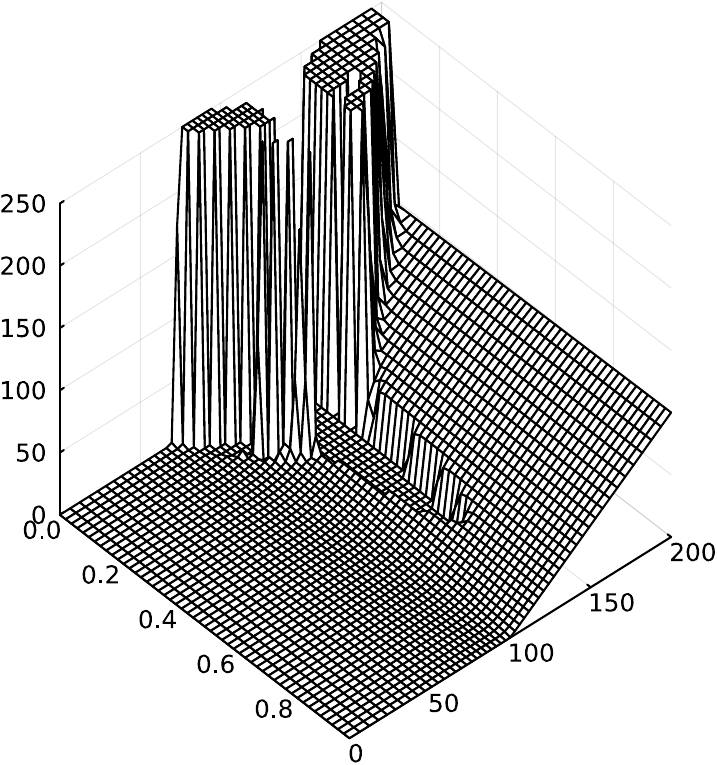}
			\subcaption{$N_t = 57990$.}
		\end{minipage}
		\begin{minipage}[t]{0.45\textwidth}
			\centering
			\includegraphics[width=0.8\textwidth]{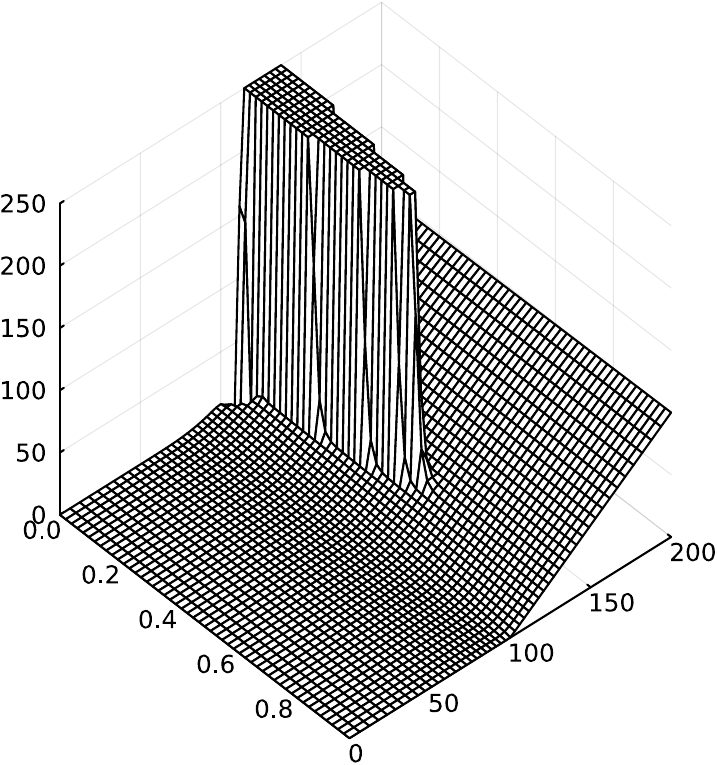}
			\subcaption{$N_t = 57995$.}
		\end{minipage}\\
		\begin{minipage}[t]{0.45\textwidth}
			\centering
			\includegraphics[width=0.8\textwidth]{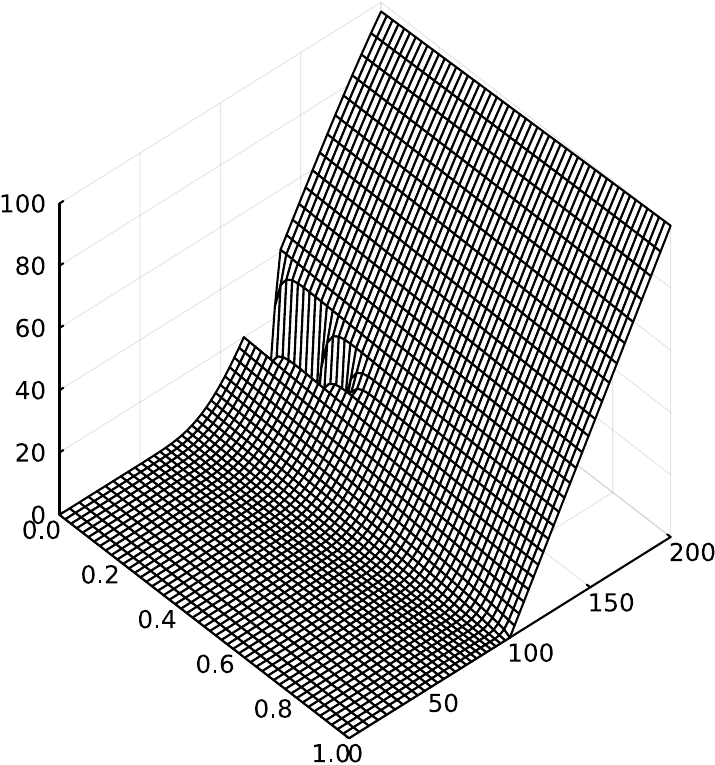}
			\subcaption{$N_t = 58000$.}
		\end{minipage}
		\begin{minipage}[t]{0.45\textwidth}
			\centering
			\includegraphics[width=0.8\textwidth]{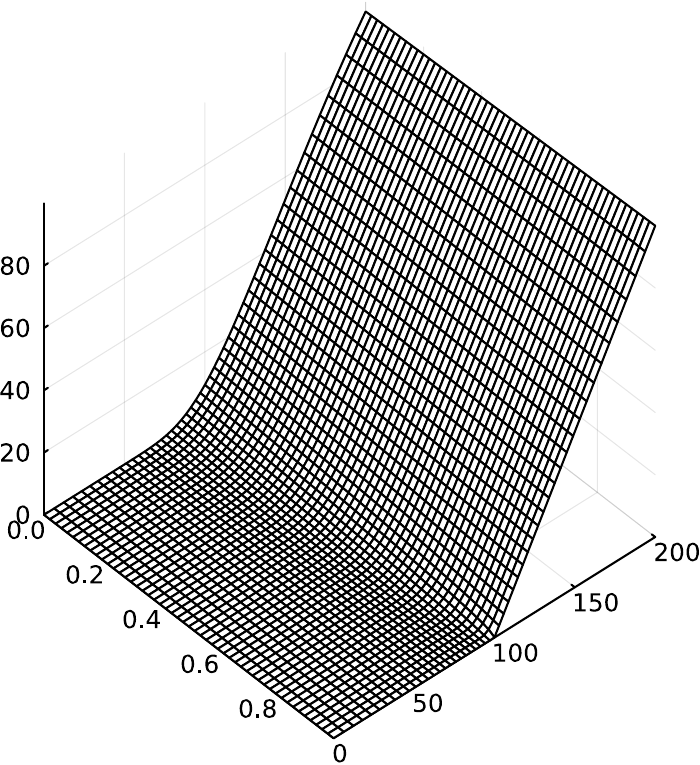}
			\subcaption{$N_t = 58005$.}
		\end{minipage}
	\end{tabular}
	\caption{Plots of numerical solutions of \eqref{sec:exp1BSlinear} using DP5 for different $N_t$. The spatial grid is $\Pi_x^{\mathrm{TR}}$ whose parameters are $x_{\mathrm{left}}=0$, $x_{\mathrm{center}}=100$, $x_{\mathrm{right}}=200$, $N_{x,0}=1000$, and $g_1=g_2=50$. For visibility, data plotted are $50\times 50$ arrays uniformly sampled and clipped into range $[0,250]$.}
	\label{fig:1}
\end{figure}

\begin{table}[htb]
	\centering
	\begin{tabular}{c|cccc}
		\hline
		& $N_t=57990$ & $N_t=57995$ & $N_t=58000$ & $N_t=58005$\\\hline
		Sup Error & $3.304\mathrm{e}{+31}$ & $1.727\mathrm{e}{+18}$ & $1.218\mathrm{e}{+5}$ & $1.326\mathrm{e}{-2}$\\
		Abs Error & $1.989\mathrm{e}{+12}$ & $1.324\mathrm{e}{+3}$ & $4.612\mathrm{e}{-4}$ & $4.612\mathrm{e}{-4}$\\
		Runtime [s] & $600.51$ & $629.94$ & $665.08$ & $668.86$\\\hline
	\end{tabular}
	\caption{Results on numerical solutions of \eqref{eq:exp1BSlinear} using DP5 for different $N_t$. Here, we spatially discretize \eqref{eq:exp1BSlinear} on $\Pi_x^{\mathrm{TR}}$ and solve the resulting system of ODEs. Here, the parameters of $\Pi_x^{\mathrm{TR}}$ are $x_{\mathrm{left}}=0$, $x_{\mathrm{center}}=100$, $x_{\mathrm{right}}=200$, $N_{x,0}=1000$, and $g_1=g_2=50$. For each $N_t$, the numerical solution is evaluated on the grid $\Pi_t^{\mathrm{Unif}}(N_t)\times\Pi_x^{\mathrm{TR}}$. The maximum absolute errors in $\Pi_t^{\mathrm{Unif}}\times([80,120]\cap\Pi_x^{\mathrm{TR}})$ are reported on the row of ``Sup Error'', the absolute errors at $(t,x)=(0,100)$ are on the row of ``Abs Error'', and the runtimes in seconds are at the bottom line.}
	\label{tab:DP5}
\end{table}

Let us solve the system of ODEs constructed from above using exponential integrators. Table \ref{table:1} reports the results of \verb+HochOst4+, a $4$-stage exponential integrator of order $4$, for different temporal steps $N_t$. It successfully calculates solutions without suffering from huge errors as appearing in \verb+DP5+, and provides accurate solutions with fewer temporal steps; for example, the solution calculated using \verb+HochOst4+ with $N_t\ge 200$ has achieved the same level of accuracy as the one using \verb+DP5+ with $N_t=58005$.
\begin{table}[htb]
    \centering
		\resizebox{\textwidth}{!}{
    \begin{tabular}{c|ccccccc}
        \hline
        & $N_t=10$ & $N_t=20$ & $N_t=50$ & $N_t=100$ & $N_t=200$ & $N_t=500$ & $N_t=1000$\\
        \hline
        Sup Error & $5.102\mathrm{e}{-1}$ & $1.374\mathrm{e}{-1}$ & $1.384\mathrm{e}{-2}$ & $1.326\mathrm{e}{-2}$ & $1.326\mathrm{e}{-2}$ & $1.326\mathrm{e}{-2}$ & $1.326\mathrm{e}{-2}$ \\
        Abs Error & $2.611\mathrm{e}{-1}$ & $2.638\mathrm{e}{-2}$ & $1.092\mathrm{e}{-3}$ & $4.693\mathrm{e}{-4}$ & $4.612\mathrm{e}{-4}$ & $4.612\mathrm{e}{-4}$ & $4.612\mathrm{e}{-4}$ \\
        Runtime[s] & $1.56$ & $2.33$ & $5.43$ & $12.21$ & $19.73$ & $49.26$ & $96.63$\\\hline
    \end{tabular}}
    \caption{Results on numerical solutions of \eqref{sec:exp1BSlinear} using HochOst4 in the same situation in Table \ref{tab:DP5}.}
    \label{table:1}
\end{table}

\subsection{European call option with higher interest rate for borrowing}
\label{sec:NonlinearBS}
Next, consider a nonlinear version of \eqref{eq:exp1BSlinear}:
\begin{equation}
	\begin{split}
	\mathcal S_t &= s_0 + \int_0^t\mu\mathcal S_sds + \int_0^t\sigma\mathcal S_sdW_s,\\
	\mathcal Y_t &= (\mathcal S_T-K)^+ -\int_t^T[r(\mathcal Y_s-\frac{\mathcal Z_s}\sigma)^+-R(\mathcal Y_s-\frac{\mathcal Z_s}\sigma)^-+\frac\mu\sigma\mathcal Z_s]ds-\int_t^T \mathcal Z_sdW_s,
	\end{split}
	\label{eq:exp2BSnonlinear}
\end{equation}
which comes from considering a more practical case that the borrowing interest rate $R$ is higher than the lending interest rate $r$. Here, $(x)^+ = \max\{x,0\}$ and $(x)^- = \min\{x,0\}$ for $x\in\mathbb R$. For details on the derivation of \eqref{eq:exp2BSnonlinear}, Section 4.5.1. (p.91) in \cite{Zhang2017}. We choose the parameters of \eqref{eq:exp2BSnonlinear} as follows:
\[
	\begin{tabular}{cccccc}
	  $T$ & $\mu$ & $\sigma$ & $R$ & $r$ & $K$ \\\hline
	  $1.0$ & $0.03$ & $0.2$ & $0.3$ & $0.01$ & $100$
  \end{tabular}
\]
We discretize \ref{eq:exp2BSnonlinear} on the Tavella-Randall grid $\Pi_x^{\mathrm{TR}}$ with $x_{\mathrm{left}}=0$, $x_{\mathrm{center}}=100$, $x_{\mathrm{right}}=200$, $N_{x,0}=1000$, and $g_1=g_2=50$, which is the same as before. Unlike the previous experiment, however, we cannot evaluate numerical errors, since we do not know the analytical solution of \eqref{eq:exp2BSnonlinear}. We only focus on numerical solutions of $\mathcal Y_t = \mathcal Y_t^{t,s}$ at $(t,s)=(0,100)$. Table \ref{tab:nonlinear1d} reports numerical solutions using different exponential integrators and the runtimes in second. For each scheme, when $N_t$ increases, the numerical solution seems to converge approximately $\mathcal Y_0^{0,100}\approx 26.3305$. We observe that multi-stage methods (ETD2RK, ETD3RK, ETD4RK, and HochOst4) converge faster than single-stage methods (Lawson-Euler and N\o rsett-Euler). As a reference, we calculate the solution $\mathcal Y_0^{0,100}$ using the LSMC method with Laguerre polynomials \eqref{eq:Laguerre} up to $p$-th order for different parameters. Here, the number of Monte Carlo simulations is $M=2^{22}$, and the pair $(N_t,p)$ of the number of temporal steps $N_t$ and the maximal order $p$ are taken from $\{(6,6),(7,7),(8,8),(9,9)\}$. The LSMC solution of $\mathcal Y_0^{0,100}$ is evaluated 50 times independently, and we compute the sample mean, the unbiased standard deviation and the total runtime. The results of the numerical solution of the LSMC method are reported in Table \ref{tab:lsmc1}. 

\begin{table}[!ht]
    \centering
    \begin{tabular}{c|c||ccccc}
    \hline
	\multicolumn{2}{c}{} & $N_t=10$ & $N_t=20$ & $N_t=50$ & $N_t=100$ & $N_t=200$ \\\hline
	\multirow{2}*{\shortstack[c]{Lawson-Euler\\\cite{Lawson1967}}} & $\mathcal Y_0^{0,s_0}$ & $25.0377$ & $26.3522$ & $26.3545$ & $26.3437$ & $26.3372$\\
	& Runtime [s] & $5.26$ & $0.61$ & $1.52$ & $2.81$ & $5.40$\\ \hline
	\multirow{2}*{\shortstack[c]{N\o rsett-Euler\\\cite{Cox2002}}} & $\mathcal Y_0^{0,s_0}$ & $26.0777$ & $26.4527$ & $26.3881$ & $26.3597$ & $26.3452$ \\
	& Runtime [s] & $1.25$ & $0.53$ & $1.30$ & $2.51$ & $5.23$\\ \hline
	\multirow{2}*{\shortstack[c]{ETD2RK\\\cite{Cox2002}}} & $\mathcal Y_0^{0,s_0}$ & $25.8658$ & $26.3182$ & $26.3310$ & $26.3307$ & $26.3306$\\
	& Runtime [s] & $1.51$ & $1.27$ & $2.89$ & $5.85$ & $10.78$ \\ \hline
	\multirow{2}*{\shortstack[c]{ETD3RK\\\cite{Cox2002}}} & $\mathcal Y_0^{0,s_0}$ & $26.1212$ & $26.3137$ & $26.3302 $ & $26.3305 $ & $26.3305$\\
	& Runtime [s] & $1.78$ & $1.71$ & $3.81$ & $8.02$ & $16.68$\\ \hline
	\multirow{2}*{\shortstack[c]{ETD4RK\\\cite{Cox2002}}} & $\mathcal Y_0^{0,s_0}$ & $26.0947$ & $26.3136$ & $26.3302$ & $26.3305$ & $26.3305$ \\
	& Runtime [s] & $3.10$ & $2.72$ & $6.47$ & $14.35$ & $27.32$ \\ \hline
	\multirow{2}*{\shortstack[c]{HochOst4\\\cite{Hochbruck2005}}} & $\mathcal Y_0^{0,s_0}$ & $26.1080$ & $26.3136$ & $26.3302$ & $26.3305$ & $26.3305$ \\
	& Runtime [s] & $2.28$ & $2.79$ & $6.84$ & $14.27$ & $29.64$ \\ \hline
    \end{tabular}
    \caption{
    Numerical solutions $\mathcal Y_0^{0,100}$ and its runtime in seconds using different exponential integrators. Here, we spatially discretize \eqref{eq:exp2BSnonlinear} on $\Pi_x^{\mathrm{TR}}$ and solve the resulting system of ODEs. Here, the parameters of $\Pi_x^{\mathrm{TR}}$ are $x_{\mathrm{left}}=0$, $x_{\mathrm{center}}=100$, $x_{\mathrm{right}}=200$, $N_{x,0}=1000$, and $g_1=g_2=50$.}
    \label{tab:nonlinear1d}
\end{table}

\begin{table}[!ht]
    \centering
    \begin{tabular}{c|c||cccc}
    \hline
    		\multicolumn{1}{c}{$M$} & \multicolumn{1}{c}{} & $(N_t,p)=(6,6)$ & $(N_t,p)=(7,7)$ & $(N_t,p)=(8,8)$ & $(N_t,p)=(9,9)$\\ \hline
        \multirow{3}*{$4194304$}& mean & $26.1462$ & $26.2239$ & $26.3001$ & $26.3404$ \\
        & std & $0.0457$ & $0.0386$ & $0.0611$ & $0.0580$ \\
        & Runtime [s] & $1436.79$ & $1706.81$ & $2021.94$ & $2340.43$\\ \hline
    \end{tabular}
    \caption{Results on numerical solutions $\mathcal Y_0^{0,100}$ of \eqref{eq:exp2BSnonlinear} using the LSMC methods with Laguerre polynomials up to $p$-th order for $N_t$ time steps. Here, $M=4194304$ is the number of samples for the Monte-Carlo approximation. We run the LSMC algorithm $50$ times independently and collect each $\mathcal Y_0^{0,100}$. The row of ``mean'' and ``std'' reports their sample means and sample (unbiased) standard deviations, respectively. The bottom of the table reports the total CPU times required for the $50$ experments.}
    \label{tab:lsmc1}
\end{table}

\subsection{European options under stochastic local volatility models}
In this section, we consider BSDEs arising from the valuation of European options in stochastic local volatility (SLV) models, a class of Markov BSDEs with 2-dimensional state processes. A general form of stochastic volatility models is given in \cite{Cui2018} as 
\begin{equation}
	\label{eq:SLV}
	\begin{cases}
		\displaystyle \mathcal S_t = S_0 + \int_0^t\omega(\mathcal S_s, v_s)ds + \int_0^tm(v_s)\Gamma(\mathcal S_s)dW_s^{(1)},\\
		\displaystyle v_t = v_0 + \int_0^t\mu(v_s)ds + \int_0^t\sigma(v_s)dW_s^{(2)},
	\end{cases}
\end{equation}
where $\langle W^{(1)},W^{(2)}\rangle_t = \rho t$ with $\rho\in(-1,1)$. Let $\mathbf L=\begin{pmatrix}1&0\\\rho&\sqrt{1-\rho^2}\end{pmatrix}$ be the lower triangular matrix constructed from the Cholesky decomposition $\mathbf C = \mathbf L\mathbf L^\ast$ of $\mathbf C = \begin{pmatrix}1&\rho\\\rho&1\end{pmatrix}$. Since $\mathbf W = \mathbf L^{-1}\begin{pmatrix}W^{(1)}\\W^{(2)}\end{pmatrix}$ is a 2-dimensional standard Brownian motion, we can reformulate \eqref{eq:SLV} as 
\[
	\mathcal X_t = \mathcal X_0 + \int_0^t\boldsymbol\mu(\mathcal X_s)ds + \int_0^t\boldsymbol\sigma(\mathcal X_s)d\mathbf W_s\quad \text{for}\quad t\in[0,T],
\]
where 
\[
	\mathcal X_t = \begin{pmatrix}\mathcal S_t\\v_t\end{pmatrix},\quad \boldsymbol\mu(x) = \begin{pmatrix}\omega(x_1,x_2)\\\mu(x_2)\end{pmatrix},\quad \boldsymbol\sigma(x) = \begin{pmatrix}m(x_2)\Gamma(x_1) & 0\\0 &\sigma(x_2)\end{pmatrix}\mathbf L.
\]

Similarly to the arguments in Section 4.5.1. (p.91) in \cite{Zhang2017}, we formulate Markov BSDEs describing the price of European options. Consider the self-financing portfolio $\boldsymbol\Delta_t=\begin{pmatrix}\Delta_t^1\\\Delta_t^2\end{pmatrix}$ consisting of $\Delta_t^1$ assets of $\mathcal S_t$ and $\Delta_t^2$ assets of $v_t$ and bonds with borrowing rate $R$ and lending rate $r$. Let $\mathcal V_t$ be the wealth dynamics of $\boldsymbol\Delta_t$. The self-financing condition reads 
\[
	d\mathcal V_t = [r(\mathcal V_t-\boldsymbol\Delta_t^\ast\mathcal X_t)^+-R(\mathcal V_t-\boldsymbol\Delta_t^\ast\mathcal X_t)^-]dt+\boldsymbol\Delta_t^\ast d\mathcal X_t.
\]
Consider that $\boldsymbol\Delta_t$ hedges the European option with payoff $g(\mathcal S_T)$. Denoting $\mathcal Y_t \coloneqq\mathcal V_t$ and $\mathcal Z_t \coloneqq \boldsymbol\Delta_t^\ast\boldsymbol\sigma(\mathcal X_t)$, we obtain
\begin{equation}
	\label{eq:SLVBSDE}
	\mathcal Y_t = g(\mathcal S_T)-\int_t^Tf(s,\mathcal X_s,\mathcal Y_s,\mathcal Z_s)ds -\int_t^T\mathcal Z_s^\ast d\mathbf W_s,
\end{equation}
where the driver is 
\[
	f(t,x,y,z)=r(y-z^\ast\boldsymbol\sigma(x)^{-1}x)^+-R(y-z^\ast\boldsymbol\sigma(x)^{-1}x)^-+z^\ast\boldsymbol\sigma(x)^{-1}\boldsymbol\mu(x).
\]
Then, $\mathcal Y_t^{t,\mathcal X_t}$ is the price of the European option.

Before moving on to the specific results, we note what follows: (1) Unlike the cases considered in the previous two experiments, we need to discretize the two-dimensional process $\mathcal X_t = (\mathcal S_t,v_t)$. To this end, we construct the Kronecker product of two grids as described in Section \ref{sec:d-MOL}. Specifically, $\mathcal S_t$ is discretized on the Tavella-Randall grid $\Pi_x^{\mathrm{TR}}$ and $v_t$ is on the uniform grid $\Pi_x^{\mathrm{Unif}}$. (2) BSDE \eqref{eq:SLVBSDE} contains the evaluation of $\boldsymbol\sigma(x)^{-1}$, and we should set each grid carefully; if $\Gamma(x)$, $m(x)$ or $\sigma(x)$ take zero at $x=0$, grid \eqref{eq:grid2} should be chosen rather than grid \eqref{eq:grid1}. (3) We calculate numerical solutions using the LSMC method for reference. Since the Euler-Maruyama discretization paths can take negative values, coefficient functions defined only on $(0,\infty)$ (e.g. the square root) may fail to evaluate. To avoid it, the absolute value is taken under such coefficient functions.

\subsubsection{European put opton under the Heston-SABR model with higher interest rate for borrowing}
\label{sec:HestonSABR}
The Heston-SABR SLV model \cite{VanderStoep2014} takes the following form.
\begin{equation}
	\begin{split}
	\mathcal S_t &= \mathcal S_0 + \int_0^tb\cdot\mathcal S_sds + \int_0^t\sqrt{v_s}\mathcal S_s^\beta dW_s^{(1)}\\
	v_t &= v_0 + \int_0^t\eta(\theta-v_s)ds + \int_0^t\alpha\sqrt{v_s}dW_s^{(2)}
	\end{split}
	\label{eq:HestonSABR}
\end{equation}
We apply \eqref{sec:HestonSABR} to \eqref{eq:SLV}, and consider BSDE \eqref{eq:SLVBSDE} with $g(s) = (K-s)^+$ a payoff function of a European put option. We choose the parameters of \eqref{sec:HestonSABR} as follows:
\[
\begin{tabular}{cccccccccc}
  $T$ & $\beta$ & $\eta$ & $\theta$ & $\alpha$ & $\rho$ & $b$ & $K$ & $R$ & $r$ \\\hline
  $1.0$ & $0.7$ & $4.0$ & $0.035$ & $0.15$ & $-0.75$ & $0.01$ & $100$ & $0.07$ & $0.01$
\end{tabular}
\]

In this experiment, we are interested in numerical solutions of $\mathcal Y_0 = \mathcal Y_0^{0, (s_0,v_0)}$ at $(s_0,v_0) = (100, 0.4)$. As noted before, since the driver of \eqref{eq:SLVBSDE} contains the evaluation of inverses of $\sqrt{v_t}\mathcal S_t^\beta$ and $\alpha\sqrt{v_s}$, we need to design the spatial grid to contain only points in the first quadrant. To this end, we set the spatial grid onto $[\Delta_1,2s_0-\Delta_1]\times[\Delta_2,2v_0-\Delta_2]$, where $\Delta_1 = s_0/N_0^{(1)}$, $\Delta_2 = v_0/N_0^{(2)}$. We take $(N_0^{(1)},N_0^{(2)})=(100, 15)$, and the resulting spatial grid $\Pi_x^{\mathrm{TR}}\otimes\Pi_x^{\mathrm{Unif}}$ contains $(2\times100+1)\times(2\times15+1) = 6231$ points in the reculangular domain $[100/100,200-100/100]\times[0.4/15,0.8-0.4/15](\approx[1,199]\times[0.0267,0.7733]$.) For $g_1$ and $g_2$, we take $g_1=g_2=1$.

Numerical solutions using exponential integrators are reported in Table \ref{tab:MOLHestonSABR}. As $N_t$ increases, the solutions seem to converge approximately $\mathcal Y_0^{0,(100,0.4)}\approx 5.6394$. As before, faster convergence of multi-stage methods than 1-stage methods has been confirmed. Numerical solutions using the LSMC method is given in Table \ref{tab:lsmcHestonSABR}. 

\begin{table}[!ht]
    \centering
    \begin{tabular}{c|c||ccccc}
    \hline
    		\multicolumn{2}{c}{}& $N_t=10$ & $N_t=20$ & $N_t=50$ & $N_t=100$ & $N_t=200$\\\hline
        \multirow{2}*{\shortstack[c]{Lawson-Euler\\\cite{Lawson1967}}} & $\mathcal Y_0^{0,(s_0,v_0)}$ & $4.7985$ & $5.5605$ & $5.6015$ & $5.6211$ & $5.6302$ \\
        & Runtime [s] & $22.79$ & $29.70$ & $66.15$ & $129.13$ & $255.58$ \\ \hline
        \multirow{2}*{\shortstack[c]{N\o rsett-Euler\\\cite{Cox2002}}} & $\mathcal Y_0^{0,(s_0,v_0)}$ & $4.9076$ & $5.5662$ & $5.6234$ & $5.6326$ & $5.6360$ \\
        & Runtime [s] & $14.71$ & $26.26$ & $61.46$ & $116.11$ & $231.01$ \\ \hline
        \multirow{2}*{\shortstack[c]{ETD2RK\\\cite{Cox2002}}} & $\mathcal Y_0^{0,(s_0,v_0)}$ & $4.9348$ & $5.6006$ & $5.6368$ & $5.6393$ & $5.6394$ \\
        & Runtime [s] & $26.15$ & $43.70$ & $112.08$ & $221.11$ & $443.11$ \\ \hline
        \multirow{2}*{\shortstack[c]{ETD3RK\\\cite{Cox2002}}} & $\mathcal Y_0^{0,(s_0,v_0)}$ & $4.9329$ & $5.6014$ & $5.6369$ & $5.6393$ & $5.6394$ \\
        & Runtime [s] & $33.04$ & $66.37$ & $159.77$ & $341.57$ & $724.57$ \\ \hline
        \multirow{2}*{\shortstack[c]{ETD4RK\\\cite{Cox2002}}} & $\mathcal Y_0^{0,(s_0,v_0)}$ & $4.9521$ & $5.6015$ & $5.6369$ & $5.6393$ & $5.6394$ \\
        & Runtime [s] & $63.90$ & $121.81$ & $293.59$ & $566.32$ & $1150.64$ \\ \hline
        \multirow{2}*{\shortstack[c]{HochOst4\\\cite{Hochbruck2005}}} & $\mathcal Y_0^{0,(s_0,v_0)}$ & $4.9328$ & $5.6013$ & $5.6369$ & $5.6393$ & $5.6394$ \\
        & Runtime [s] & $60.23$ & $114.64$ & $285.99$ & $551.47$ & $1117.35$ \\ \hline
    \end{tabular}
    \caption{
    Numerical solutions $\mathcal Y_0^{0,(100,0.4)}$ and its runtime in seconds using different exponential integrators. Here, we spatially discretize \eqref{eq:SLVBSDE} on $\Pi_x^{\mathrm{TR}}\otimes\Pi_x^{\mathrm{Unif}}$ and solve the resulting system of ODEs. Here, the parameters of $\Pi_x^{\mathrm{TR}}$ are $(x_{\mathrm{left}},x_{\mathrm{center}},x_{\mathrm{right}},N_{x,0},g_1,g_2)=(1,100,199,100,1,1)$, and of $\Pi_x^{\mathrm{Unif}}$ are $(x_{\mathrm{left}},x_{\mathrm{center}},x_{\mathrm{right}},N_{x,0}) \approx (0.0267, 0.4, 0.7733, 15)$.}
    \label{tab:MOLHestonSABR}
\end{table}

\begin{table}[!ht]
    \centering
    \begin{tabular}{c|c||cccc}
    \hline
        \multicolumn{1}{c}{$M$} &  & $(N,p)=(6,6)$ & $(N,p)=(7,7)$ & $(N,p)=(8,8)$ & $(N,p)=(9,9)$ \\ \hline
        \multirow{3}*{$1048576$} & mean & $5.5956$ & $5.6132$ & $5.6247$ & $5.6249$ \\
         & std & $0.0153$ & $0.0153$ & $0.0179$ & $0.0192$ \\
         & Runtime [s] & $3250.84$ & $4737.51$ & $6788.49$ & $9863.74$ \\\hline
    \end{tabular}
    \caption{
    Results on numerical solutions $\mathcal Y_0^{0,(100,0.4)}$ of \eqref{eq:SLVBSDE} under the Heston-SABR model \eqref{eq:HestonSABR} using the LSMC methods with Laguerre polynomials up to $p$-th order for $N_t$ time steps. Here, $M=2^20=1048576$ is the number of samples for the Monte-Carlo approximation. We run the LSMC algorithm $50$ times independently and collect each $\mathcal Y_0^{0,(100,0.4)}$. The row of ``mean'' and ``std'' reports their sample means and sample (unbiased) standard deviations, respectively. The bottom of the table reports the total CPU times required for the $50$ experments.}
    \label{tab:lsmcHestonSABR}
\end{table}

\subsubsection{Calls combination with different interest rates under the Hyp Hyp SLV model}
\label{sec:HypHyp}
Consider the Hyp Hyp SLV model \cite{Jaeckel2007} defined as:
\begin{equation}
\begin{split}
	\mathcal S_t &= \mathcal S_0 + \int_0^tb\cdot\mathcal S_sds + \int_0^t\sigma_0\cdot F(\mathcal S_s)G(v_s)dW_s^{(1)},\\
	v_t &= v_0 - \int_0^t\kappa\cdot v_sds + \int_0^t\alpha\sqrt{2\kappa}dW_s^{(2)},
\end{split}
\label{eq:HypHyp}
\end{equation}
where 
\begin{align*}
	F(x) &= \left[(1-\beta+\beta^2)\cdot x + (\beta-1)\cdot(\sqrt{x^2+\beta^2(1-x)^2}-\beta)\right]/\beta,\\
	G(v) &= v +\sqrt{v^2+1}.
\end{align*}
We apply \eqref{eq:HypHyp} to \eqref{eq:SLV}, and consider BSDE \eqref{eq:SLVBSDE} with 
\[
	g(s) = (s-95)^+-2(s-105)^+,
\]
which is a payoff function of a combination of two European call options. We choose the parameters of \eqref{eq:HypHyp} as follows:
\[
\begin{tabular}{cccccccccc}
  $T$ & $\beta$ & $\kappa$ & $\sigma_0$ & $\alpha$ & $\rho$ & $b$ & $K$ & $R$ & $r$ \\\hline
  $1.0$ & $0.25$ & $0.5$ & $0.25$ & $0.3$ & $0.8$ & $0.04$ & $100$ & $0.06$ & $0.006$
\end{tabular}
\]
We discretize $\mathcal X_t$ in the same way of Section \ref{sec:HestonSABR} except the discretization of $\mathcal S_t$; because the payoff function $g$ is non-smooth at $s=95$ and $s=105$, we should employ the concatenation of two Tavella-Randall grids $\Pi_x^{\mathrm{TR}}$ with $(x_{\mathrm{left}},x_{\mathrm{center}},x_{\mathrm{right}})=(1,95,100)$ and $(x_{\mathrm{left}},x_{\mathrm{center}},x_{\mathrm{right}})=(100,105,199)$; the other parameters are commonly taken as $N_{x,0} = N_0^{(1)}/2=50$ and $g_1=g_2=1$.

Results on numerical solutions using exponential integrators are shown in Table \ref{tab:MOLHypHyp}. Numerical solutions of $\mathcal Y_0^{0,(s_0,v_0)}$ at $(s_0,v_0)=(95,0.4),(100,0.4),(105,0.4)$ have been reported; as $N_t$ increases, they seem to converge towards
\[
	\mathcal Y_0^{0,(95,0.4)}\approx 4.4061,\quad \mathcal Y_0^{0,(100,0.4)}\approx5.8218,\quad \text{and}\quad \mathcal Y_0^{0,(105,0.4)}\approx5.6394.
\]
Numerical solutions using the LSMC method are presented in Table \ref{tab:lsmcHypHyp}.

\begin{table}[!ht]
    \centering
    \begin{tabular}{c|c||ccccc}
    \hline
    		\multicolumn{2}{c||}{}& $N_t=10$ & $N_t=20$ & $N_t=50$ & $N_t=100$ & $N_t=200$\\\hline
        \multirow{4}*{\shortstack[c]{Lawson-Euler\\\cite{Lawson1967}}} & $\mathcal Y_0^{0,(95,0.4)}$ & $4.5589$ & $4.4829$ & $4.4368$ & $4.4214$ & $4.4137$ \\
        & $\mathcal Y_0^{0,(100,0.4)}$ & $6.0971$ & $5.9595$ & $5.8768$ & $5.8492$ & $5.8355$ \\
        & $\mathcal Y_0^{0,(105,0.4)}$ & $5.9113$ & $5.7756$ & $5.6938$ & $5.6665$ & $5.6529$ \\
        & Runtime[s] & $15.32$ & $24.69$ & $57.58$ & $113.83$ & $231.95$ \\ \hline
        \multirow{4}*{\shortstack[c]{N\o rsett-Euler\\\cite{Cox2002}}}& $\mathcal Y_0^{0,(95,0.4)}$ & $4.5561$ & $4.4787$ & $4.4343$ & $4.4200$ & $4.4130$ \\
        & $\mathcal Y_0^{0,(100,0.4)}$ & $5.9936$ & $5.9035$ & $5.8531$ & $5.8371$ & $5.8293$ \\
        & $\mathcal Y_0^{0,(105,0.4)}$ & $5.7821$ & $5.707 $& $5.6652 $& $5.6520$ & $5.6456$\\
        & Runtime[s] & $13.52$ & $22.99$ & $52.94$ & $108.18$ & $198.73$ \\ \hline
        \multirow{4}*{\shortstack[c]{ETD2RK\\\cite{Cox2002}}}& $\mathcal Y_0^{0,(95,0.4)}$ & $4.4217$ & $4.4115$ & $4.4074$ & $4.4065$ & $4.4063$ \\
        & $\mathcal Y_0^{0,(100,0.4)}$ & $5.8444$ & $5.8294$ & $5.8236$ & $5.8224$ & $5.8220$ \\
        & $\mathcal Y_0^{0,(105,0.4)}$ & $5.6631$ & $5.6473$ & $5.6412$ & $5.6400$ & $5.6396$ \\
        & Runtime[s] & $22.80$ & $41.33$ & $99.87$ & $206.08$ & $419.02$ \\ \hline
        \multirow{4}*{\shortstack[c]{ETD3RK\\\cite{Cox2002}}}& $\mathcal Y_0^{0,(95,0.4)}$ & $4.4082$ & $4.4068$ & $4.4063$ & $4.4062$ & $4.4061$ \\
        & $\mathcal Y_0^{0,(100,0.4)}$ & $5.8248$ & $5.8227$ & $5.8220$ & $5.8218$ & $5.8218$ \\
        & $\mathcal Y_0^{0,(105,0.4)}$ & $5.6424$ & $5.6403$ & $5.6396$ & $5.6395$ & $5.6394$ \\
        & Runtime[s] & $36.45$ & $64.68$ & $157.76$ & $327.25$ & $631.46$ \\ \hline
        \multirow{4}*{\shortstack[c]{ETD4RK\\\cite{Cox2002}}}& $\mathcal Y_0^{0,(95,0.4)}$ & $4.4080$ & $4.4068$ & $4.4063$ & $4.4062$ & $4.4061$ \\
        & $\mathcal Y_0^{0,(100,0.4)}$ & $5.8243$ & $5.8226$ & $5.8220$ & $5.8218$ & $5.8218$ \\
        & $\mathcal Y_0^{0,(105,0.4)}$ & $5.6421$ & $5.6403$ & $5.6396$ & $5.6395$ & $5.6394$ \\
        & Runtime[s] & $59.71$ & $113.26$ & $257.03$ & $499.43$ & $1060.09$ \\ \hline
        \multirow{4}*{\shortstack[c]{HochOst4\\\cite{Hochbruck2005}}}& $\mathcal Y_0^{0,(95,0.4)}$ & $4.4080$ & $4.4068$ & $4.4063$ & $4.4062$ & $4.4061$ \\
        & $\mathcal Y_0^{0,(100,0.4)}$ & $5.8244$ & $5.8226$ & $5.8220$ & $5.8218$ & $5.8218$ \\
        & $\mathcal Y_0^{0,(105,0.4)}$ & $5.6421$ & $5.6403$ & $5.6396$ & $5.6395$ & $5.6394$ \\
        & Runtime[s] & $59.85$ & $113.38$ & $275.02$ & $544.55$ & $1043.06$ \\ \hline
    \end{tabular}
    \caption{Numerical solutions $\mathcal Y_0^{0,(s_0,v_0)}$ at $(s_0,v_0)=(95,0.4),(100,0.4),(105,0.4)$ of \eqref{eq:HypHyp} and its runtime in seconds using different exponential integrators. Here, we spatially discretize \eqref{eq:SLVBSDE} on $\Pi_x^{\mathrm{TR}}\otimes\Pi_x^{\mathrm{Unif}}$, and solve the resulting system of ODEs. We employ concatenation of two Tavella-Randall grids $\Pi_x^{\mathrm{TR}}$ with $(x_{\mathrm{left}},x_{\mathrm{center}},x_{\mathrm{right}})=(1,95,100)$ and $(x_{\mathrm{left}},x_{\mathrm{center}},x_{\mathrm{right}})=(100,105,199)$ in the $\mathcal S_t$-direction. In both grids, we take $N_{x,0}=50$ and $g_1=g_2=1$.}
    \label{tab:MOLHypHyp}
\end{table}

\begin{table}[!ht]
    \centering
    \begin{tabular}{c|c|c||cccc}
    \hline
        \multicolumn{1}{c}{$M$} & \multicolumn{1}{c}{$(\mathcal S_0,v_0)$} & \multicolumn{1}{c}{} & $(N,p)=(6,6)$ & $(N,p)=(7,7)$ & $(N,p)=(8,8)$ & $(N,p)=(9,9)$\\ \hline
        \multirow{7}*{$1048576$} & \multirow{2}*{$(95,0.4)$} & mean & $4.4666$  & $4.4503$  & $4.4448$  & $4.4369$  \\
        & & std & $0.0138$  & $0.0140$  & $0.0145$  & $0.0178$  \\ \cmidrule{2-7}
        & \multirow{2}*{$(100,0.4)$} & mean & $5.7952$ & $5.7897$ & $5.7851$ & $5.7824$  \\
        & & std & $0.0200$ & $0.0207$ & $0.0196$ & $0.0253$  \\ \cmidrule{2-7}
        & \multirow{2}*{$(105,0.4)$} & mean & $5.6619$ & $5.6384$ & $5.6293$ & $5.6384$  \\
        & & std & $0.0250$ & $0.0299$ & $0.0350$ & $0.0331$  \\ \hline
        \multicolumn{3}{c||}{Avg Runtime [s]} & $3614.21$ & $5123.17$ & $7363.07$ & $8898.37$ \\ \hline
    \end{tabular}
    \caption{
    Results on numerical solutions $\mathcal Y_0^{0,(100,0.4)}$ of \eqref{eq:SLVBSDE} under the HypHyp SLV model \eqref{eq:HypHyp} at $(s_0,v_0)=(95,0.4),(100,0.4),(105,0.4)$ using the LSMC methods with Laguerre polynomials up to $p$-th order for $N_t$ time steps. Here, $M=2^{20}=1048576$ is the number of samples for the Monte-Carlo approximation. For each $(s_0,v_0)$, we run the LSMC algorithm $50$ times independently and collect each $\mathcal Y_0^{0,(s_0,v_0)}$. The row of ``mean'' and ``std'' reports their sample means and sample (unbiased) standard deviations, respectively. The bottom of the table reports the average of the total CPU times for the three cases.}
    \label{tab:lsmcHypHyp}
\end{table}

\subsection{European call option under the SABR model}
In this section, we consider two ways of spatial discretizations: (1) the discretization on the Kronecker product of grids (we shall call it a ``full grid''), the same as considered in all the previous experiments. (2) the multilevel discretization on a sparse grid presented in Section \ref{sec:SGcomb}. We calculate numerical solutions for each case using exponential integrators and then compare their performances.

The stochastic-alpha-beta-rho (SABR) model is a SLV model being commonly used, and designed to model price dynamics of a forward contract. Consider the forward price process $F_t = \mathcal S_te^{r(T-t)}$ for some asset price $\mathcal S_t$ and interest rate $r$. The SABR model assumes that $F_t$ satisfies the following SLV model
\begin{equation}
	\label{eq:SABR}
	F_t = F_0 + \int_0^tv_sF_s^\beta dW_s^{(1)},\quad v_t = v_0 + \int_0^t\alpha v_sdW_s^{(2)}.
\end{equation}
We consider pricing of the European call option under the SABR model. To this end, we additionally suppose that the volatility process $v_t$ also satisfies $v_t = v_t^0e^{r(T-t)}$ for some underlying asset $v_t^0$. Taking the hedge portofolio of two risky assets $\mathcal S_t$ and $v_t^0$ and a bond with riskless rate $r$ into account, the corresponding Makov BSDE \eqref{eq:SLVBSDE} becomes
\begin{equation}
	\label{eq:SABRBSDE}
	\mathcal Y_t = (F_T-K)^+-\int_t^Tr\mathcal Y_sds -\int_t^T\mathcal Z_sd\mathbf W_s
\end{equation}
with the state process \eqref{eq:SABR}. It has an approximation formula
\begin{equation}
	\label{eq:Hagan}
	\mathcal Y_t \approx \exp(-r(T-t))[F_t\cdot\Psi(d_1)-K\cdot\Psi(d_2)],
\end{equation}
where
\[
	d_1 = \frac{\ln\left(\frac{F_t}K\right)+\left(r+\frac{\sigma_B^2}2\right)(T-t)}{\sigma_B\sqrt{T-t}},\quad d_2 = \frac{\ln\left(\frac{F_t}K\right)+\left(r-\frac{\sigma_B^2}2\right)(T-t)}{\sigma_B\sqrt{T-t}},
\]
and $\sigma_B$ is the approximated implied volatility by Hagan et al. \cite{Hagan2002} defined as
\[
	\sigma_B = \frac{v_t\left\{1+\left[\frac{(1-\beta)^2}{24}\frac{v_t^2}{(F_tK)^{1-\beta}}+\frac14\frac{\rho\beta\alpha v_t}{(F_tK)^{(1-\beta)/2}}+\frac{2-3\rho^2}{24}\right](T-t)\right\}}{(F_tK)^{(1-\beta)/2}\left\{1+\frac{(1-\beta)^2}{24}\log^2(F_t/K)+\frac{(1-\beta)^4}{1920}\log^4(F_t/K)\right\}}\frac{z}{\chi(z)},
\]
where $z$ and $\chi(z)$ are
\[
	z = \frac\alpha{v_t}(F_tK)^{\frac{1-\beta^2}2}\log\left(\frac{F_t}K\right),\quad \chi(z) = \log\left(\frac{\sqrt{1-2\rho z+z^2}+z-\rho}{1-\rho}\right).
\]
We regard the solutions calculated using \eqref{eq:Hagan} as the exact solutions. We choose the parameters as follows:
\[
\begin{tabular}{cccccc}
  $T$ & $\alpha$ & $\beta$ & $\rho$ & $r$ & $K$ \\\hline
  $1.0$ & $0.4$ & $0.9$ & $0.3$ & $0.05$ & $100$
\end{tabular}
\]

\subsubsection{Spatial discretization on a ``full grid''}
\label{sec:SABRFG}
First, we discretize \eqref{eq:SABRBSDE} on the Kronecker product $\Pi_x^{\mathrm{TR}}\otimes\Pi_x^{\mathrm{Unif}}$ of grids similarly to the previous experiments. Precisely, $\mathcal S_t$ is discretized on the Tavella-Randall grid $\Pi_x^{\mathrm{TR}}$ with $(x_{\mathrm{left}},x_{\mathrm{center}},x_{\mathrm{right}},N_{x,0},g_1,g_2)=(0,100,200,100,5,5)$, and $v_t$ is discretized on the uniform grid $\Pi_x^{\mathrm{Unif}}$ with $(x_{\mathrm{left}},x_{\mathrm{center}},x_{\mathrm{right}},N_{x,0})=(0, 0.4, 0.8, 15)$. The size of the resulting grid is $6231$.
\begin{table}[ht]
	\centering
	\begin{tabular}{c|c||ccccc}
		\hline
		\multicolumn{2}{c}{}& $N_t=10$ & $N_t=20$ & $N_t=50$ & $N_t=100$ & $N_t=200$\\\hline
		\multirow{3}*{\shortstack[c]{Lawson-Euler\\\cite{Lawson1967}}} & Sup Error & $4.500\mathrm{e}{-1}$ & $9.640\mathrm{e}{-2}$ & $3.264\mathrm{e}{-2}$ & $3.245\mathrm{e}{-2}$ & $3.242\mathrm{e}{-2}$\\
		& Abs Error & $1.157\mathrm{e}{-1}$ & $1.538\mathrm{e}{-2}$ & $2.136\mathrm{e}{-3}$ & $1.853\mathrm{e}{-3}$ & $1.791\mathrm{e}{-3}$\\
		& Runtime[s] & $20.21 $ & $24.97 $ & $58.12 $ & $120.34 $ & $238.51 $\\\hline
		\multirow{3}*{\shortstack[c]{N\o rsett-Euler\\\cite{Cox2002}}} & Sup Error & $4.354\mathrm{e}{-1}$ & $9.113\mathrm{e}{-2}$ & $3.053\mathrm{e}{-2}$ & $3.140\mathrm{e}{-2}$ & $3.190\mathrm{e}{-2}$\\
		 & Abs Error & $1.003\mathrm{e}{-1}$ & $8.260\mathrm{e}{-3}$ & $3.903\mathrm{e}{-4}$ & $6.003\mathrm{e}{-4}$ & $1.166\mathrm{e}{-3}$\\
		& Runtime[s] & $14.79 $ & $24.19 $ & $56.82 $ & $106.12 $ & $212.07 $\\\hline
		\multirow{3}*{\shortstack[c]{ETD2RK\\\cite{Cox2002}}} & Sup Error & $4.362\mathrm{e}{-1}$ & $9.127\mathrm{e}{-2}$ & $3.253\mathrm{e}{-2}$ & $3.240\mathrm{e}{-2}$ & $3.240\mathrm{e}{-2}$\\
		& Abs Error & $1.117\mathrm{e}{-1}$ & $1.394\mathrm{e}{-2}$ & $1.872\mathrm{e}{-3}$ & $1.730\mathrm{e}{-3}$ & $1.730\mathrm{e}{-3}$\\
		& Runtime[s] & $25.10 $ & $44.03 $ & $106.28 $ & $221.38 $ & $418.79 $\\\hline
		\multirow{3}*{\shortstack[c]{ETD3RK\\\cite{Cox2002}}} & Sup Error & $4.362\mathrm{e}{-1}$ & $9.123\mathrm{e}{-2}$ & $3.253\mathrm{e}{-2}$ & $3.240\mathrm{e}{-2}$ & $3.240\mathrm{e}{-2}$\\
		 & Abs Error & $1.119\mathrm{e}{-1}$ & $1.397\mathrm{e}{-2}$ & $1.876\mathrm{e}{-3}$ & $1.731\mathrm{e}{-3}$ & $1.730\mathrm{e}{-3}$\\
		& Runtime[s] & $34.49 $ & $69.38 $ & $166.38 $ & $359.62 $ & $670.06 $\\\hline
		\multirow{3}*{\shortstack[c]{ETD4RK\\\cite{Cox2002}}} & Sup Error & $4.364\mathrm{e}{-1}$ & $9.126\mathrm{e}{-2}$ & $3.253\mathrm{e}{-2}$ & $3.240\mathrm{e}{-2}$ & $3.240\mathrm{e}{-2}$\\
		 & Abs Error & $1.120\mathrm{e}{-1}$ & $1.397\mathrm{e}{-2}$ & $1.876\mathrm{e}{-3}$ & $1.731\mathrm{e}{-3}$ & $1.730\mathrm{e}{-3}$\\
		& Runtime[s] & $57.08 $ & $117.35 $ & $286.77 $ & $550.36 $ & $1077.41 $\\\hline
		\multirow{3}*{\shortstack[c]{HochOst4\\\cite{Hochbruck2005}}} & Sup Error & $4.362\mathrm{e}{-1}$ & $9.123\mathrm{e}{-2}$ & $3.253\mathrm{e}{-2}$ & $3.240\mathrm{e}{-2}$ & $3.240\mathrm{e}{-2}$\\
		 & Abs Error & $1.119\mathrm{e}{-1}$ & $1.397\mathrm{e}{-2}$ & $1.876\mathrm{e}{-3}$ & $1.731\mathrm{e}{-3}$ & $1.730\mathrm{e}{-3}$\\
		& Runtime[s] & $56.55 $ & $110.74 $ & $270.06 $ & $541.89 $ & $1046.20$\\\hline
	\end{tabular}
	\caption{
	Results on numerical solutions of \eqref{eq:SABRBSDE} using exponential integrators. Here, we spatially discretize it on $\Pi_x^{\mathrm{TR}}\otimes\Pi_x^{\mathrm{Unif}}$ and solve the resulting system of ODEs. Here, the parameters of $\Pi_x^{\mathrm{TR}}$ are $x_{\mathrm{left}}=0$, $x_{\mathrm{center}}=100$, $x_{\mathrm{right}}=200$, $N_{x,0}=1000$, and $g_1=g_2=50$. For each $N_t$, the numerical solution is evaluated on the grid $\Pi_t^{\mathrm{Unif}}(N_t)\times\Pi_x^{\mathrm{TR}}\times\Pi_x^{\mathrm{Unif}}$. Maximum absolute errors in $\Pi_t^{\mathrm{Unif}}\times([80,120]\cap\Pi_x^{\mathrm{TR}})\times([0.32,0.48]\cap\Pi_x^{\mathrm{Unif}})$ are reported on the row of ``Sup Error'', absolute errors at $(t,s_0,v_0)=(0,100,0.4)$ are on the row of ``Abs Error'', and runtimes in seconds are at the bottom line.}
	\label{tab:sabr_expint}
\end{table}

Table \ref{tab:sabr_expint} reports the maximum absolute errors in $(t,s,v)\in\Pi_t^{\mathrm{Unif}}\times([80,120]\cap\Pi_x^{\mathrm{TR}})\times([0.32,0.48]\cap\Pi_x^{\mathrm{Unif}})$, the absolute errors at $(t,s,v)=(0,100,0.4)$, and the runtime in seconds for different exponential integrators. In this spatial discretization, when $N_t$ increases, the maximum absolute error and the absolute error seem to converge towards $3.240\times10^{-2}$ and $1.730\times10^{-3}$, respectively. To further improve the accuracy of the solutions, increasing the number of points of the spatial grid or expanding the approximated domain should be required.

\subsubsection{The multilevel discretization on a sparse grid}
\label{sec:SABRSG}
Next, we discretize \eqref{eq:SABRBSDE} on a sparse grid using the algorithm presented in Section \ref{sec:SGcomb}. Unlike the discretization on ``full grids'', it approximates BSDE \eqref{eq:SABRBSDE} (driven by Brownian motion) with a sequence of BSDEs driven by CTMCs on grids with different resolutions and constructs the numerical solution by combining their solutions. Recalling that the algorithm presented in Section \ref{sec:SGcomb}, the numerical solution of \eqref{eq:SABRBSDE} can be constructed in the following steps: (1) Let $q\in\mathbb N$ be fixed. (2) For each $\mathbf l=(l_1,l_2)$ such that $q-1\le|l|_1\le q$, discretize $\mathcal S_t$ and $v_t$ on $\Pi_x^{\mathrm{TR}}(0,100,200,2^{l_1-1},5,5)$ and $\Pi_x^{\mathrm{Unif}}(0,0.4,0.8,2^{l_2-1})$, respectively. Solve the resulting system of ODEs and obtain piecewise linear interpolants on $[0,200]\times[0,0.8]$ for each discrete time $t_m$. (3) Combine them using \eqref{eq:SGformula}. In Fig. \ref{fig:SGex}, we specifically show the spatial grids that result from the multilevel discretization with $q=7$.

\begin{table}[!ht]
    \centering
    \begin{tabular}{c|c||cccccc}
    \hline
        $q$ &  & $N_t=10$ & $N_t=20$ & $N_t=50$ & $N_t=100$ & $N_t=200$ & $N_t=500$ \\ \hline
        \multirow{3}*{$7$} & Sup Error & $4.521\mathrm{e}{-2}$ & $4.521\mathrm{e}{-2}$ & $4.521\mathrm{e}{-2}$ & $4.521\mathrm{e}{-2}$ & $4.488\mathrm{e}{-2}$ & $4.468\mathrm{e}{-2}$\\
        & Abs Error & $6.921\mathrm{e}{-3}$ & $6.921{e-3}$ & $6.921\mathrm{e}{-3}$ & $6.921\mathrm{e}{-3}$ & $6.921\mathrm{e}{-3}$ & $6.921\mathrm{e}{-3}$\\
        & Runtime[s] & $1.24$ & $2.38$ & $5.57$ & $10.53$ & $20.22$ & $48.33$\\\hline
        \multirow{3}*{$8$} & Sup Error & $3.469\mathrm{e}{-2}$ & $3.465\mathrm{e}{-2}$ & $3.465\mathrm{e}{-2}$ & $3.465\mathrm{e}{-2}$ & $3.423\mathrm{e}{-2}$ & $3.398\mathrm{e}{-2}$\\
        & Abs Error & $2.454\mathrm{e}{-3}$ & $2.419\mathrm{e}{-3}$ & $2.419\mathrm{e}{-3}$ & $2.419\mathrm{e}{-3}$ & $2.419\mathrm{e}{-3}$ & $2.419\mathrm{e}{-3}$\\
        & Runtime[s] & $2.20$ & $4.08$ & $9.34$ & $17.38$ & $33.07$ & $78.77$\\\hline
        \multirow{3}*{$9$}& Sup Error & $3.961\mathrm{e}{-1}$ & $8.479\mathrm{e}{-2}$ & $3.244\mathrm{e}{-2}$ & $3.230\mathrm{e}{-2}$ & $3.189\mathrm{e}{-2}$ & $3.164\mathrm{e}{-2}$\\
        & Abs Error & $1.193\mathrm{e}{-1}$ & $1.498\mathrm{e}{-2}$ & $1.474\mathrm{e}{-3}$ & $1.309\mathrm{e}{-3}$ & $1.308\mathrm{e}{-3}$ & $1.308\mathrm{e}{-3}$\\
        & Runtime[s] & $3.35$ & $6.18$ & $14.75$ & $28.49$ & $53.03$ & $127.02$\\\hline
        \multirow{3}*{$10$}& Sup Error & $5.537\mathrm{e}{+0}$ & $3.187\mathrm{e}{+0}$ & $2.353\mathrm{e}{-1}$ & $6.068\mathrm{e}{-2}$ & $3.152\mathrm{e}{-2}$ & $3.100\mathrm{e}{-2}$\\
        & Abs Error & $4.942\mathrm{e}{+0}$ & $2.854\mathrm{e}{+0}$ & $5.267\mathrm{e}{-2}$ & $6.812\mathrm{e}{-3}$ & $1.351\mathrm{e}{-3}$ & $1.032\mathrm{e}{-3}$\\
        & Runtime[s] & $6.75$ & $12.74$ & $30.60$ & $59.37$ & $126.54$ & $327.56$\\\hline
    \end{tabular}
    \caption{Results on numerical solutions based on the multilevel discretization on sparse grids. The row of ``Sup Error'' reports the maximum absolute errors on the same points $(t,s,v)$ in Table \ref{tab:sabr_expint} (i.e. $\Pi_t^{\mathrm{Unif}}\times(\Pi_x^{\mathrm{TR}}\cap[80,120])\times(\Pi_x^{\mathrm{Unif}}\cap[0.32,0.48])$.)}
	\label{tab:SABRSG}
\end{table}

\begin{table}[!ht]
	\centering
  \begin{tabular}{c|cccc}
  	\hline
  	$q$ & $7$ & $8$ & $9$ & $10$\\\hline
  	SG & $1475$ & $3333$ & $7431$ & $16393$\\
  	FG & $4225$ & $16641$ & $66049$ & $263169$\\\hline
  \end{tabular}
  \caption{SG : Total numbers of spatial points required for calculating the numerical solution in Section \ref{sec:SABRSG}. FG : The size of the corresponding full grid $(2^{q-d+1}+1)^d$.}
  \label{tab:SGcost}
 \end{table}
Table \ref{tab:SABRSG} reports the result on the numerical solutions for different $q$ and temporal steps $N_t$. Here, we used \verb+HochOst4+ to solve the resulting systems of ODEs and have evaluated maximum absolute errors of $\mathcal Y_t^{t,(s,v)}$ in the same grid as in Section \ref{sec:SABRFG}, namely, $(t,s,v)\in\Pi_t^{\mathrm{Unif}}\times([80,120]\cap\Pi_x^{\mathrm{TR}})\times([0.32,0.48]\cap\Pi_x^{\mathrm{Unif}})$. We observe that the discretization on the sparse grid can provide solutions more efficiently than on the full grid. Table \ref{tab:SGcost} presents the total numbers of spatial points the sparse grid method comsumes in Section \ref{tab:SABRSG} and the numbers of the corresponding full grids. Although the size of our ``full grid'' is $6231$, the sparse grid approach still outperforms in terms of runtime even if the total number of spatial points exceeds this (i.e. $q>8$).

% how to crop margins in pdfs: pdfcrop hoge.pdf
% how to get the bounding box: rungs -o /dev/null -sDEVICE=bbox 20230814_Nt62115.pdf
% https://tex.stackexchange.com/questions/124340/latex-error-cannot-determine-size-of-graphic-in-simlinkerror-pdf-no-bounding
\begin{figure}[!ht]
	\centering
	\begin{tabular}{cc}
		\begin{minipage}[t]{0.245\textwidth}
			\centering
			\includegraphics[width=0.8\textwidth, bb=0.900070 0.460055 352.439989 356.439927]{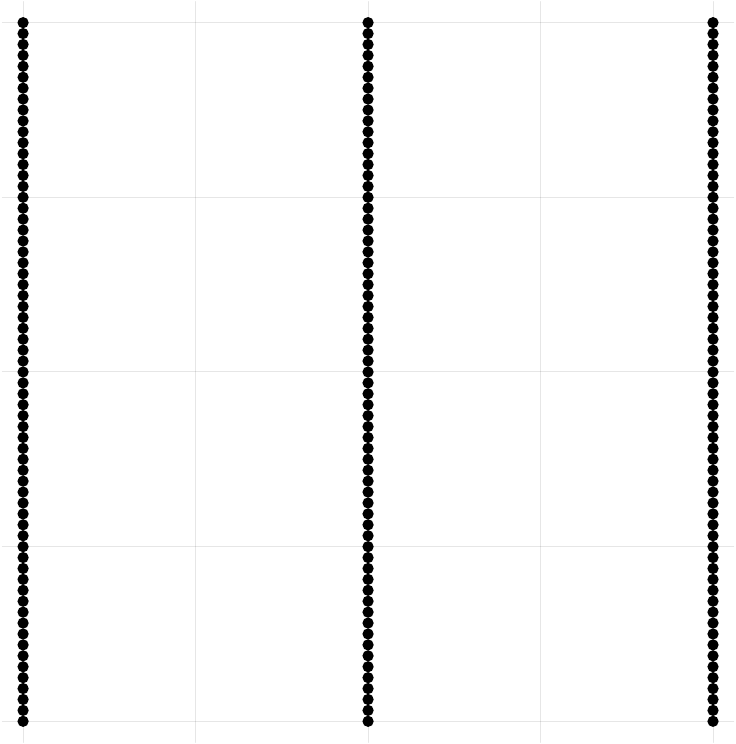}
			\subcaption{$3\times65$.}
			\label{fig:sg1-6}
		\end{minipage}
		\begin{minipage}[t]{0.245\textwidth}
			\centering
			\includegraphics[width=0.8\textwidth, bb=0.900070 0.460055 352.439989 356.439927]{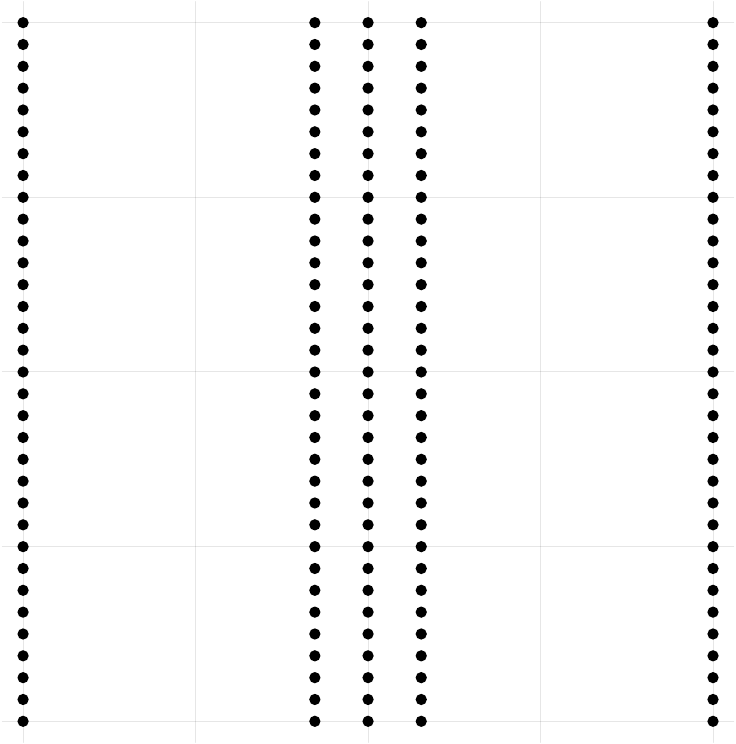}
			\subcaption{$5\times33$.}
			\label{fig:sg2-5}
		\end{minipage}
		\begin{minipage}[t]{0.245\textwidth}
			\centering
			\includegraphics[width=0.8\textwidth, bb=0.900070 0.460055 352.439989 356.439927]{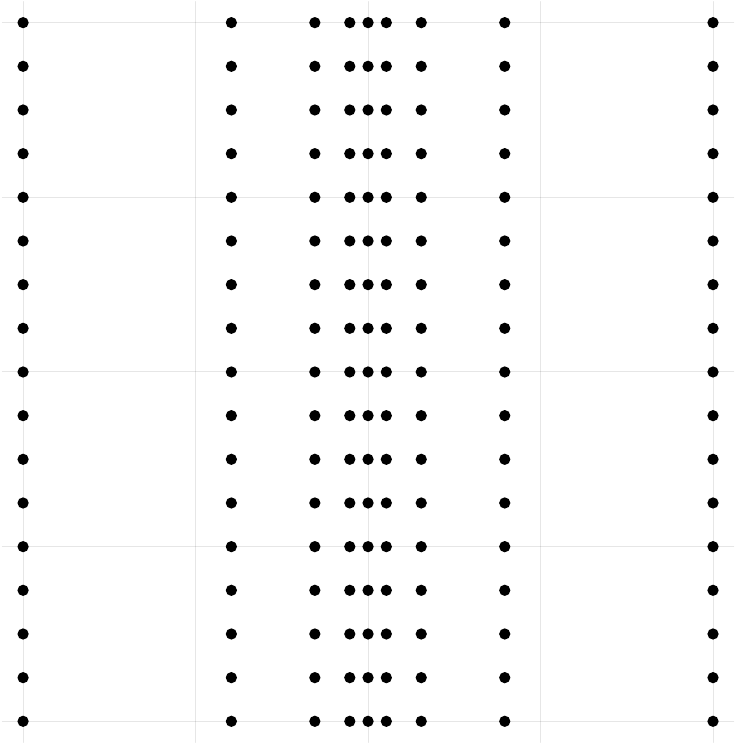}
			\subcaption{$9\times17$.}
			\label{fig:sg3-4}
		\end{minipage}
		\begin{minipage}[t]{0.245\textwidth}
			\centering
			\includegraphics[width=0.8\textwidth, bb=0.900070 0.460055 352.439989 356.439927]{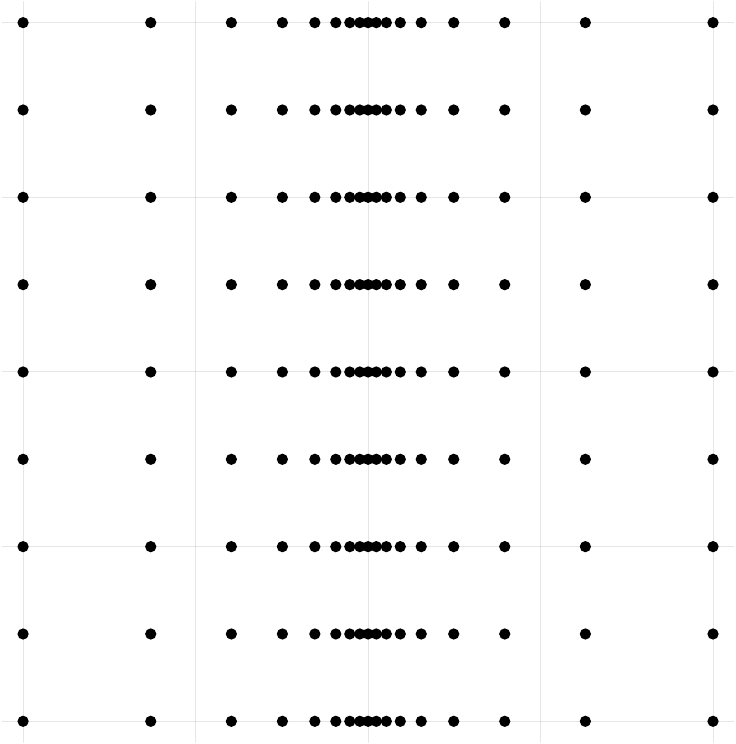}
			\subcaption{$17\times9$.}
			\label{fig:sg4-3}
		\end{minipage}\\
		\begin{minipage}[t]{0.245\textwidth}
			\centering
			\includegraphics[width=0.8\textwidth, bb=0.900070 0.460055 352.439989 356.439927]{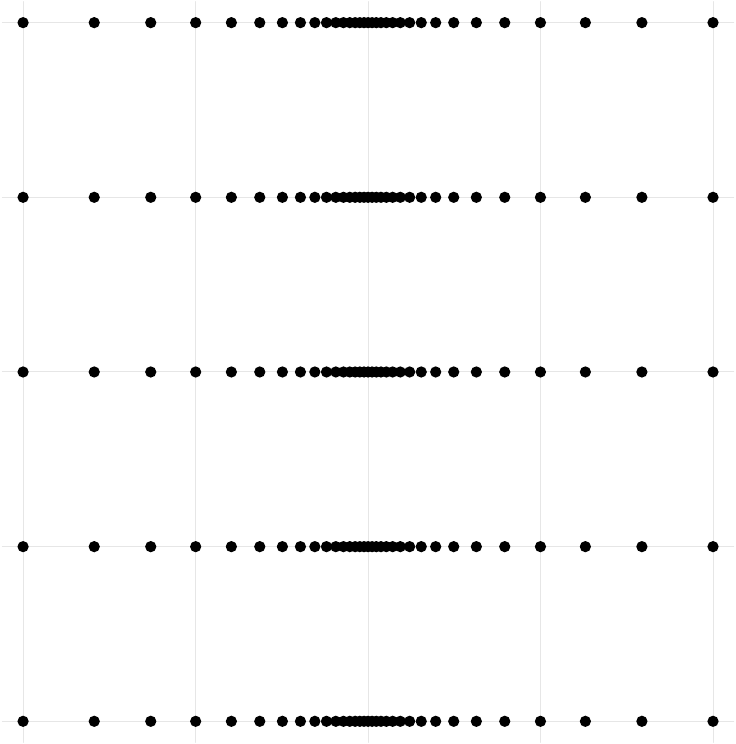}
			\subcaption{$33\times5$.}
			\label{fig:sg5-2}
		\end{minipage}
		\begin{minipage}[t]{0.245\textwidth}
			\centering
			\includegraphics[width=0.8\textwidth, bb=0.900070 0.460055 352.439989 356.439927]{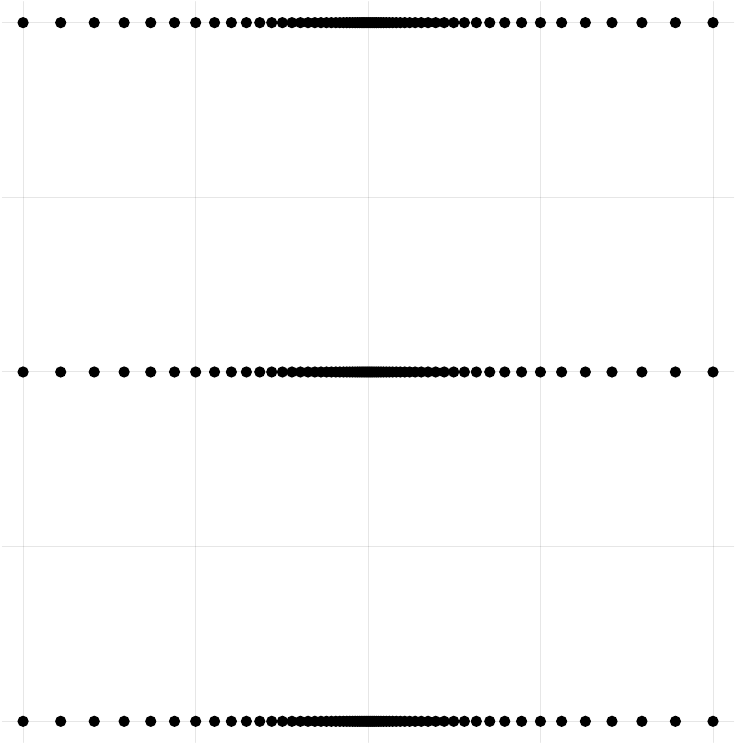}
			\subcaption{$65\times3$.}
			\label{fig:sg6-1}
		\end{minipage}
		\begin{minipage}[t]{0.245\textwidth}
			\centering
			\includegraphics[width=0.8\textwidth, bb=0.900070 0.460055 352.439989 356.439927]{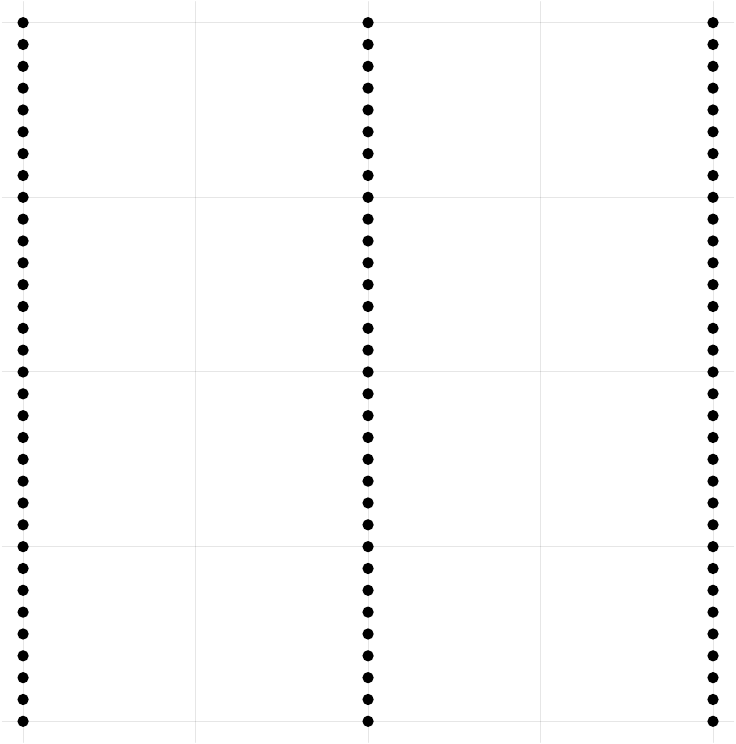}
			\subcaption{$3\times33$.}
			\label{fig:sg1-5}
		\end{minipage}
		\begin{minipage}[t]{0.245\textwidth}
			\centering
			\includegraphics[width=0.8\textwidth, bb=0.900070 0.460055 352.439989 356.439927]{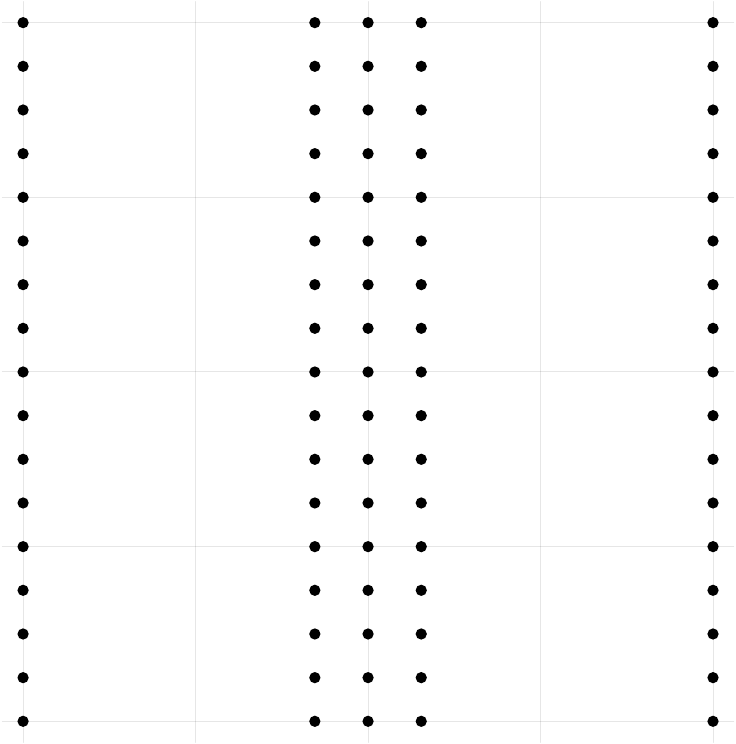}
			\subcaption{$5\times17$.}
			\label{fig:sg2-4}
		\end{minipage}\\
		\begin{minipage}[t]{0.245\textwidth}
			\centering
			\includegraphics[width=0.8\textwidth, bb=0.900070 0.460055 352.439989 356.439927]{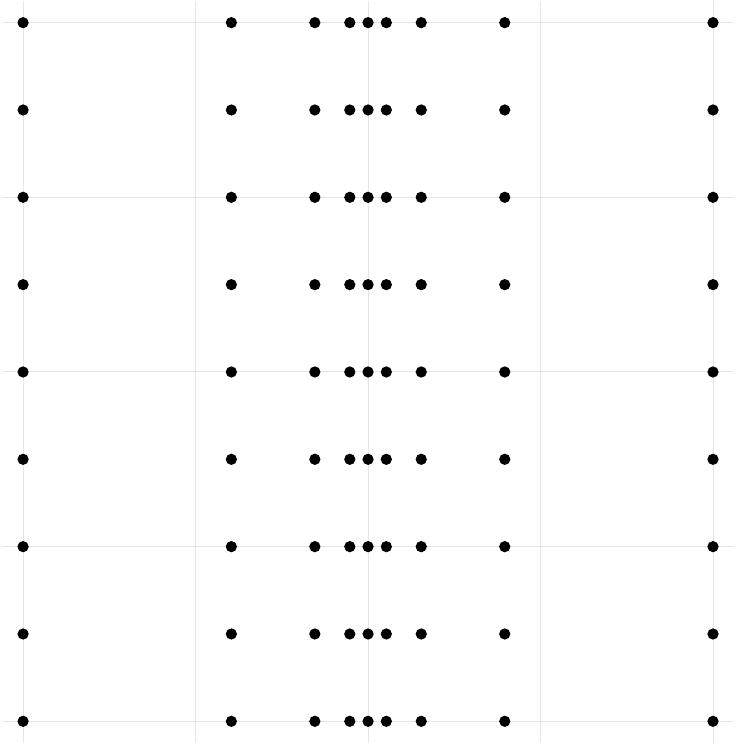}
			\subcaption{$9\times9$.}
			\label{fig:sg3-3}
		\end{minipage}
		\begin{minipage}[t]{0.245\textwidth}
			\centering
			\includegraphics[width=0.8\textwidth, bb=0.900070 0.460055 352.439989 356.439927]{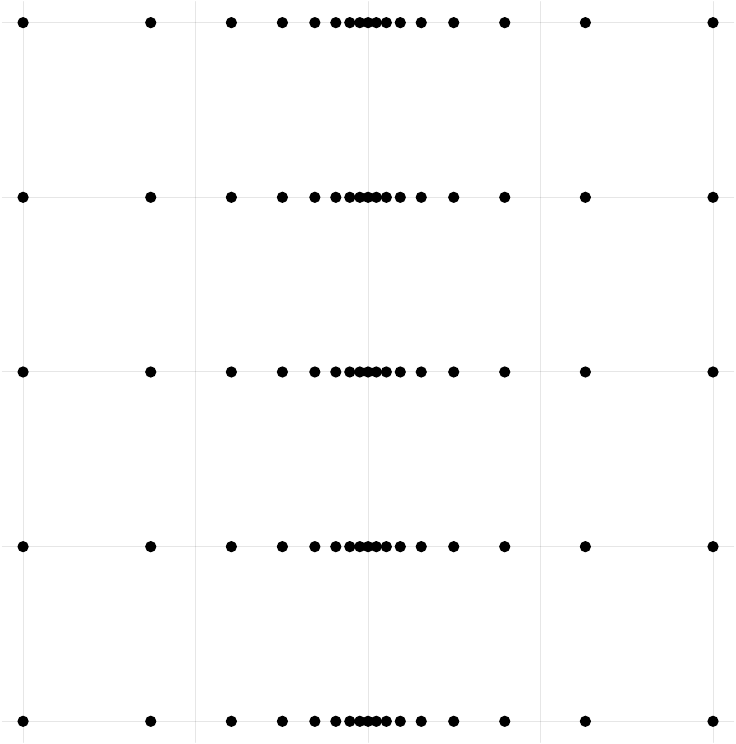}
			\subcaption{$17\times5$.}
			\label{fig:sg4-2}
		\end{minipage}
		\begin{minipage}[t]{0.245\textwidth}
			\centering
			\includegraphics[width=0.8\textwidth, bb=0.900070 0.460055 352.439989 356.439927]{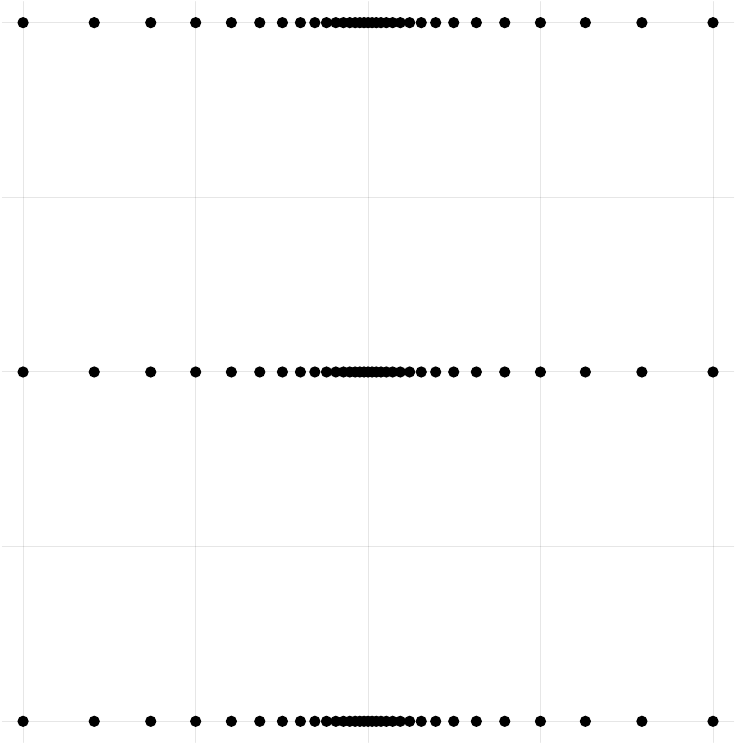}
			\subcaption{$33\times3$.}
			\label{fig:sg5-1}
		\end{minipage}
	\end{tabular}
	\caption{The grids arising from the multilevel discretization for $q=7$ in the situation of Section \ref{sec:SABRSG}. \eqref{eq:SABRBSDE} has been approximated with a combination of 11 BSDEs driven by CTMCs on these grids (a)-(k).}
	\label{fig:SGex}
\end{figure}
\subsection{Multi-asset option pricing using the multilevel discretization on sparse grids}
The final experiment is devoted to solving high-dimensional BSDEs arising from multi-asset option pricing under SLV models. Consider $d$ SLV models, that is, for $i=1,\dots,d$, 
\begin{equation}
	\begin{aligned}
	\mathcal S_t^{(i)} &= S_0^{(i)} + \int_0^t\omega^{(i)}(\mathcal S_s^{(i)}, v_s^{(i)})ds + \int_0^tm^{(i)}(v_s^{(i)})\Gamma^{(i)}(\mathcal S_s^{(i)})dW_s^{(\mathcal S,i)},\\
	v_t^{(i)} &= v_0^{(i)} + \int_0^t\mu^{(i)}(v_s^{(i)})ds + \int_0^t\sigma^{(i)}(v_s^{(i)})dW_s^{(v,i)}.
	\end{aligned}
	\label{eq:d-SLV}
\end{equation}
Here, $W^{(\mathcal S,i)}$ and $W^{(v,i)}$ are correlated as 
\begin{align*}
	\langle W^{(\mathcal S,i)},W^{(\mathcal S,j)}\rangle_t &= c_{i,j}t,\quad \langle W^{(\mathcal S,i)},W^{(v,i)}\rangle_t = \rho_{i,j}t,\quad \langle W^{(v,i)},W^{(v,i)}\rangle_t = r_{i,j}t.
\end{align*}
Let $\mathbf C = \begin{pmatrix} C_{\mathcal S} & C_{\mathcal S,v}\\C_{\mathcal S,v}^\ast & C_v\end{pmatrix}$ be the correlation matrix, where
\[
	C_{\mathcal S} = \begin{pmatrix} c_{1,1} & \ldots & c_{1,d}\\\vdots & \ddots & \vdots\\c_{d,1} & \ldots & c_{d,d}\end{pmatrix},\quad C_{\mathcal S,v} = \begin{pmatrix} \rho_{1,1} & \ldots & \rho_{1,d}\\\vdots & \ddots & \vdots\\\rho_{d,1} & \ldots & \rho_{d,d}\end{pmatrix},\quad C_v = \begin{pmatrix} r_{1,1} & \ldots & r_{1,d}\\\vdots & \ddots & \vdots\\r_{d,1} & \ldots & r_{d,d}\end{pmatrix}.
\]
Using a lower triangular matrix $\mathbf L$ constructed from the Cholesky decomposition $\mathbf C = \mathbf L\mathbf L^\ast$ and $\mathbf W$ defined as 
\[
	\mathbf W = \mathbf L^{-1}\begin{pmatrix} W^{(\mathcal S,1)}\\\vdots\\W^{(\mathcal S,d)}\\W^{(v,1)}\\\vdots\\W^{(v,d)}\end{pmatrix}
\]
is $2d$-dimensional standard Brownian motion. The corresponding BSDE is 
\begin{equation}
	\label{eq:d-SLVBSDE}
	\begin{split}
	\mathcal X_t &= \mathcal X_0 + \int_0^t\boldsymbol\mu(\mathcal X_s)ds + \int_0^t\boldsymbol\sigma(\mathcal X_s)d\mathbf W_s\quad \text{for}\quad t\in[0,T]\\
	\mathcal Y_t &= g(\mathcal S_T^{(1)},\dots,\mathcal S_T^{(d)})-\int_t^Tf(s,\mathcal X_s,\mathcal Y_s,\mathcal Z_s)ds -\int_t^T\mathcal Z_s^\ast d\mathbf W_s,
	\end{split}
\end{equation}
where 
\[
	\mathcal X_t = \begin{pmatrix}\mathcal S_t^{(1)}\\\vdots\\\mathcal S_t^{(d)}\\v_t^{(1)}\\\vdots\\v_t^{(d)}\end{pmatrix},\quad 
	\boldsymbol\mu(x) = \begin{pmatrix}\omega^{(1)}(x_1,x_{d+1})\\\vdots\\\omega^{(d)}(x_d,x_{2d})\\\mu^{(1)}(x_{d+1})\\\vdots\\\mu^{(d)}(x_{2d})\end{pmatrix},\quad \boldsymbol\sigma(x) = \operatorname{diag}\begin{pmatrix}m^{(1)}(x_{d+1})\Gamma^{(1)}(x_1)\\\vdots\\m^{(d)}(x_{2d})\Gamma^{(d)}(x_{2d})\\ \sigma^{(1)}(x_{d+1})\\\vdots\\\sigma^{(d)}(x_{2d})
	\end{pmatrix}\mathbf L,
\]
and the driver is 
\[
	f(t,x,y,z)=r(y-z^\ast\boldsymbol\sigma(x)^{-1}x)^+-R(y-z^\ast\boldsymbol\sigma(x)^{-1}x)^-+z^\ast\boldsymbol\sigma(x)^{-1}\boldsymbol\mu(x).
\]

\subsubsection{Basket option under 2 Heston-SABR models}
Consider pricing of the European basket call option whose basket is comprised of two Heston-SABR models. That is, the coefficient functions in \eqref{eq:d-SLV} are 
\begin{equation}
	\begin{gathered}
		\omega^{(i)}(s,v) = b^{(i)}\cdot s,\quad m^{(i)}(v) = \sqrt{v},\quad \Gamma^{(i)}(s) = s^{\beta^{(i)}},\\
		\mu^{(i)}(v) = \eta^{(i)}(\theta^{(i)}-v),\quad \sigma^{(i)}(v) = \alpha^{(i)}\sqrt{v}.
	\end{gathered}
	\label{eq:2hestonsabrcoef}
\end{equation}
for $i=1,2$. The corresponding BSDE is four-dimensional and results in \eqref{eq:SLVBSDE} with $g$ replaced as 
\begin{align*}
	g(f_1,f_2) = (\lambda_1f_1+\lambda_2f_2-K)^+,
\end{align*}
where $\lambda_1$ and $\lambda_2$ are constants. Since a non-differentiability of $g$ is appreared in the hyperplane $\{(s^{(1)},s^{(2)},v^{(1)},v^{(2)}):\lambda_1s^{(1)}+\lambda_2s^{(2)}-K=0\}$, we use a linear coordinate transformation, say
\begin{align*}
	\widehat{\mathcal X_t} = \begin{pmatrix}\widehat{\mathcal S}_t^{(1)}\\\widehat{\mathcal S}_t^{(2)}\\v_t^{(1)}\\v_t^{(2)}\end{pmatrix} = B\begin{pmatrix}\mathcal S_t^{(1)}\\\mathcal S_t^{(2)}\\v_t^{(1)}\\v_t^{(2)}\end{pmatrix},\quad\text{where}\quad B = \begin{pmatrix} \lambda_1 & \lambda_2 & 0 & 0\\-\lambda_1 & \lambda_2 & 0 & 0\\0 & 0 & 1 & 0\\0 & 0 & 0 & 1\end{pmatrix},
\end{align*}
before the spatial discretization, that turns out to be 
\begin{equation}
	\begin{split}
	\widehat{\mathcal X}_t &= \widehat{\mathcal X}_0 + \int_0^t\mathbf B\boldsymbol\mu(\mathbf B^{-1}\widehat{\mathcal X}_s)ds + \int_0^t\mathbf B\boldsymbol\sigma(\mathbf B^{-1}\widehat{\mathcal X}_s)d\mathbf W_s,\\
	\mathcal Y_t &= (\widehat{\mathcal S}_T^{(1)}-K)^+-\int_t^Tf(s,\mathbf B^{-1}\widehat{\mathcal X}_s,\mathcal Y_s,\mathcal Z_s)ds -\int_t^T\mathcal Z_s^\ast d\mathbf W_s.
	\end{split}
	\label{eq:HestonSABRBasket}
\end{equation}
The parameters chosen here are:
\[
\begin{tabular}{c|cccccccccc}
  $i$ & $T$ & $K$ & $R$ & $r$ & $\lambda^{(i)}$ & $\beta^{(i)}$ & $\eta^{(i)}$ & $\theta^{(i)}$ & $\alpha^{(i)}$ & $b^{(i)}$\\\hline
  $1$ & \multirow{2}*{$1.0$} & \multirow{2}*{$100$} & \multirow{2}*{$0.07$} & \multirow{2}*{$0.01$} & $0.5$ & $0.6$ & $0.9$ & $0.02$ & $0.65$ & $0.01$\\
  $2$ & & & &  & $0.5$ & $0.07$ & $0.2$ & $0.3$ & $0.3$ & $0.01$
\end{tabular}
\]
The correlation matrices are:
\[
	C_{\mathcal S} = \begin{pmatrix}1.0 & 0.5\\0.5 &1.0\end{pmatrix},\quad
	C_{\mathcal S,v} = \begin{pmatrix}0.65 & 0.3\\-0.1 &0.05\end{pmatrix},\quad
	C_v = \begin{pmatrix}1.0 & 0.7\\0.7 &1.0\end{pmatrix}.
\]
For $(\widehat s^{(1)},\widehat s^{(2)},\widehat v^{(1)},\widehat v^{(2)})^\ast = \mathbf B(s^{(1)},s^{(2)},v^{(1)},v^{(2)})^\ast$, we approximate the spatial domain as $(\widehat s^{(1)},\widehat s^{(2)},\widehat v^{(1)},\widehat v^{(2)}) \in [51,149]\times[-49,49]\times[0.01,0.79]\times[0.01,0.59]$, and apply the multilevel discretization on a sparse grid. We apply a Tavella-Randall grid to the first dimension and the uniform grids to the others. The parameters of the Tavella-Randall grid are $g_1=g_2=1.0$.

Table \ref{tab:HestonSABRSG} reports sparse grid solutions $\mathcal Y_0^{0,(s^{(1)},s^{(2)},v^{(1)},v^{(2)})}$ at $(s^{(1)},s^{(2)},v^{(1)},v^{(2)})=(100,100,0.4,0.3)$ calculated using \verb+HochOst4+ and their computational times in seconds, for different $q$ and $N_t$. The numerical solutions seem to converge towards approximately $7.517$.

\begin{table}[!ht]
    \centering
    \begin{tabular}{c|c||ccccc}
		    \hline
        $q$ & & $N_t=10$ & $N_t=20$ & $N_t=50$ & $N_t=100$ & $N_t=200$ \\ \hline
				\multirow{2}*{$8$} & $\mathcal Y_0^{0,(100,100,0.4,0.3)}$ & $7.50016$ & $7.50013$ & $7.50013$ & $7.50013$ & $7.50013$\\
				 & Runtime[s] & $50.71 $ & $57.89 $ & $141.68 $ & $276.37 $ & $552.95 $\\\hline
				\multirow{2}*{$9$} & $\mathcal Y_0^{0,(100,100,0.4,0.3)}$ & $7.51457$ & $7.51455$ & $7.51455$ & $7.51455$ & $7.51455$\\
				 & Runtime[s] & $119.65 $ & $195.74 $ & $484.67 $ & $931.01 $ & $1904.49 $\\\hline
				\multirow{2}*{$10$} & $\mathcal Y_0^{0,(100,100,0.4,0.3)}$ & $7.51522$ & $7.51656$ & $7.51658$ & $7.51658$ & $7.51658$\\
				 & Runtime[s] & $838.15 $ & $1487.51 $ & $3376.85 $ & $6542.38 $ & $12817.34 $\\\hline
				\multirow{2}*{$11$} & $\mathcal Y_0^{0,(100,100,0.4,0.3)}$ & $7.43057$ & $7.49986$ & $7.51675$ & $7.51713$ & $7.51714$\\
				 & Runtime[s] & $5965.28 $ & $10932.85 $ & $24935.39 $ & $48705.96$ & $97810.97$\\\hline
    	\end{tabular}
	\caption{Results on numerical solutions $\mathcal Y_0^{0,(100,100,0.4,0.3)}$ of \eqref{eq:HestonSABRBasket} using a multilevel spatial discretization. }
	\label{tab:HestonSABRSG}
\end{table}

\begin{table}[!ht]
	\centering
  \begin{tabular}{c|cccc}
  	\hline
  	$q$ & $8$ & $9$ & $10$ & $11$\\\hline
  	SG & $36901$ & $112105$ & $320675$ & $877655$\\
  	FG & $1185921$ & $17850625$ & $276922881$ & $4362470401$\\\hline
  \end{tabular}
  \caption{SG : Total numbers of spatial points of grids comprised of the sparse grid solutions for different $q$. FG : The size of the corresponding full grid $(2^{q-d+1}+1)^d$.}
  \label{tab:SGcost2}
 \end{table}

\section{Conclusion}
\label{sec:conclusion}
A Markov BSDE driven by a CTMC associates with a system of ODEs. With arguments based on this observation, we proposed the multi-stage Euler-Maruyama methods for the BSDE, directly related to exponential integrators for solving the system of ODEs. Together with a suitable spatial discretization, these methods can be applied to solve BSDEs driven by Brownian motion. We also proposed a multilevel spatial discretization based on a sparse grid combination technique for handling a high-dimensional BSDE driven by Brownian motion, and it enables us to calculate solutions with less computational cost. The efficiency of our numerical methods has been confirmed through numerical experiments using derivative pricing problems in mathematical finance.

\section*{Acknowledgement}
% \bmhead{Acknowledgments}
I would like to express my sincere gratitude to my supervisor, Professor Jun Sekine, for his helpful discussions and invaluable advice. This work was supported by JST SPRING, Grant Number JPMJSP2138.
\appendix
\section{Proofs}
\label{appendix:proof}
\subsection{Proof of Theorem \ref{thm:CTMCFeynmanKac}}
\label{sec:ProofOfThm2.2}
\begin{proof}
	For a solution $U_t$ to \eqref{eq:ODE}, the It\^o formula immediately implies $(Y_t,Z_t)=(X_t^\ast U_t,U_t)$ solves \eqref{eq:MarkovCTMCBSDE}.

	Let $t_1$ and $t_2$ be fixed. Without loss of generality, assume that $t_2>t_1$. 
	\begin{align*}
		Y_{t_1}^{t_1,e_i}-Y_{t_2}^{t_2,e_i} &= \mathbb E[Y_{t_1}^{t_1,e_i}-Y_{t_2}^{t_1,e_i}+(X_{t_2}^{t_1,e_i})^\ast V_{t_2}-e_i^\ast X_{t_2}^{t_2,e_i}]\\
		&= \begin{multlined}[t]
			\mathbb E\left[\int_{]t_1,t_2]}h(X_{u-}^{t_1,e_i},u,Y_{u-}^{t_1,e_i},Z_u^{t_1,e_i})du -\int_{]t_1,t_2]}dM_u^\ast Z_u^{t_1,e_i}\right.\\
			\left.+\left(\int_{]t_1,t_2]}Q_u^\ast X_{u-}^{t_1,e_i}-M_{t_2}+M_{t_1}\right)^\ast V_{t_2}\right]
		\end{multlined}\\
		&= \mathbb E\left[\int_{]t_1,t_2]}[h(X_{u-}^{t_1,e_i},u,Y_{u-}^{t_1,e_i},Z_u^{t_1,e_i})+(X_{u-}^{t_1,e_i})^\ast Q_uV_{t_2}]du\right]
	\end{align*}
	Hence, 
	\[
		|Y_{t_1}^{t_1,e_i}-Y_{t_2}^{t_2,e_i}| \le C\sqrt{t_2-t_1}\sqrt{\int_{]t_1,t_2]}\mathbb E[|h(X_{u-}^{t_1,e_i},u,Y_{u-}^{t_1,e_i},Z_u^{t_1,e_i})+(X_{u-}^{t_1,e_i})^\ast Q_uV_{t_2}|^2]du}
	\]
	Using the uniform boundedness of $Q_u$, the Lipschitz continuity of $h$, evaluate the integrand as 
	\begin{multline*}
		|h(X_{u-}^{t_1,e_i},u,Y_{u-}^{t_1,e_i},Z_u^{t_1,e_i})+(X_{u-}^{t_1,e_i})^\ast Q_uV_{t_2}|^2\\
		\le C(|Y_{u-}^{t_1,e_i}|^2+\|Z_u^{t_1,e_i}\|_{X_{u-}^{t_1,e_i}}^2 + |h(X_{u-}^{t_1,e_i},u,0,0)|^2 + \sup_{\substack{0\le s,u\le T\\i=1,\dots,N}}|e_j^\ast Q_sV_u|^2).
	\end{multline*}
	Recall that $Q_s$ is assumed to be uniform bounded and that
	\begin{align*}
		\mathbb E\int_{]t_1,t_2]}|h(X_{u-}^{t_1,e_i},u,0,0)|^2du &= \sum_{j=1}^N\int_{]t_1,t_2]}|h(e_j,u,0,0)|^2\mathbb P(X_u^{t_1,e_i}=e_j|X_{t_1}^{t_1,e_i}=e_i)du\\
		&\le C\sup_{j=1,\dots,N}\int_{]0,T]}|h(e_j,u,0,0)|^2du.
	\end{align*}
	Hence,
	\[
		|Y_{t_1}^{t_1,e_i}-Y_{t_2}^{t_2,e_i}|
		\le C\sqrt{t_2-t_1}
			\begin{multlined}[t]
			\left(\mathbb E\left[\sup_{0\le u\le T}|Y_u^{t_1,e_i}|^2 + \int_{]0,T]}\|Z_u^{t_1,e_i}\|_{X_{u-}^{t_1,e_i}}^2\right.\right.\\
			\left.\left.+\sup_{i=1,\dots,N}\int_{]0,T]}|h(e_i,u,0,0)|^2du+1\right]du\right),
			\end{multlined}
	\]
	from which the continuity of $t\mapsto Y_t^{t,e_i}$ directly follows, as well as of $V_t$. Using 
	\[
		(\Delta X_u^{t,e_i})^\ast V_u = \Delta((X_u^{t,e_i})^\ast V_u) = \Delta Y_u^{u,X_u^{t,e_i}} = \Delta Y_u^{t,e_i} = \Delta M_u^\ast Z_u^{t,e_i} = (\Delta X_u^{t,e_i})^\ast Z_u^{t,e_i},
	\]
	we obtain $\int_{]0,t]}(dX_u^{t,e_i})^\ast (Z_u^{t,e_i}-V_u) = 0$. Notice that
	\[
		\int_{]0,t]}dM_u^\ast (Z_u^{t,e_i}-V_u) = -\int_{]0,t]}(X_{u-}^{t,e_i})^\ast Q_u(Z_u^{t,e_i}-V_u)du = 0,
	\]
	since any predictable finite variation martingales starting at $0$ takes zero constantly (e.g. Corollary 8.2.14, p.204 in \cite{Cohen2015}.) Hence
	\[
		\mathbb E\int_{]0,T]}\|Z_u^{t,e_i}-V_u\|_{X_{u-}^{t,e_i}}^2du = \mathbb E\left|\int_{]0,T]}dM_u^\ast(Z_u^{t,e_i}-V_u)\right|^2 = 0,
	\]
	which means $Z_u^{t,e_i}\sim_MV_u$. Together it with the Lipschitz continuity, 
	\begin{equation}
		\label{eq:A11}
		h(X_{u-}^{t,e_i},u,(X_{u-}^{t,e_i})^\ast V_u,Z_u^{t,e_i})=h(X_{u-}^{t,e_i},u,(X_{u-}^{t,e_i})^\ast V_u,V_u),\quad du\otimes d\mathbb P\text{-a.s.}
	\end{equation}
	Plugging it into the conditional expectation representation of $Y_t^{t,e_i}$, we obtain for $t\in[0,T]$,
	\[
		e_i^\ast V_t = Y_t^{t,e_i} = e_i^\ast\Phi(T,t)G+e_i^\ast\int_{]t,T]}\Phi(u,t)H(u,V_s)du,
	\]
	which results in the variation-of-constants of \eqref{eq:ODE} in what follows:
	\[
		V_t = \Phi(T,t)G+\int_t^T\Phi(s,t)H(s,V_s)ds.
	\]
\end{proof}
\subsection{Proof of Proposition \ref{3.1}}
\label{sec:ProofOfLem3.1}
\begin{proof}
	We show $e_j^\ast Qe_i \ge0$ for $i\neq j$ and $\sum_ie_j^\ast Qe_i=0$ for all $j=1,\dots,N$. It is trivial for $j=1$ and $N$ since $e_1^\ast Q=e_N^\ast Q=0$. Let $j=2,\dots,N-1$ be fixed. The condition $\sum_{i=1}^Ne_j^\ast Qe_i = 0$ clearly holds since 
	\begin{align*}
		e_i^\ast Qe_{i-1} + e_i^\ast Qe_{i+1}^\ast &= 
		\frac{\sigma^2(x_i)-\delta x_i\mu(x_i)}{\delta x_{i-1}(\delta x_{i-1}+\delta x_i)} + \frac{\sigma^2(x_i)+\delta x_{i-1}\mu(x_i)}{\delta x_i(\delta x_{i-1}+\delta x_i)} \\
		&= -\frac{(\delta x_i-\delta x_{i-1})\mu(x_i)-\sigma^2(t,x_i)}{\delta x_i\delta x_{i-1}} = -e_i^\ast Qe_i.
	\end{align*}
	We remain to prove the nonnegativity of off-diagonal elements of $Q$. Denote $\mu(x_i)=\mu(x_i)^+-\mu(x_i)^-$, where
	\[
		\mu(x)^+ = \max\left\{\mu(x),0\right\}(\ge0)\quad\text{and}\quad \mu(x)^- = -\min\left\{\mu(x),0\right\}(\ge0).
	\]
	We obtain
	\begin{equation}
		\frac{\sigma^2(x_i)-\delta x_i\mu(x_i)}{\delta x_{i-1}(\delta x_{i-1}+\delta x_i)} =\frac{\mu(x_i)^-}{\delta x_{i-1}}+\frac{\sigma^2(s_i)-(\delta x_{i-1}\mu(x_i)^-+\delta x_i\mu(x_i)^+)}{\delta x_{i-1}(\delta x_{i-1}+\delta x_i)}\label{eq8}
	\end{equation}
	and
	\begin{equation}
		\frac{\sigma^2(x_i)+\delta x_{i-1}\mu(x_i)}{\delta x_i(\delta x_{i-1}+\delta x_i)} = \frac{\mu(x_i)^+}{\delta x_i}+\frac{\sigma^2(x_i)-(\delta x_{i-1}\mu(x_i)^-+\delta x_i\mu(x_i)^+)}{\delta x_i(\delta x_{i-1}+\delta x_i)}.\label{eq9}
	\end{equation}
	The first terms on the right-hand side of \eqref{eq8} and \eqref{eq9} are clearly nonnnegative. Under the condition \eqref{cond}, we obtain 
	\begin{align*}
		\sigma^2(x_i)\ge\max_{1\le j\le N-1}\{\delta x_j\}\cdot|\mu(x_i)|&=\max_{1\le j\le N-1}\{\delta x_j\}(\mu^+(x_i)+\mu^-(x_i))\\
		&\ge \delta x_{i-1}\mu^-(x_i)+\delta x_i\mu^+(x_i),
	\end{align*}
	so that the second terms in \eqref{eq8} and \eqref{eq9} are nonnegative. Moreover, if \eqref{cond} is strict, \eqref{eq8} and \eqref{eq9} are positive.
\end{proof}

\subsection{Proof of Proposition \ref{lemma3.2}}
\label{sec:ProofOfLem3.2}
In this subsection, we denote $I_N$ as the $N\times N$ identity matrix, $\delta_{ij}$ as the Kronecker's delta, $\mathbbm{1}=(1,1,\dots,1)^\ast\in\mathbb R^N$, and
\begin{equation}
	\label{eq:psiei}
	\psi_{M_1,e_{i}} = \operatorname{diag}(M_1^\ast e_{i})-M_1^\ast\operatorname{diag}(M_1,e_{i})-\operatorname{diag}(e_{i})M_1.
\end{equation}
We sometimes omit subscripts when they can be unambiguously determined from the context.
\paragraph{Step 1.}
Let $h(e_{i},t,y,z) \coloneqq f(t,x_i,e_{i}^\ast z,\sigma^\ast(x_i)(e_{i}^\ast\widetilde D_1^{(1)}z,\dots,e_{i,N}^\ast\widetilde D_1^{(d)}z)^\ast)$ for $t\in[0,T]$, $i=1,\dots,N$, $y\in\mathbb R$ and $z\in\mathbb R^N$. The Lipschitz continuity for $f$ implies
\begin{multline*}
	|h(e_{i},t,y,z)-h(e_{i},t,y',z')|^2\\
	\le L(|y-y'|^2+\|\sigma^\ast(x_i)(e_{i}^\ast\widetilde D_1^{(1)}(z-z'),\dots,e_{i}^\ast\widetilde D_1^{(d)}(z-z'))^\ast\|^2).
\end{multline*}
To obtain the desired result, it is sufficient to show 
\begin{equation}
	\|\sigma^\ast(x_i)(e_{i}^\ast\widetilde D_1^{(1)}(z-z'),\dots,e_{i}^\ast\widetilde D_1^{(d)}(z-z'))^\ast\|^2\le C\|z-z'\|_{e_i}^2 \label{eq:est}
\end{equation}
for some constant $C>0$ for any $i=1,\dots,N$. As the left-hand side of \eqref{eq:est} can be represented as a quadratic form of symmetric matrix
\begin{equation}
	\label{eq:M0}
	M_0 \coloneqq \sum_{p,q=1}^d(\sigma\sigma^\ast)^{(p,q)}(x_i)(\widetilde D_1^{(q)})^\ast e_{i}e_{i}^\ast\widetilde D_1^{(p)},
\end{equation}
\eqref{eq:est} is equivalent to the positive semi-definiteness of $M_{i,C} \coloneqq C\psi_{Q,e_i}-M_0$. Before showing this, we require to several lemmas in Step 2.

\paragraph{Step 2.}
In this step, for any $N\times N$ matrix $M_1$, we denote $M_1^{(N-1)}$ as the $(N-1)\times(N-1)$ matrix obtained by removing the last row and column vector of $M_1$.
\begin{lemma}
	\label{lem:possemlemma}
	Let $M_1, M_2$ be $N\times N$ real symmetric matrices satisfying $M_1\mathbbm{1}=0$ and $M_2\mathbbm{1}=0$. If $M_1^{(N-1)}$ is positive definite, $cM_1-M_2$ is positive semi-definite for sufficiently large $c>0$.
\end{lemma}
\begin{proof}
	Let $\lambda_1>0$ be the minimum eigenvalue of $M_1^{(N-1)}$, and $\lambda_2$ be the maximum eigenvalue of $M_2^{(N-1)}$. $cM_1^{(N-1)}-M_2^{(N-1)}$ is positive definite for $c>(\lambda_2/\lambda_1)\vee0$. Indeed, for any non-zero vector $z$, we see that
	\[
		z^\ast (cM_1^{(N-1)}-M_2^{(N-1)})z \ge c\lambda_1\|z\|^2-\lambda_2\|z\|^2 > 0.
	\]
	Since any real symmetric matrix $M$ satisfying $M\mathbbm{1}=0$ has the following block matrix representation
	\[
		M_0 = \begin{pmatrix}
			M_0^{(N-1)} & -M_0^{(N-1)}\mathbbm{1}\\
			-\mathbbm{1}^\ast M_0^{(N-1)} & \mathbbm{1}^\ast M_0^{(N-1)}\mathbbm{1}
		\end{pmatrix},
	\]
	the quadratic form of $M$ can be written as 
	\begin{equation}
		\label{eq:quadraticform}
		z^\ast M_0z
		= (x^\ast, y)\begin{pmatrix}
			M_0^{(N-1)} & -M_0^{(N-1)}\mathbbm{1}\\
			-\mathbbm{1}^\ast M_0^{(N-1)} & \mathbbm{1}^\ast M_0^{(N-1)}\mathbbm{1}
		\end{pmatrix}
		\begin{pmatrix}x\\y\end{pmatrix}
		= (x-y\mathbbm{1})^\ast M_0^{(N-1)}(x-y\mathbbm{1}),
	\end{equation}
	for $z=(x^\ast,y)^\ast\in\mathbb R^N$. Applying $M_0=cM_1-M_2$ to \eqref{eq:quadraticform}, the quadratic form takes a positive value for $z=(x^\ast, y)$ except $x-y\mathbbm{1}=0$, and takes zero if $x-y\mathbbm{1}=0$. Hence the positive semi-definiteness of $cM_1-M_2$ is obtained.
\end{proof}

\begin{lemma}
	For $i,j=1,\dots,N$ and any matrix $M_1$,
	\begin{equation}
		\label{eq:psielements}
		\psi_{M_1,e_i}e_j = (e_i^\ast M_1e_j)(e_j-e_i) -\delta_{ij}M_1^\ast e_j.
	\end{equation}
\end{lemma}
\begin{proof}
	We can see it directly as:
	\begin{align*}
		\psi_{M_1,e_i}e_j &= \operatorname{diag}(M_1^\ast e_i)e_j -M_1^\ast e_ie_i^\ast e_j-e_ie_i^\ast M_1e_j\\
		&= e_i^\ast M_1e_j\cdot e_j-\delta_{ij}M_1^\ast e_i-e_i^\ast M_1e_j\cdot e_i= (e_i^\ast M_1e_j)(e_j-e_i) -\delta_{ij}M_1^\ast e_j.
	\end{align*}
\end{proof}

\begin{lemma}
	\label{lem:psiN-1positive}
	Let $i=1,\dots,N$ be fixed, and $M_1$ be a matrix satisfying $e_i^\ast M_1e_j>0$ for all $j\in\{1,\dots,N\}\setminus\{i\}$ and $e_i^\ast M_1\mathbbm{1}=0$ (Note that $e_i^\ast M_1e_i$ is equal to $-\sum_{j=1,j\neq i}^Ne_i^\ast M_1e_j$.) Then, $\psi_{M_1,e_i}^{(N-1)}$ is positive definite.
\end{lemma}
\begin{proof}
	Since \eqref{eq:psielements}, $e_i^\ast M_1e_j$ for $j\neq i$ is strictly positive under the assumption of this lemma. Thus $\psi_{M_1,e_N}^{(N-1)} = \operatorname{diag}(e_N^\ast M_1e_1,\dots,e_N^\ast M_1e_{N-1})$ is positive definite obviously. Let $x\in\mathbb R^{N-1}\setminus\{0\}$ and $i<N$ be fixed. 
	\begin{align}
		&x^\ast\psi_{M_1,e_i}^{(N-1)}x =\sum_{j=1}^{N-1}\sum_{k=1}^{N-1}[(\delta_{jk}-\delta_{ki})e_i^\ast M_1e_j-\delta_{ij}e_i^\ast M_1e_k]e_k^\ast xe_j^\ast x\\
		&\begin{multlined}
			=[(\delta_{ii}-\delta_{ii})e_i^\ast M_1e_i-\delta_{ii}e_i^\ast M_1e_i]e_i^\ast xe_i^\ast x \\
			+ \sum_{\substack{j=1\\j\neq i}}^{N-1}[(\delta_{ji}-\delta_{ii})e_i^\ast M_1e_j-\delta_{ij}e_i^\ast M_1e_i]e_i^\ast xe_j^\ast x\\
			+\sum_{\substack{k=1\\k\neq i}}^{N-1}[(\delta_{ik}-\delta_{ki})e_i^\ast M_1e_i-\delta_{ii}e_i^\ast M_1e_k]e_k^\ast xe_i^\ast x\\
			+\sum_{\substack{j=1\\j\neq i}}^{N-1}\sum_{\substack{k=1\\k\neq i}}^{N-1}[(\delta_{jk}-\delta_{ki})e_i^\ast M_1e_j-\delta_{ij}e_i^\ast M_1e_k]e_k^\ast xe_j^\ast x
		\end{multlined}\\
		% &=-e_i^\ast M_1e_i(e_i^\ast x)^2-\sum_{\substack{j=1\\j\neq i}}^{N-1}e_i^\ast M_1e_je_i^\ast xe_j^\ast x-\sum_{\substack{k=1\\k\neq i}}^{N-1}e_i^\ast M_1e_ke_k^\ast xe_i^\ast x+\sum_{\substack{j=1\\j\neq i}}^{N-1}e_i^\ast M_1e_j(e_j^\ast x)^2\\
		&=-e_i^\ast M_1e_i(e_i^\ast x)^2-2\left[\sum_{\substack{j=1\\j\neq i}}^{N-1}e_i^\ast M_1e_je_j^\ast x\right]e_i^\ast x+\sum_{\substack{j=1\\j\neq i}}^{N-1}e_i^\ast M_1e_j(e_j^\ast x)^2.\label{eq:quadraticeq}
	\end{align}
	Note that $-e_i^\ast M_1e_i>0$. The discriminant of \eqref{eq:quadraticeq} as a quadratic polynomial of $e_i^\ast x$ can be evaluated in what follows:
	\begin{equation}
		\begin{aligned}
		&\left[\sum_{\substack{j=1\\j\neq i}}^{N-1}\sqrt{e_i^\ast M_1e_j}(\sqrt{e_i^\ast M_1e_j}e_j^\ast x)\right]^2+e_i^\ast M_1e_i\sum_{\substack{j=1\\j\neq i}}^{N-1}e_i^\ast M_1e_j(e_j^\ast x)^2\\
		&\le \sum_{\substack{j=1\\j\neq i}}^{N-1}e_i^\ast M_1e_j\sum_{\substack{k=1\\k\neq i}}^{N-1}e_i^\ast M_1e_j(e_j^\ast x)^2+e_i^\ast M_1e_i\sum_{\substack{j=1\\j\neq i}}^{N-1}e_i^\ast M_1e_j(e_j^\ast x)^2\\
		&=\sum_{j=1}^{N-1}e_i^\ast M_1e_j\sum_{\substack{k=1\\k\neq i}}^{N-1}e_i^\ast M_1e_j(e_j^\ast x)^2,
		\end{aligned}
		\label{eq:discriminant}
	\end{equation}
	where we have applied the Cauchy-Schwarz inequality to obtain the first inequality. The assumptions on $M_1$ leads to $\sum_{j=1}^{N-1}e_i^\ast M_1e_j=-e_i^\ast M_1e_N < 0$, and the discriminant \eqref{eq:discriminant} is negative for any $x$. Hence \eqref{eq:quadraticeq} is always positive which amounts to the positive definiteness of $\psi_{M_1,e_i}^{(N-1)}$.
\end{proof}

\paragraph{Step 3.}
Suppose that $e_i^\ast Q$ contains no elements that equal $0$. Then, $Q$ satisfies the assumptions of Lemma \ref{lem:psiN-1positive}, and the positive definiteness of $\psi_{Q,e_i}^{(N-1)}$ is obtained. $\psi_{Q,e_i}\mathbbm{1}=0$ is clear, and $M_0\mathbbm{1}=0$ follows from assumption \ref{cond:q}. Applying $M_1=\psi_{Q,e_i}\mathbbm{1}$ and $M_2=M_0$ into Lemma \ref{lem:possemlemma}, the positive semi-definiteness of $M_{i,C}$ is obtained for sufficiently large $C>0$.

\paragraph{Step 4.}
In the case of $e_i^\ast Q$ possibly containing element $0$, the following arguments are required for obtaining the desired result. To this end, We additionally introduce some notations: Denote $e_{i,N}$ as the $i$-th unit vector in $\mathbb R^N$ whose $i$-th element is 1. Note that we sometimes omit subscripts $N$ and simply write $e_i$ when they can be unambiguously determined from the context. For $N$-dimensional vector $v$, denote $\mathcal I(v) = \{n_1,\dots, n_K\}\subset\{1,\dots, N\}$ as the collection of indices of the elements that are non-zero. Similarly, for a $N\times N$ real symmetric matrix $M_1$, denote $\mathcal I(M_1)=\{n_1,\dots, n_K\}\subset\{1,\dots,N\}$ as the collection of indices of non-zero row/column vectors in $M_1$. In both cases, $n_k$ is sorted in ascending order, and $K$ means the total number. For a $N\times N$ matrix $M_1$ and a collection of indices $\mathcal J=\{n_1,\dots, n_K\}\subset\{1,\dots,N\}$, denote $M_1^{\mathcal J}$ as a $K\times K$ matrix obtained by removing the rows and column vectors that do not belong to $\mathcal J$. Equivalently, it can be defined by 
\begin{equation}
	\label{eq:M1J}
	M_1^{\mathcal J} = I_{\mathcal J,N}M_1I_{\mathcal J,N}^\ast,
\end{equation}
where $I_{\mathcal J,N}$ is a $N\times K$ matrix as 
\[
	I_{\mathcal J,N} = \begin{pmatrix}e_{n_1,N}&\ldots & e_{n_K,N}\end{pmatrix}.
\]
Note that $I_{\mathcal J,N}I_{\mathcal J,N}^\ast = I_K$. 
First, we confirm the follwing two lemmas.
\begin{lemma}
	\label{lem:reduced}
	Let $\mathcal J=\{n_1,\dots,n_K\}\subset\{1,\dots,N\}$ be a collection of indices. For any $N\times N$ matrix $M_1$, 
	\[
		e_{k,K}^\ast M_1^{\mathcal J}e_{l,K} = e_{n_k,N}^\ast M_1e_{n_l,N}\quad\text{for}\quad k,l=1,\dots,K.
	\]
\end{lemma}
\begin{proof}
	It immediately follows from \eqref{eq:M1J}:
	\begin{align*}
		e_{k,K}^\ast M_1^{\mathcal J}e_{l,K} &= e_{k,K}^\ast I_{\mathcal J,N}^\ast M_1I_{\mathcal J,N}e_{l,K}\\
		&=e_{k,K}^\ast\begin{pmatrix}e_{n_1,N}^\ast\\\vdots\\e_{n_K,N}^\ast\end{pmatrix}M_1\begin{pmatrix}e_{n_1,N}&\ldots & e_{n_K,N}\end{pmatrix}e_{l,K}
		=e_{n_k,N}^\ast M_1e_{n_l,N},
	\end{align*}
	for $k=1,\dots,K$.
\end{proof}
\begin{lemma}
	\label{lem:lemmazerosum}
	Let $\mathcal J=\{n_1,\dots,n_K\}\subset\{1,\dots,N\}$ be a collection of indices and $M_1$ a $N\times N$ matrix satisfying $M_1\mathbbm{1}=0$. If $\mathcal I(M_1)\subset\mathcal J$, then $M_1^{\mathcal J}\mathbbm{1}=0$ holds.
\end{lemma}
\begin{proof}
	Note that $M_1e_{i,N}=0$ for $i\not\in\mathcal J$ since $\mathcal I(M_1)\subset \mathcal J$. Hence,
	\begin{align*}
		e_{k,K}^\ast M_1^{\mathcal J}\mathbbm{1} &= \sum_{l=1}^Ke_{k,K}^\ast M_1^{\mathcal J}e_{l,K} = \sum_{i\in\mathcal J}e_{n_k,N}^\ast M_1e_{i,N} + 0\\
		&= \sum_{i\in\mathcal J}e_{n_k,N}^\ast M_1e_{i,N}+\sum_{i\not\in\mathcal J}e_{n_k,N}^\ast M_1e_{i,N}= \sum_{i=1}^Ne_{n_k,N}^\ast M_1e_{i,N} = e_{n_k,N}^\ast M_1\mathbbm{1} = 0,
	\end{align*}
	for $k=1,\dots,K$.
\end{proof}

Now, we proceed to show the positive semi-definiteness of $M_{i,C}$. If $e_i^\ast Q=0$, it is trivial since $M_{i,C}$ equals the zero matrix. Hereafter, we suppose that $e_i^\ast Q\neq 0$. Clearly, $M_{i,C}$ is positive semi-definite if $M_{i,C}^{\mathcal I(M_{i,C})}$ is positive semi-definite. We confirm the relationship between $\mathcal I(M_{i,C})$, $\mathcal I(e_{i}^\ast Q)$, $\mathcal I(\psi_{Q,e_{i}})$ as well as $\mathcal I(M_0)$.
\begin{lemma}
	\label{lem:nonzeroind}
	Let $N\in\mathbb N$ be fixed. ($e_{i,N}$ is abbreviated to $e_i$ in this lemma and its proof.)
	\begin{enumerate}
		\item For all $C>0$, $i,j=1,\dots,N$, 
		\begin{gather}
			\psi_{Q,e_i}e_j=0\iff e_{i}^\ast Qe_{j}=0\Longrightarrow M_0e_{j}=0,\label{eq:relation1}\\
			e_{i}^\ast Qe_{j}=0 \Longrightarrow M_{i,C}e_{j}=0.\label{eq:relation2}
		\end{gather}
		\item For sufficiently large $C>0$,
		\begin{equation}
			M_{i,C}e_{j}=0\Longrightarrow \psi_{Q,e_{i}}e_{j}=0,\label{eq:relation3}
		\end{equation}
		for $i,j=1,\dots,N$.
	\end{enumerate}
	Therefore $\mathcal I(M_{i,C})=\mathcal I(e_{i}^\ast Q) = \mathcal I(\psi_{Q,e_{i}}) \supset \mathcal I(M_0)$ holds for sufficiently large $C>0$.
\end{lemma}
\begin{proof}
	\begin{enumerate}
		\item	For $i\neq j$, $\psi_{Q,e_{i}}e_{j} = (e_{i}^\ast Qe_{j})(e_{j}-e_{i})$ implies that $e_{i}^\ast Qe_{j}=0$ if and only if $\psi_{Q,e_{i}}e_j=0$. For $i=j$, $0=\psi_{Q,e_{i}}e_{i} = -Q^\ast e_{i}\iff e_{i}^\ast Qe_{i}=0$. Hence $\psi_{Q,e_i}e_j=0\iff e_{i}^\ast Qe_{j}=0$ is obtained for any $i,j$.

		Next, assume that $e_i^\ast Qe_j=0$. \eqref{cond:q} leads to $e_{i}^\ast\widetilde D_1^{(p)}e_{j}=0$ for all $p=1,\dots,d$, which yields $M_0e_j=0$. 

		Finally, \eqref{eq:relation2} can be seen using \eqref{eq:relation1}. 
		\item Take
		\[
			C > \sup\left\{\frac{e_k^\ast M_0e_{j}}{e_k^\ast \psi_{Q,e_{i}}e_{j}}:k=1,\dots,N\quad\text{and}\quad e_k^\ast \psi_{Q,e_{i}}e_{j}\neq0\right\}\vee0,
		\]
		where $\sup\emptyset=-\infty$. Assuming that $e_k^\ast\psi_{e_i}e_j\neq0$ for some $k\in\{1,\dots,N\}$,
		\[
			e_k^\ast M_{i,C}e_j = Ce_k^\ast\psi_{Q,e_i}e_j - e_k^\ast M_0e_j > 0.
		\]
		It contradicts that $e_k^\ast M_{i,C}e_j=0$.
	\end{enumerate}
\end{proof}
Denote $\mathcal I=\{n_1,\dots,n_K\}\coloneqq \mathcal I(M_{i,C}) = \mathcal I(\psi_{Q,e_{i,N}}) = \mathcal I(Q^\ast e_{i,N})$, and denote $K$ as its total number. Observe that
\begin{equation}
	\label{eq:MiCrep1}
	M_{i,C}^{\mathcal I} = C\psi_{Q,e_{i,N}}^{\mathcal I} - M_0^{\mathcal I}.
\end{equation}
As we suppose $e_i^\ast Q\neq0$, $e_i^\ast Qe_i<0$ and hence $i$ belongs to $\mathcal I$. Take $i\in\{1,\dots,K\}$ that satisfies $i=n_q$. For any $k,l=1,\dots,K$,
\begin{align*}
	e_{l,K}^\ast\psi_{Q,e_p}^{\mathcal I}e_{k,K} &= e_{n_l,N}^\ast\psi_{Q,e_p}e_{n_k,N} = e_{n_l,N}^\ast[e_{p,N}^\ast Qe_{n_k,N}(e_{n_k,N}-e_{p,N})-\delta_{p,n_k}Q^\ast e_{n_k,N}]\\
	&= e_{q,K}^\ast Q^{\mathcal I}e_{k,K}e_{l,K}^\ast(e_{k,K}-e_{q,K})-\delta_{e,k}e_{l,K}^\ast(Q^{\mathcal I})^\ast e_{k,K}\\
	&= e_{l,K}^\ast[e_{q,K}^\ast Q^{\mathcal I}e_{k,K}(e_{k,K}-e_{q,K})-\delta_{e,k}(Q^{\mathcal I})^\ast e_{k,K}] = e_{l,K}^\ast\psi_{Q^{\mathcal I},e_{q,K}}e_{k,K},
\end{align*}
where we used Lemma \ref{lem:reduced}. As a result, we obtain $\psi_{Q,e_{i,N}}^{\mathcal I} = \psi_{Q^{\mathcal I},e_{q,K}}$. Notice what follows:
\begin{itemize}
	\item $\psi_{Q^{\mathcal I},e_{q,K}}$ and $M_0^{\mathcal I}$ are symmetric.
	\item Since $\mathcal I(\psi_{Q,e_{i,N}}), \mathcal I(M_0)\subset\mathcal I$, Lemma \ref{lem:lemmazerosum} implies $\psi_{Q,e_{i,N}}^{\mathcal I}\mathbbm{1}=0$ (i.e. $\psi_{Q^{\mathcal I},e_{q,K}}\mathbbm{1}$) and $M_0^{\mathcal I}\mathbbm{1}=0$.
	\item Since $Q^{\mathcal I}=Q^{\mathcal I(e_i^\ast Q)}$, $e_q^\ast Q^{\mathcal I}e_l >0$ for $k\in\{1,\dots,K\}\setminus\{q\}$ and $e_q^\ast Q^{\mathcal I}\mathbbm{1}=0$. Applying $M_1=Q^{\mathcal I}$ to Lemma \ref{lem:psiN-1positive}, we obtain the positive definiteness of $\psi_{Q^{\mathcal I},e_{q,K}}^{(K-1)}$.
\end{itemize}
Thus we can apply $M_1=\psi_{Q^{\mathcal I},e_{q,K}}$ and $M_2=M_0^{\mathcal I}$ to Lemma \ref{lem:possemlemma}, and we obtain the positive semi-definiteness of $M_{i,C}^{\mathcal I}$. 
\section{Convergence results}
\label{sec:convergence}
In this section, we establish a convergence result for the numerical solution discussed in Section 4 to the true solution of BSDE \eqref{BMBSDE}. For simplicity, we only consider a situation with (i) one-dimensional space variable, (ii) the corresponding PDE being uniquely solvable in the classical sense, as well as (iii) a spatial discretization using central difference with constant step size i.e. $\delta x_i\equiv \Delta x>0$.

Throughout the section, the following notations are introduced;
\begin{gather}
	\begin{multlined}
		\widecheck f:[0,T]\times\mathbb R\times\mathbb R\times\mathbb R\times\mathbb R\ni(t,x,z,p,r)\\
		\mapsto\mu(x)\cdot p+\frac{\sigma^2(x)}2\cdot r+f(t,x,z,\sigma(x)\cdot p)\in\mathbb R.
	\end{multlined}\\
	U_t^{(N,k)}\text{ is a unique solutions of \eqref{TVP} for }N\text{ and }k.
\end{gather}
Denote \eqref{PDE} as 
\begin{equation}
\label{eq:B1}
\begin{cases}
  \partial_tu(t,x)+\widecheck f(t,x,u(t,x),\partial_xu(t,x),\partial_{xx}u(t,x))=0,\quad (t,x)\in[0,T]\times\mathbb R,\\
  u(T,x)=g(x),\quad x\in\mathbb R.
\end{cases}
\end{equation}
The present analysis in this section is mainly based on the textbook written by Walter et al.\cite{Walter2012}. Suppose that the following conditions are satisfied.
\begin{assumption}[\cite{Walter2012}, pp.287 and 302]
	\label{basicassumption}	
	\begin{enumerate}
		\item $\mu(x),\sigma(x)$ and $f(t,x,y,z)$ are twice continuously differentiable in all variables. $f$ and its first- and second-order derivatives are bounded and uniformly continuous in $(t,x,y,z)\in[0,T]\times\mathbb R\times B$ for any bounded set $B\subset\mathbb R^2$.
		\item For some constant $C>0$ and a continuous function $\lambda:[0,\infty)\to(0,\infty)$ satisfying
		\begin{equation}
			\label{ubound}
			\lim_{s\to\infty}\lambda(s)=\infty\quad\text{and}\quad\int_0^\infty\frac{ds}{s\cdot \lambda(s)}=\infty,
		\end{equation}
		it holds 
		\[
			f(t,x,y,0)\vee(-f(t,x,-y,0))\le C+y\cdot\lambda(y),\quad\text{for any}\quad (t,x,y)\in[0,T]\times\mathbb R\times[0,\infty).
		\]
		%for any $t\in[0,T],x\in\mathbb R$ and $y\ge0$.
		\item $\sigma^2(x)>0$ for any $x\in\mathbb R$.
		\item For any $M>0$, there exists $C_0>0$ and $\lambda_0:[0,\infty)\to(0,\infty)$ satisfying the following conditions:
		\begin{itemize}
			\item $\lambda_0$ is continuous and satisfies \eqref{ubound}.
			\item For $(t,x)\in[0,T]\times\mathbb R$, $|y|\le M$ and $z\ge0$,
			\[
				|z|\partial_yf(t,x,y,\sigma(x)z)\le |z|\lambda_0(z)+C_0.
			\]
			% \begin{align*}
			% 	(-\partial_x&\check f(t,x,y,z,0)-z\cdot\partial_z\check f(t,x,y,z,0))\\
			% 	&\vee(\partial_x\check f(t,x,y,-z,0)-z\cdot\partial_z\check f(t,x,y,-z,0))\le z\cdot\lambda_0(z) + C_0.
			% \end{align*}
			% \item For $(t,x)\in[0,T]\times\mathbb R$ and $|y|,|z|\le M$, 
			% \begin{align*}
			% 	|\partial_x&\check f(t,x,y,z,r)+y\partial_y\check f(t,x,y,z,r)|\\
			% 	&\vee|\partial_t\check f(t,x,y,z,r)+\check f(t,x,y,z,r)\partial_y\check f(t,x,y,z,r)|\le C_0(1+r^2).
			% \end{align*}
		\end{itemize}
		\item $g\in C_b^2(\mathbb R)$.
	\end{enumerate}
\end{assumption}
We obtain the following lemmas. For details, see \cite{Walter2012}, Chapter IV, Section 36. 
\begin{lemma}[\cite{Walter2012}, p.292]
	\label{wellposedness}
	Cauchy problem \eqref{eq:B1} with $\mu(x),\sigma(x),f(t,x,y,z)$ and $g(x)$ satisfying  Assumption \ref{basicassumption} is uniquely solvable if $g$ is three times continuously differentiable and $g$, $\partial_xg$, $\partial_{xx}g$ and $\partial_{xxx}g$ are bounded and Lipschitz continuous.
\end{lemma}
\begin{lemma}[\cite{Walter2012}, p.301]
	\label{convergence1}
	If \eqref{eq:B1} satisfying Assumption \ref{basicassumption} (not necessarily the case in Lemma \ref{wellposedness}) admits a unique solution $u$, then for every $\Delta x>0$, the infinite system of ODEs,
	\begin{equation}
		\label{infiniteODE}
		\begin{cases}
		\displaystyle\frac{de_i^\ast U_t^{(\infty,\Delta x)}}{dt} = \widecheck f(t,i\Delta x,e_i^\ast U_t^{(\infty,\Delta x)},e_i^\ast D_1U_t^{(\infty,\Delta x)},e_i^\ast D_2U_t^{(\infty,\Delta x)}),\\
		e_i^\ast U_T^{(\infty,\Delta x)} = g(i\Delta x),
		\end{cases} (t,i)\in[0,T]\times\mathbb Z,
	\end{equation}
	which is derived from the spatial discretization described in Section 4 with step size $\Delta x$, admits a unique solution $U^{(\infty,\Delta x)}:[0,T]\ni t\mapsto U_t^{(\infty,\Delta x)} = (e_i^\ast U_t^{(\infty,\Delta x)})_{i\in\mathbb Z}\in l^\infty$. Here, $l^\infty$ is the Banach space consisting of all real sequences $x=(x_i)_{i\in\mathbb Z}$ with finite supremum norm $\|x\|_\infty = \sup_{i\in\mathbb Z}|x_i|$, and 
	\begin{align*}
		e_i^\ast D_1U_t^{(\infty,\Delta x)} &\coloneqq \frac{-1}{2\Delta x}e_{i-1}^\ast U_t^{(\infty,\Delta x)} + \frac{1}{2\Delta x}e_{i+1}^\ast U_t^{(\infty,\Delta x)},\\
		e_i^\ast D_2U_t^{(\infty,\Delta x)} &\coloneqq \frac{1}{\Delta x^2}e_{i-1}^\ast U_t^{(\infty,\Delta x)}+\frac{-2}{\Delta x^2}e_i^\ast U_t^{(\infty,\Delta x)} + \frac{1}{\Delta x^2}e_{i+1}^\ast U_t^{(\infty,\Delta x)},
	\end{align*}
	for $i\in\mathbb Z$. Furthermore, $U^{(\infty,\Delta x)}$ converges to $u$ in the following sense; for any compact set $K\subset\mathbb R$,
	\[
		\lim_{k\to0}\sup_{\substack{t\in[0,T],i\Delta x\in\mathbb Z\\i\Delta x\in K}}|u(t,i\Delta x)-e_i^\ast U_t^{(\infty,\Delta x)}|=0.
	\]
\end{lemma}
Note that Walter et al. \cite{Walter2012} do not consider results on the convergence of ``truncated'' finite systems of ODEs \eqref{TVP} to \eqref{infiniteODE}. Fortunately, it can be carried out by the standard diagonalization argument using the Arzel\`a-Ascoli theorem. To this end, suppose the following Lipschitz condition additionally.
\begin{assumption}
	\label{globallipshcitz}
	There exists $L>0$ such that 
	\[
		|f(t,x,y,z)-f(t,x,y',z')| \le L(|y-y'|+|z-z'|)
	\]
	for any $(t,x)\in[0,T]\times\mathbb R,y,y',z,z'\in\mathbb R$.
\end{assumption}
It leads to the Lipschitz continuity of $F$ defined by \eqref{eq:definitionF}, and \eqref{TVP} admits a unique solution $U^{(N,\Delta x)}$. Then, we define $\overline U^{(N,\Delta x)} = (e_i^\ast \overline U^{(N,\Delta x)})_{i\in\mathbb Z}\in l^\infty$ for $N\in\mathbb N$ as 
\[
	\displaystyle e_i^\ast\overline U^{(N,\Delta x)} = 
	\begin{cases}
		e_N^\ast U^{(N,\Delta x)},&\text{if}\quad i=N+1,N+2,\dots,\\
		e_i^\ast U^{(N,\Delta x)},&\text{if}\quad i=-N,-N+1,\dots,N-1,N,\\
		e_{-N}^\ast U^{(N,\Delta x)},&\text{if}\quad i=-N-1,-N-2,\dots.
	\end{cases}
\]
\begin{lemma}
	\label{convergence2}
	Suppose that Assumption \ref{basicassumption} and \ref{globallipshcitz} hold, and that Cauchy problem \eqref{eq:B1} admits a unique solution $u$. Then, for any $\Delta x>0$,
	\[
		\lim_{N\to\infty}\sup_{t\in[0,T],-N\le i\le N}|e_i^\ast U_t^{(\infty,\Delta x)}-e_i^\ast\overline U_t^{(N,\Delta x)}|_\infty = 0.
	\]
	Here, $U_t^{(\infty,\Delta x)}\in l^\infty$ is a unique solution of \eqref{infiniteODE}.
\end{lemma}
\begin{proof}
	Let $k>0$ be fixed. 
	\paragraph{Uniform boundedness of $(\overline U_t^{(N,\Delta x)})_{N=1}^\infty$.}
	Note that
	\begin{equation}
		\label{eq:B4}
		\sup_{N\in\mathbb N}|e_i^\ast U_T^{(N,\Delta x)}|\le\|g(x)\|_\infty,\quad \sup_{(t,x)\in[0,T]\times\mathbb R}|f(t,x,0,0)|<C,
	\end{equation}
	which follow from 5 and 2 in Assumption \ref{basicassumption}, respectively. For $N\in\mathbb N$ and $i=\pm N$, 2 in Assumption \ref{basicassumption} and the Lipschitz continuity \ref{globallipshcitz} lead to
	\[
		|f(t,i\Delta x,e_i^\ast U_t^{(N,\Delta x)},\sigma(i\Delta x)e_i^\ast D_1U_t^{(N\Delta x)}))|=|f(t,i\Delta x,e_i^\ast U_t^{(N,\Delta x)},0)| \le C + L|e_i^\ast U_t^{(N,\Delta x)}|.
	\]
	These estimates imply
	\begin{equation}
		\label{eq:B5}
		|e_i^\ast U_t^{(N,\Delta x)}| \le C+\|g(x)\|_\infty +L\int_t^T|e_i^\ast U_s^{(N,\Delta x)}|ds\quad\text{for}\quad i=\pm N.
	\end{equation}
	Applying the Gr\"onwall inequality to \eqref{eq:B5}, we obtain 
	\begin{equation}
		|e_{-N}^\ast U_t^{(N,\Delta x)}|\vee|e_N^\ast U_t^{(N,\Delta x)}|\le e^{LT}\cdot[C+\|g(x)\|_\infty]\eqqcolon c_1,
	\end{equation}
	for any $N\in\mathbb N$ and $t\in[0,T]$. For $N>i>-N$, notice that using \eqref{eq:B4} and the Lipschitz continuity, 
	\begin{align*}
		|f(t,i\Delta x,e_i^\ast U_t^{(N,\Delta x)},&\sigma(i\Delta x)e_i^\ast D_1U_t^{(N,\Delta x)})|\\
		&\le C + L\left[|e_i^\ast U_t^{(N,\Delta x)}|+\|\sigma\|_\infty\frac{|e_{i+1}^\ast U_t^{(N,\Delta x)}|+|e_{i-1}^\ast U_t^{(N,\Delta x)}|}{2\Delta x}\right]\\
		&\le (C\vee L\vee L\frac{\|\sigma\|_\infty}{2\Delta x})[1+|e_i^\ast U_t^{(N,\Delta x)}|+|e_{i+1}^\ast U_t^{(N,\Delta x)}|+|e_{i-1}^\ast U_t^{(N,\Delta x)}|],
	\end{align*}
	and
	\begin{multline*}
			\mu(ki)\frac{|e_{i+1}^\ast U_t^{(N,\Delta x)}|+|e_{i-1}^\ast U_t^{(N,\Delta x)}|}{2k}+\frac{\sigma^2(i\Delta x)}2\frac{2|e_{i+1}^\ast U_t^{(N,\Delta x)}|+|e_i^\ast U_t^{(N,\Delta x)}|+2|e_{i-1}^\ast U_t^{(N,\Delta x)}|}{2k^2}\\
			\le 2\left(\frac{\|\mu\|_\infty}{2\Delta x}\vee\frac{\|\sigma\|^2}{4\Delta x^2}\right)\sum_{j=i-1}^{i+1}|e_j^\ast U_t^{(N,\Delta x)}|.
	\end{multline*}
	Let $C'\coloneqq C\vee L\vee L\frac{\|\sigma\|_\infty}{2\Delta x}$, $C''=2\left(\frac{\|\mu\|_\infty}{2k}\vee\frac{\|\sigma\|^2}{4\Delta x^2}\right)$, and $c_2\coloneqq C'+C''$. Since  
	\begin{align*}
		|e_i^\ast U_t^{(N,\Delta x)}| &\le \|g(x)\|_\infty + C'T+c_2\int_t^T\sum_{m=i-1}^{i+1}|e_m^\ast U_s^{(N,\Delta x)}|ds\\
		&\le \|g(x)\|_\infty + C'T+c_2\int_t^T\sup_{|j|<N}\sum_{m=j-1}^{j+1}|e_m^\ast U_s^{(N,\Delta x)}|ds\\
		&\le \|g(x)\|_\infty + C'T+2c_2c_1T + 3c_2\int_t^T\sup_{|j|<N}|e_j^\ast U_s^{(N,\Delta x)}|ds
	\end{align*}
	for $N>i>-N$, we obtain 
	\[
		\sup_{t\in[0,T],|i|<N}|e_i^\ast U_t^{(N,\Delta x)}| \le e^{3c_2T}[\|g(x)\|_\infty + C'T+2c_2c_1T]\eqqcolon c_3.
	\]
	Therefore, $\sup_{t\in[0,T]}\displaystyle\|\overline U_t^{(N,\Delta x)}\|_\infty = \sup_{|i|\le N}|e_i^\ast U_t^{(N,\Delta x)}|\le c_1\vee c_3\eqqcolon c_4$.
	
	\paragraph{Uniform boundedness of $(\frac{dU_t^{(N,\Delta x)}}{dt})_{N=1}^\infty$}
	It follows from
	\[
		\left|\frac{de_i^\ast U_t^{(N,\Delta x)}}{dt}\right| \le C+ L|e_i^\ast U_t^{(N,\Delta x)}| \le C+ Lc_4
	\]
	for $i=\pm N$, and
	\[
		\left|\frac{de_i^\ast U_t^{(N,\Delta x)}}{dt}\right| \le C' + 2c_2c_1+3c_2\sup_{|i|<N}|e_i^\ast U_s^{(N,\Delta x)}|\le C' + 2c_2c_1+3c_2c_4
	\]
	for $N>i>-N$. $(\frac{dU_t^{(N,\Delta x)}}{dt})_{N=1}^\infty$ has an uniform upper bound $c_5\coloneqq (C+ Lc_4)\vee(C'+2c_2c_1+3c_2c_4)$.

	\paragraph{Equicontinuity of $(U_t^{(N,\Delta x)})_{N=1}^\infty$}
	It follows from the mean value theorem and the uniform boundedness of $(\frac{dU_t^{(N,\Delta x)}}{dt})_{N=1}^\infty$. Precisely, for any $t,t'\in[0,T]$, $-N\le i\le N$, and $N\in\mathbb N$,
	\[
		|U_t^{(N,\Delta x)}e_i-U_{t'}^{(N,\Delta x)}e_i|\le\sup_{s\in[0,T]}\left|\frac{dU_s^{(N,\Delta x)}e_i}{ds}\right||t-t'|<c_5|t-t'|,
	\]
	which immediately leads to the equicontinuity.
	% Hence, for $\epsilon>0$, taking $\displaystyle\delta = \epsilon/c_6$, we obtain 
	% \[
	% 	\|\overline U_t^{(N,\Delta x)}-\overline U_{t'}^{(N,\Delta x)}\|_\infty < c_6|t-t'| < \epsilon
	% \]
	% for any $N\in\mathbb N$ and $|t-t'|<\delta$, which is equivalent to the uniform equicontinnuity.

	\paragraph{Equicontinuity of $(\frac{dU_t^{(N,\Delta x)}}{dt})_{N=1}^\infty$.}
	Let $s,t\in[0,T]$ be fixed. Since $(\frac{dU_t^{(N,\Delta x)}}{dt})_{N=1}^\infty$ is uniformly bounded, using the mean value theorem,
	\begin{equation}
		\label{4-1}
		\begin{multlined}
		|f(s,i\Delta x,U_s^{(N,\Delta x)}e_i,\sigma(i\Delta x)e_i^\ast D_1U_s^{(N,\Delta x)})-f(s,i\Delta x,e_i^\ast U_t^{(N,\Delta x)},\sigma(i\Delta x)e_i^\ast D_1U_t^{(N,\Delta x)})|\\
		\le L\left(1\vee \frac{\|\sigma\|_\infty}\Delta x\right)\sum_{j=i-1}^{i+1}|e_j^\ast U_s^{(N,\Delta x)}-e_j^\ast U_t^{(N,\Delta x)}|\\
		=\underbrace{3Lc_5\left(1\vee \frac{\|\sigma\|_\infty}\Delta x\right)}_{\eqqcolon c_6}\cdot|s-t|= c_6|s-t|
		\end{multlined}
	\end{equation}
	and
	\begin{equation}
		\label{4-2}
		\begin{split}
		&|f(s,i\Delta x,e_i^\ast U_t^{(N,\Delta x)},\sigma(i\Delta x)e_i^\ast D_1U_t^{(N,\Delta x)})-f(t,i\Delta x,e_i^\ast U_t^{(N,\Delta x)},\sigma(i\Delta x)e_i^\ast D_1U_t^{(N,\Delta x)})|\\
		&\le\sup_{t'\in[0,T]}|\partial_tf(t',i\Delta x,e_i^\ast U_t^{(N,\Delta x)},\sigma(i\Delta x)e_i^\ast D_1U_t^{(N,\Delta x)})|\cdot|s-t|\\
		&\le \sup_{(t',x,y,z)\in[0,T]\times\mathbb R^3}|\partial_tf(t',x,y,z)|\cdot|s-t| < c_7|s-t|.
		\end{split}
	\end{equation}
	Using \eqref{4-1}, \eqref{4-2} and the triangle inequality,
	\begin{align*}
		&\left|\frac{de_i^\ast U_s^{(N,\Delta x)}}{dt}-\frac{de_i^\ast U_t^{(N,\Delta x)}}{dt}\right|\\
		&\begin{multlined}
		\le|f(s,i\Delta x,e_i^\ast U_s^{(N,\Delta x)},\sigma(i\Delta x)e_i^\ast D_1U_s^{(N,\Delta x)})-f(t,i\Delta x,e_i^\ast U_t^{(N,\Delta x)},\sigma(i\Delta x)e_i^\ast D_1U_t^{(N,\Delta x)})|\\
		+ C''\left[|e_{i+1}^\ast U_s^{(N,\Delta x)}|+|e_i^\ast U_s^{(N,\Delta x)}|+|e_{i-1}^\ast U_s^{(N,\Delta x)}|\right]
		\end{multlined}\\
		&\le [C''+c_6+c_7]\cdot|s-t|,
	\end{align*}
	from which $(\frac{dU_t^{(N,\Delta x)}}{dt})_{N=1}^\infty$ enjoys the equicontinuity.

	\paragraph{Convergence.}
	Let $(N_n)_{n=1}^\infty$ be an arbitrary subsequence of $\mathbb N$. Applying the Arzel\`a-Ascoli theorem (\cite{Munkres2000}, P.290) guarantees the existence of convergent subsequences $U_t^{(N_{n(j)},\Delta x)}$ and $\displaystyle\frac{dU_t^{(N_{n(j)},\Delta x)}}{dt}$. Let $U_t^{(N_{n(\infty)},\Delta x)}$ and $\displaystyle\frac{dU_t^{(N_{n(\infty)},\Delta x)}}{dt}$ be the limit functions of $t$ in the sense as
	\[
		\begin{aligned}	
			\lim_{j\to\infty}\sup_{t\in[0,T]}\|U_t^{(N_{n(\infty)},\Delta x)}-U_t^{(N_{n(j)},\Delta x)}\|_\infty&=0,\\
			\lim_{j\to\infty}\sup_{t\in[0,T]}\left\|\frac{dU_t^{(N_{n(\infty)},\Delta x)}}{dt}-\frac{dU_t^{(N_{n(j)},\Delta x)}}{dt}\right\|_\infty&=0.
		\end{aligned}
	\]
	Using the triangle inequality, 
	\begin{align*}
		&\left|\frac{de_i^\ast U_t^{(N_{n(\infty)},\Delta x)}}{dt}-\frac{de_i^\ast U_t^{(N_{n(j)},\Delta x)}}{dt}\right|\\
		&=\left|\frac{de_i^\ast U_t^{(N_{n(\infty)},\Delta x)}}{dt}-\widecheck f(t,i\Delta x,e_i^\ast U_t^{(N_{n(j)})},e_i^\ast D_1U_t^{(N_{n(j)})},e_i^\ast D_2U_t^{(N_{n(j)})})\right|\\
		&\begin{multlined}
			\ge-c_8\sup_{t\in[0,T]}\|U_t^{(N_{n(\infty)})}-U_t^{(N_{n(j)})}\|_\infty\\
			+\left|\frac{de_i^\ast U_t^{(N_{n(\infty)},\Delta x)}}{dt}-\widecheck f(t,i\Delta x,e_i^\ast U_t^{(N_{n(\infty)})},e_i^\ast D_1U_t^{(N_{n(\infty)})},e_i^\ast D_2U_t^{(N_{n(\infty)})})\right|.
		\end{multlined}
	\end{align*}
	Considering $j\to\infty$ yields
	\[
		\frac{de_i^\ast U_t^{(N_{n(\infty)},\Delta x)}}{dt}=\widecheck f(t,i\Delta x,e_i^\ast U_t^{(N_{n(\infty)})},\sigma(i\Delta x)e_i^\ast D_1U_t^{(N_{n(\infty)})},e_i^\ast D_2U_t^{(N_{n(\infty)},\Delta x)})
	\]
	for $t\in[0,T]$ and $i\in\mathbb Z$, which implies that $U_t^{(N_{n(\infty)},\Delta x)}$ solves \eqref{infiniteODE}. Using a uniqueness result in Lemma \ref{wellposedness}, we obtain $U^{(\infty,\Delta x)}=U^{(N_{n(\infty)},\Delta x)}$. As $(N_j)_{j=1}^\infty$ is arbitrary, we obtain the desired result.
\end{proof}

\begin{theorem}
	Suppose that Assumption \ref{basicassumption}, \ref{globallipshcitz} and \eqref{eq:B1} admits a unique solution $u$. For any compact set $K\subset\mathbb R$, it holds
	\[
		\lim_{\Delta x\to0}\lim_{N\to\infty}\sup_{\substack{-N_0\le i\le N_0\\i\in\mathbb Z,i\Delta x\in K}}|u(t,i\Delta x)-e_i^\ast U_t^{(N,\Delta x)}| = 0.
	\]
\end{theorem}
\begin{proof}
	The argument follows from Lemma \ref{convergence1}, \ref{convergence2} and the triangle inequality; 
	\begin{align*}
		\sup_{\substack{i\in\mathbb Z\\i\Delta x\in K}}|u(t,i\Delta x)-e_i^\ast U_t^{(N,\Delta x)}| &\le \sup_{\substack{i\in\mathbb Z\\i\Delta x\in K}}|u(t,i\Delta x)-e_i^\ast U_t^{(\infty,\Delta x)}|+\|U_t^{(\infty,\Delta x)}-U_t^{(N,\Delta x)}\|_\infty\\
		&\to \sup_{\substack{i\in\mathbb Z\\i\Delta x\in K}}|u(t,i\Delta x)-e_i^\ast U_t^{(\infty,\Delta x)}|, \quad N\to\infty.\\
		&\to 0,\quad \Delta x\to0.
	\end{align*}
\end{proof}

% \begin{lemma}[Theorem 3.2 in Pardoux and Peng \cite{Pardoux1992}]
% 	Suppose that the following conditions are fulfilled.
% 	\begin{itemize}
% 		\item $\mu,\sigma\in C_b^3(\mathbb R)$.
% 		\item $g\in C^3(\mathbb R)$ whose derivatives and itself grow at most like a polynomial function at infinity.
% 		\item For $t\in[0,T]$, mapping $(x,y,z)\mapsto f(t,x,y,z)$ is of class $C^3(\mathbb R\times\mathbb R\times\mathbb R)$.
% 		\item For $t\in[0,T]$, mapping $x\mapsto f(t,x,0,0)$ and its derivatives grow at most like a polynomial function at infinity.
% 		\item $\frac{\partial f}{\partial y}$ and $\frac{\partial f}{\partial z}$ are bounded, whose derivatives of order one and two with respect to $x, y, z$ are also bounded.
% 	\end{itemize}
% 	Then, $u(t,x)\coloneqq\mathcal Y_t^{t,x}$ is of class $C^{1,2}([0,T]\times\mathbb R;\mathbb R)$ and a classical solution of \eqref{eq:B1}.
% \end{lemma}

\nocite{*}
\bibliographystyle{unsrt}
\bibliography{references}
\end{document}